\documentclass[12pt]{article}
\usepackage{amsmath}
\usepackage{amssymb}
\usepackage{scalefnt}
\usepackage{amsthm}
\usepackage{verbatim}
\usepackage{mathrsfs}
\usepackage[dvips]{graphicx}
\usepackage[nohug]{diagrams}
\newcommand{\U}{{\mathcal U}}
\newcommand{\V}{{\mathcal V}}
\newcommand{\0}{{\mathcal O}}
\newcommand{\D}{{\mathcal D}}
\newcommand{\W}{{\mathcal W}}
\newcommand{\C}{{\mathcal C}}

\newcommand{\Zj}{{\mathcal Z}_j}
\newcommand{\Ps}{{\mathcal P}}
\newcommand{\id}{{\tt id}}
\newcommand{\finishproof}{\hfill $\Box$ \vspace{5mm}}

\newcommand{\infR}[0]{\inf_{x \in \mathbb R^n}}
\newcommand{\R}[0]{\mathbb R}
\newcommand{\Hs}[0]{H}
\newcommand{\Ds}[0]{\mathcal D}
\newcommand{\N}[0]{\mathbb N}
\newcommand{\image}{\mathop{\tt Image}}
\newcommand{\hnorm}[2]{\| #1 \|_{#2}}

\newcommand{\tspace}{C_c^\infty(\R^n,\R)}
\newcommand{\tspacen}{C_0^\infty(\R^n,\R^n)}
\newcommand{\dx}[1]{\partial_{x_{#1}}}
\newcommand{\Z}[0]{\mathbb{Z}_{\geq 0}}
\newcommand{\Tn}[0]{\mathbb{T}^n}

\newtheorem{Th}{Theorem}[section]
\newtheorem{Lemma}{Lemma}[section]
\newtheorem{Prop}[Lemma]{Proposition}
\newtheorem{Coro}[Th]{Corollary}
\newtheorem{Def}{Definition}[section]
\newtheorem{Rem}{Remark}[section]

\begin{document}

\scalefont{1.06}

\title{On the regularity of the composition of diffeomorphisms}

\author{H. Inci\footnote{Supported in part by the
Swiss National Science Foundation} , T. Kappeler\footnotemark[1] , P. Topalov\footnote{Supported in
part by NSF DMS-0901443}}

\maketitle

\begin{abstract}
For $M$ a closed manifold or the Euclidean space $\R^n$ we present a detailed proof of regularity properties of the composition of $H^s$-regular diffeomorphisms of $M$ for $s > \frac{1}{2}\dim M+1$.
\end{abstract}

\section{Introduction}\label{Sec:Introduction}
In this paper we are concerned with groups of diffeomorphisms on a
smooth manifold $M$. Our interest in these groups stems from Arnold's seminal paper \cite{1} on hydrodynamics. He suggested
that the Euler equation modeling a perfect fluid on a (oriented) Riemannian
manifold $M$ can be reformulated
as the equation for geodesics on the group of volume (and orientation)
preserving diffeomorphims of $M$. In this way properties of solutions
of the Euler equation can be expressed in geometric terms -- see \cite{1}. In
the sequel, Ebin and Marsden \cite{EM}, \cite{EMF} used this approach to great success to study the initial value
problem for the Euler equation on a compact manifold, possibly with boundary.
Later it was observed that other nonlinear evolution equations such as Burgers equation \cite{3}, KdV, or
the Camassa Holm equation \cite{4}, \cite{13} can be viewed in a similar way -- see \cite{19bis}, \cite{25}, 
as well as \cite{2}, \cite{15}, and \cite{KW}. In particular, for the study of the solutions of the Camassa Holm
equation, this approach has turned out to be very useful -- see e.g. \cite{10}, \cite{misiolek}. 
In addition, following Arnold's suggestions \cite{1}, numerous papers aim at relating the stability of the flows
to the geometry of the groups of diffeomorphisms considered -- see e.g. \cite{2}.

\medskip
In various settings, the space of diffeomorphisms of a given manifold
with prescribed regularity turns out to be a (infinite dimensional)
topological group with the group operation given by the
composition  -- see e.g. \cite[p 155]{EMF} for a quite detailed historical account.
In order for such a group of diffeomorphisms
to be a Lie group, the composition and the inverse map have to
be $C^\infty$-smooth. A straightforward formal computation shows that the
differential of the left translation $L_\psi : \varphi \mapsto
\psi \circ \varphi $ of a diffeomorphism $\varphi $ by a
diffeomorphism $\psi $ in direction $h : M \rightarrow TM$ can be
formally computed to be
\[
(d_\varphi L_\psi ) (h) (x) = (d_{\varphi(x)} \psi )(h(x)), \ x \in M
\]
and hence involves a loss of derivative of $\psi $.
As a consequence, for a space of diffeomorphisms
of $M$ to be a Lie group it is necessary that they are $C^\infty $-smooth
and hence such a group cannot have the structure of a Banach
manifold, but only of a Fr\'echet manifold. It is well known that
the calculus in  Fr\'echet manifolds is quite involved as the
classical inverse function theorem does not hold, cf. e.g. \cite{14},
\cite{21}. Various aspects of Fr\'echet Lie groups of diffeomorphisms
have been investigated -- see e.g. \cite{14}, \cite{24}, \cite{27}, \cite{27bis}. In
particular, Riemann exponential maps have been studied in \cite{6},
\cite{8}, \cite{16}, \cite{16bis}.

\medskip

However, in many situations, one has to consider diffeomorphisms
of Sobolev type -- see e.g. \cite{10}, \cite{11}, \cite{EM}. In this paper we are concerned with
composition of maps in $H^s(M) \equiv H^s(M,M)$. It seems to be unknown whether, in general, the composition of two maps
in $H^s(M)$ with $s$ an integer satisfying $s > n/2$ is again in $H^s(M)$. In all known proofs
one needs that one of the maps is a diffeomorphism or, alternatively, is
$C^\infty $-smooth.

\medskip

First we consider the case where $M$ is the Euclidean space
$\R^n$, $n\ge 1$. Denote by $\mbox{Diff}_+^1(\R^n)$ the space
of orientation preserving $C^1$-diffeomorphisms of $\R^n$, i.e. the space of bijective $C^1$-maps
$\varphi:\R^n \to \R^n$ so that $\det(d_x\varphi) > 0$ for any $x \in \R^n$ and $\varphi^{-1}:\R^n \to \R^n$
is a $C^1$-map as well. For any integer $s$ with $s > n/2+1$ introduce
\[
\Ds^s(\R^n):=\{\varphi \in \mbox{Diff}_+^1(\R^n) \, | \, \varphi - \id \in \Hs^s(\R^n) \}
\]
where $H^s(\R^n)=H^s(\R^n,\R^n)$ and $H^s(\R^n,\R^d)$ is the Hilbert space
\[
H^s(\R^n,\R^d):=\{f=(f_1,\ldots,f_d) \, | \, f_i \in H^s(\R^n,\R), i=1,\ldots,d\}
\]
with $H^s$-norm $\hnorm{\cdot}{s}$ given by
\[
\hnorm{f}{s}=\big( \sum_{i=1}^d \hnorm{f_i}{s}^2\big)^{1/2} 
\]
and $H^s(\R^n,\R)$ is the Hilbert space of elements $g \in L^2(\R^n,\R)$ with the property that the distributional
derivatives $\partial^\alpha g$, $\alpha \in \mathbb Z^n_{\geq 0}$, up to order $|\alpha| \leq s$ are in 
$L^2(\R^n,\R)$. Its norm is given by
\begin{equation}\label{hsnorm}
 \hnorm{g}{s}=\big( \sum_{|\alpha|\leq s}
\int_{\R^n} |\partial^{\alpha} g|^2 dx \big)^{1/2}.
\end{equation}
Here we used multi-index notation, i.e. $\alpha=(\alpha_1,\ldots,\alpha_n)
\in \Z^n$, $|\alpha|=\sum_{i=1}^n \alpha_i$, $x=(x_1,\ldots,x_n)$, and
$\partial^\alpha \equiv \partial_x^{\alpha}=\partial_{x_1}^{\alpha_1}
\cdots \partial_{x_n}^{\alpha_n}$. As $s>n/2+1$ it follows from the
Sobolev embedding theorem that 
\[
\Ds^s(\R^n)- \id=\{ \varphi -\id \,\,|\,\, \varphi \in \Ds^s(\R^n)\}
\]
is an open subset of $H^s(\R^n)$ -- see Corollary \ref{cor_openness} below. In this way $\Ds^s(\R^n)$ becomes a
Hilbert manifold modeled on $H^s(\R^n)$.
In Section \ref{Section 2} of this paper we present a detailed proof of the following
\begin{Th}\label{thm1}
For any $r \in \Z$ and any integer $s$ with $s > n/2+1$
\begin{equation}
\mu : \Hs^{s+r}(\R^n,\R^d) \times \Ds^s(\R^n) \to \Hs^s(\R^n,\R^d), \quad  (u,\varphi) \mapsto u \circ \varphi
\end{equation}
and
\begin{equation}
{\tt inv} : \Ds^{s+r}(\R^n) \to \Ds^s(\R^n), \quad \varphi \mapsto \varphi^{-1}
\end{equation}
are $C^r$-maps.
\end{Th}

\medskip

\begin{Rem}
To the best of our knowledge there is no proof of Theorem \ref{thm1} available in the literature.
Besides being of interest in itself we will use Theorem \ref{thm1} and its proof to show Theorem \ref{thm2}
stated below. Note that the case $r=0$ was considered in \cite{cantor}.
\end{Rem}

\begin{Rem}
The proof for the $C^r$-regularity of the inverse
map is valid in a much more general context: using that $\Ds^s(\R^n)$
is a topological group and that the composition
\[
\Ds^{s+r}(\R^n) \times \Ds^s(\R^n) \to \Ds^s(\R^n), \quad (\psi,\varphi) \mapsto
\psi \circ \varphi
\]
is $C^r$-smooth we apply the implicit function theorem to show that the
inverse map 
\[
\Ds^{s+r}(\R^n) \to \Ds^s(\R^n),\quad \varphi \mapsto \varphi^{-1}
\]
is a $C^r$-map as well.
\end{Rem}

\begin{Rem}
By considering lifts to $\R^n$ of diffeomorphisms of $\Tn=\R^n/\mathbb{Z}^n$,
the same arguments as in the proof of Theorem \ref{thm1} can be used
to show corresponding results for the group $\Ds^s(\Tn)$ of $H^s$-regular diffeomorphisms on $\Tn$.
\end{Rem}

\noindent In Section \ref{Section 3} and Section \ref{sec:diff_structure} of this paper we discuss various classes of
diffeomorphisms on a closed\footnote{i.e., a compact $C^\infty$-manifold without boundary} manifold $M$.
For any integer $s$ with $s>n/2$ the set $H^s(M)$ of Sobolev maps is defined by using coordinate
charts of $M$. More precisely, let $M$ be a closed manifold of
dimension $n$ and $N$ a $C^\infty$-manifold of dimension $d$. We say that
a continuous map $f:M \to N$ is an element in $H^s(M,N)$ if for any $x \in M$ there exists a chart $\chi:\U \to U \subseteq \R^n$ of $M$ with $x \in \U$, and a chart $\eta:\V \to V \subseteq \R^d$ of $N$ with $f(x) \in \V$, such that $f(\U) \subseteq \V$ and
\[
\eta \circ f \circ \chi^{-1}: U \to V
\]
is an element in the Sobolev space $H^s(U,\mathbb{R}^d)$. Here $H^s(U,\R^d)$ -- similarly defined as $H^s(\R^n,\R^d)$ -- is the Hilbert space of elements in $L^2(U,\R^d)$ whose distributional derivatives up to order $s$ are $L^2$-integrable. In Section \ref{Section 3} we introduce a
$C^\infty$-differentiable structure on the space $H^s(M,N)$ in terms of a specific cover
by open sets which is especially  well suited for proving regularity properties of the composition
of mappings as well as other applications presented in subsequent work.
The main property of this cover of $H^s(M,N)$ is that each of its open sets can be embedded into
a finite cartesian product of Sobolev spaces of $\Hs^s$-maps between Euclidean spaces.

It turns out that this cover makes $H^s(M,N)$ into a $C^\infty$-Hilbert
manifold -- see Section \ref{sec:diff_structure} for details.
In addition, we show in Section \ref{sec:diff_structure} that the $C^\infty$-differentiable structure for $H^s(M,N)$ defined in this way coincides with the one, introduced by Ebin and Marsden in \cite{EM}, \cite{EMF} and defined in terms of a Riemannian metric on $N$. In particular it follows
that the standard differentiable structure does not depend on the
choice of the metric. Now assume in addition that $M$ is
oriented. Then, for any linear isomorphism $A:T_xM \to T_yM$ between
the tangent spaces of $M$ at arbitrary points $x$ and $y$ of $M$, the
determinant $\det(A)$ has a well defined sign. For any integer $s$ with $s>\frac{n}{2}+1$ define
\[
{\mathcal D}^s(M):=\big\{\varphi\in \mbox{Diff}_+^1(M) \big\arrowvert
\varphi \in H^s(M,M) \big\}
\]
where $\mbox{Diff}_+^1(M)$ denotes the set of all orientation
preserving $C^1$ smooth diffeomorphisms of $M$. We will show that ${\mathcal D}^s(M)$ is open in
$H^s(M,M)$ and hence is a $C^\infty$-Hilbert manifold. Elements in $\Ds^s(M)$ are referred to as
orientation preserving $H^s$-diffeomorphisms.

\medskip

\noindent In Section \ref{Section 3} we prove the following

\begin{Th}\label{thm2} Let $M$ be a closed oriented manifold of dimension $n$, $N$ a $C^\infty$-manifold,
and $s$ an integer satisfying $s > n/2 + 1$. Then for any $r \in {\mathbb Z}_{\geq 0}$,
\begin{itemize}
\item[(i)] $\qquad \mu : H^{s + r }(M,N) \times {\mathcal D}^s(M)\rightarrow H^s(M,N),
\ (f, \varphi ) \mapsto f \circ \varphi $
\end{itemize}
and
\begin{itemize}
\item[(ii)]
$\qquad {\tt inv} : {\mathcal D}^{s + r } (M) \rightarrow {\mathcal D}^s(M), \
\varphi \mapsto \varphi ^{-1} $
\end{itemize}
are both $C^r$-maps.
\end{Th}

\begin{Rem}
Various versions of Theorem \ref{thm2} can be found in the literature,
however mostly without proofs -- see e.g. \cite{11}, \cite{EM}, \cite{El}, \cite{27}, \cite{27bis}, \cite{28}, \cite{29};
cf. also \cite{34bis}. A complete, quite involved proof of statement
$(i)$ of Theorem \ref{thm2} can be found in \cite{27bis}, Proposition 3.3 of Chapter 3 and
Theorem 2.1 of Chapter 6. Using the approach sketched above we present
an elementary proof of Theorem \ref{thm2}. In particular, our approach allows us to apply elements of the proof of Theorem \ref{thm1} to show statement $(i)$.
\end{Rem}

\begin{Rem}
Actually Theorem \ref{thm1} and Theorem \ref{thm2} continue to hold if instead of $s$ being an integer it is an arbitrary real number $s > n/2+1$. In order to keep the exposition as elementary as possible we prove Theorem \ref{thm1} and Theorem \ref{thm2} as stated in the main body of the paper and discuss the extension to the case where $s > n/2+1$ is real in Appendix \ref{appendix B}.
\end{Rem}

We finish this introduction by pointing out results on compositions of maps in function spaces different from the ones considered here and some additional literature. In the paper \cite{Llave}, de la Llave and Obaya prove a version of Theorem \ref{thm1} for H\"older continuous maps between open sets of Banach spaces. Using the paradifferential calculus of Bony, Taylor \cite{Taylor} studies the continuity of the composition of maps of low regularity between open sets in $\R^n$ -- see also \cite{Alinhac}.

\hspace{0.1cm}

{\it Acknowledgment}: We would like to thank Gerard Misiolek and Tudor Ratiu for very valuable feedback on an earlier version of this paper.

\section{Groups of diffeomorphisms on $\R^n$}\label{Section 2}

In this section we present a detailed and elementary proof of Theorem
\ref{thm1}. First we prove that the composition map $\mu$ is a $C^r$-map (Proposition
\ref{prop1}) and then, using this result, we show that the inverse map is a
$C^r$-map as well (Proposition \ref{prop2}). To simplify notation we write $\Ds^s \equiv \Ds^s(\R^n)$ and $H^s
\equiv H^s(\R^n)$. Throughout this section, $s$ denotes a nonnegative integer if not stated otherwise.

\subsection{Sobolev spaces $H^s(\R^n,\R)$}\label{subsection_sobolev_spaces}

In this subsection we discuss properties of the Sobolev spaces
$H^s(\R^n,\R)$ needed later. First let us introduce some more
notation. For any $x,y \in \R^n$ denote by $x \cdot y$ the Euclidean
inner product, $x \cdot y=\sum_{k=1}^n x_ky_k$, and by $|x|$ the
corresponding norm , $|x|=(x \cdot x)^{1/2}$. Recall that for $s \in
\Z$, $H^s(\R^n,\R)$ consists of all $L^2$-integrable functions
$f:\R^n \to \R$ with the property that the distributional derivatives
$\partial^\alpha f, \alpha \in {\mathbb Z}^n_{\geq 0}$, up to order $|\alpha|
\leq s$ are $L^2$-integrable as well. Then $H^s(\R^n,\R)$, endowed with
the norm (\ref{hsnorm}), is a Hilbert space and for any multi-index
$\alpha \in {\mathbb Z}^n_{\geq 0}$ with $|\alpha| \leq s$, the differential
operator $\partial^\alpha$ is a bounded linear map,
\[
 \partial^\alpha:H^s(\R^n,\R) \to H^{s-|\alpha|}(\R^n,\R).
\]
Alternatively, one can characterize the spaces $H^s(\R^n,\R)$ via the
Fourier transform. For any $f \in L^2(\R^n,\R) \equiv H^0(\R^n,\R)$,
denote by $\hat f$ its Fourier transform
\[
 \hat f(\xi):=(2 \pi)^{-n/2} \int_{\R^n} f(x) e^{-ix \cdot \xi} dx.
\]
Then $\hat f \in L^2(\R^n,\R)$ and $\|\hat f\|=\|f\|$, where $\|f\|
\equiv \|f\|_0$ denotes the $L^2$-norm of $f$. The formula for the
inverse Fourier transform reads
\[
 f(x) = (2 \pi)^{-n/2} \int_{\R^n} \hat f(\xi) e^{ix \cdot \xi} d\xi.
\]
When expressed in terms of the Fourier transform $\hat f$ of
$f$, the operator $\partial^\alpha, \alpha \in {\mathbb Z}^n_{\geq 0}$ is the
multiplication operator
\[
 \hat f \mapsto (i\xi)^\alpha \hat f
\]
where $\xi^\alpha=\xi_1^{\alpha_1} \cdots \xi_n^{\alpha_n}$ and one
can show that $f \in L^2(\R^n,\R)$ is an element in $H^s(\R^n,\R)$ iff
$(1+|\xi|)^s \hat f$ is in $L^2(\R^n,\R)$ and the $H^s$-norm of $f$,
$\|f\|_s=\big( \sum_{|\alpha| \leq s} \|\xi^\alpha \hat f\|^2
\big)^{1/2}$, satisfies 
\begin{equation}\label{equivalent_norm}
C_s^{-1} \|f\|_s \leq \|f\|_s^\sim \leq C_s \|f\|_s 
\end{equation}
for some constant $C_s \geq 1$ where
\begin{equation}\label{fourier_hsnorm}
 \|f\|_s^\sim := \left( \int_{\R^n} (1+|\xi|^2)^{s} |\hat f(\xi)|^2
 d\xi \right)^{1/2}.
\end{equation}
In this way the Sobolev space $H^s(\R^n,\R)$ can be defined for $s \in \R_{\geq 0}$ arbitrary.
See Appendix \ref{appendix B} for a study of these spaces.\\
Using the Fourier transform one gets the following approximation
property for functions in $H^s(\R^n,\R)$.
\begin{Lemma}\label{lemma_density}
For any $s$ in $\Z$, the subspace $C_c^\infty(\R^n,\R)$ of
$C^\infty$ functions with compact support is dense
in $H^s(\R^n,\R)$.
\end{Lemma}

\begin{Rem}
The proof shows that Lemma \ref{lemma_density} actually holds for any $s$ real with $s \geq 0$.
\end{Rem}

\begin{proof}
In a first step we show that $C^\infty(\R^n,\R) \cap H^{s'}(\R^n,\R)$ is dense in $H^s(\R^n,\R)$ for any integer
$s' \geq s$. Let $\chi:\R \to \R$ be a decreasing $C^\infty$ function
satisfying
\[
 \chi(t)=1 \quad \forall t \leq 1 \quad \mbox{and} \quad \chi(t)=0
 \quad \forall t \geq 2.
\]
For any $f \in H^s(\R^n,\R)$ and $N \in {\mathbb Z}_{\geq 1}$ define
\[
 f_N(x) = (2\pi)^{-n/2} \int_{\R^n} \chi\big( \tfrac{|\xi|}{N} \big)
 \hat f(\xi) e^{ix \cdot \xi} d\xi.
\]
The support of $\chi\big( \tfrac{|\xi|}{N}\big) \hat f(\xi)$ is contained in the ball $\{
|\xi| \leq 2N\}$. Hence $f_N(x)$ is in $C^\infty(\R^n,\R) \cap H^{s'}(\R^n,\R)$ for any $s' \geq 0$. In addition, by the
Lebesgue convergence theorem,
\[
\lim_{N \to \infty} \int_{\R^n} (1+|\xi|)^{2s}
\big(1-\chi(\tfrac{|\xi|}{N})\big)^2 |\hat f(\xi)|^2 d\xi =0.
\]
In view of (\ref{fourier_hsnorm}), we have $f_N \to f$ in
$H^s(\R^n,\R)$. In a second step we show that $C_c^\infty(\R^n,\R)$ is dense in
$C^\infty(\R^n,\R) \cap H^{s'}(\R^n,\R)$ for any integer $s' \geq 0$. We get the desired approximation of
an arbitrary function $f \in C^\infty(\R^n,\R) \cap H^{s'}(\R^n,\R)$ by truncation in the $x$-space. For any
$N \in \mathbb Z_{\geq 1}$, let
\[
 \tilde f_N(x) = \chi\big(\tfrac{|x|}{N}\big)\cdot f(x).
\] 
The support of $\tilde f_N$ is contained in the ball $\{ |x| \leq 2N \}$ and thus
$\tilde f_N \in C_c^\infty(\R^n,\R)$. To see that $f-\tilde f_N=(1-\chi\big(\frac{|x|}{N}\big))f$ converges to $0$ in
$H^{s'}(\R^n,\R)$, note that $f(x)-\tilde f_N(x) = 0$ for any $x \in \R^n$ with $|x| \leq N$. Furthermore it is easy
to see that
\[
 \sup_{\substack{x \in \R^n \\ |\alpha| \leq s'}} \big| \partial^\alpha
\big( 1 - \chi\big(\tfrac{|x|}{N}\big)\big)\big| \leq M_{s'}
\]
for some constant $M_{s'} > 0$ independent on $N$. Hence for any $\alpha \in \mathbb Z^n_{\geq 0}$ with
$|\alpha| \leq s'$, by Leibniz' rule,
\begin{eqnarray*}
\|\partial^\alpha f - \partial^\alpha \tilde f_N\|&=& \|\partial^\alpha
\left( \big(1-\chi\big(\tfrac{|x|}{N}\big)\big)\cdot f(x)\right) \|\\
&\leq& \sum_{\beta + \gamma=\alpha} \| \partial^\beta\big(1-\chi\big(\tfrac{|x|}{N}\big)\big)
\cdot \partial^\gamma f\|.
\end{eqnarray*}
Using that  $1-\chi\big(\frac{|x|}{N}\big)=0$ for any $|x|\leq N$ we conclude that
\[
\|\partial^\beta \big(1-\chi\big(\tfrac{|x|}{N}\big)\big) \cdot \partial^\gamma f\|
\leq M_{s'} \left(\int_{|x|\geq N} |\partial^\gamma f|^2 dx\right)^{1/2}
\]
and hence, as $f \in H^{s'}(\R^n,\R)$,
\[
 \lim_{N \to \infty} \|\partial^\alpha f - \partial^\alpha \tilde f_N\| =0.
\]
\end{proof}

\noindent To state regularity properties of elements in $H^s(\R^n,\R)$, introduce
for any $r \in \Z$ the space $C^r(\R^n,\R)$ of functions $f:\R^n \to \R$
with continuous partial derivatives up to order $r$. Denote by
$\|f\|_{C^r}$ the $C^r$-norm of $f$,
\[
 \|f\|_{C^r} = \sup_{x \in \R^n} \sup_{|\alpha| \leq r}
 |\partial^\alpha f(x)|.
\]
By $C^r_b(\R^n,\R)$ we denote the Banach space of functions $f$ in
$C^r(\R^n,\R)$ with $\|f\|_{C^r} < \infty$ and by $C_0^r(\R^n,\R)$ the
subspace of functions $f$ in $C^r(\R^n,\R)$ vanishing at
infinity. These are functions in $C^r(\R^n,\R)$ with the property that
for any $\varepsilon > 0$ there exists $M \geq 1$ so that
\[
 \sup_{|\alpha| \leq r} \sup_{|x| \geq M} |\partial^\alpha f(x)| < \varepsilon.
\]
Then
\[
C_0^r (\R^n,\R) \subseteq C_b^r(\R^n,\R) \subseteq C^r(\R^n,\R).
\]
By the triangle inequality one sees that $C_0^r(\R^n,\R)$ is a closed
subspace of $C_b^r(\R^n,\R)$. The following result is often referred to
as Sobolev embedding theorem.
\begin{Prop}\label{prop_sobolev_imbedding}
For any $r \in \Z$ and any integer $s$ with $s > n/2$, the space
$H^{s+r}(\R^n,\R)$ can be embedded into $C_0^r(\R^n,\R)$. More
precisely $H^{s+r}(\R^n,\R) \subseteq C_0^r(\R^n,\R)$ and there exists $K_{s,r} \geq 1$ so that 
\[
\|f\|_{C^r} \leq K_{s,r} \|f\|_{s+r} \quad \forall f \in H^{s+r}(\R^n,\R).
\]
\end{Prop}

\begin{Rem}\label{rem_prop_sobolev_imbedding}
The proof shows that Proposition \ref{prop_sobolev_imbedding} holds for any real $s$ with $s>n/2$.
\end{Rem}

\begin{proof}
As for $s > n/2$
\[
\int_{\R^n} (1+|\xi|^2)^{-s} d\xi < \infty
\]
one gets by the Cauchy-Schwarz inequality for any $f \in C_c^\infty(\R^n,\R)$ and $\alpha \in \mathbb
Z^n_{\geq 0}$ with $|\alpha| \leq r$
\begin{eqnarray}
\nonumber
\sup_{x \in \R^n} |\partial^\alpha f(x)| \leq (2 \pi)^{-n/2}
\int_{\R^n} |\hat f(\xi)| \,\, |\xi|^\alpha d\xi\\
\nonumber
\leq \Big( \int_{\R^n} (1+|\xi|^2)^{-s} d\xi \Big)^{1/2} (2\pi)^{-n/2}
\Big( \int_{\R^n} |\hat f(\xi)|^2 (1+|\xi|^2)^{s+r}d\xi \Big)^{1/2}\\
\label{imbedding_ineq_b}
\leq K_{r,s} \|f\|_{r+s}
\end{eqnarray}
for some $K_{r,s}>0$. By Lemma \ref{lemma_density}, an arbitrary element $f \in H^{s+r}(\R^n,\R)$ can be approximated by a sequence $(f_N)_{N \geq 1}$ in $C_c^\infty(\R^n,\R)$. As $C_0^r(\R^n,\R)$ is a Banach space, it then follows from (\ref{imbedding_ineq_b}) that $(f_N)_{N \geq 1}$ is a Cauchy sequence in $C_0^r(\R^n,\R)$ which converges to some function $\tilde f$ in $C_0^r(\R^n,\R)$. In particular, for any compact subset $K \subseteq \R^n$,
\[
 \left. f_N \right|_K \to \tilde f \left. \mbox{} \right|_K \mbox{ in } L^2(K,\R). 
\]
This shows that $\tilde f \equiv f$ a.e. and hence $f \in C_0^r(\R^n,\R)$.
\end{proof}

\noindent As an application of Proposition \ref{prop_sobolev_imbedding} one gets
the following
\renewcommand{\thefootnote}{\fnsymbol{footnote}}
\begin{Coro}\label{cor_openness}
Let $s$ be an integer with $s >n/2+1$. Then the following statements hold:
\begin{itemize}
\item[(i)] For any $\varphi \in \Ds^s$, the linear operators $d_x\varphi, d_x\varphi^{-1}:\R^n \to \R^n$
are bounded uniformly in $x \in \R^n$\footnote{Here $d_x\varphi^{-1}\equiv d_x(\varphi^{-1})$ where 
$\varphi\circ\varphi^{-1}=\id$.}. In particular,
\[
 \inf_{x \in \R^n} \det d_x\varphi > 0.
\]
\item[(ii)] $\Ds^s - \id=\{\varphi - \id
\,|\, \varphi \in \Ds^s \}$ is an open subset of $H^s$. Hence the map
\[
\Ds^s \to H^s, \quad \varphi \mapsto \varphi - \id
\] 
provides a global chart for $\Ds^s$, giving $\Ds^s$ the structure of a $C^\infty$-Hilbert manifold modeled
on $H^s$. 
\item[(iii)] For any $\varphi_\bullet \in \D^s$ such that
\[
 \inf_{x \in \R^n} \det d_x\varphi_\bullet > M > 0
\]
there exist an open neighborhood $\;U_{\varphi_\bullet}$ of $\varphi_\bullet$ in $\D^s$ and $C>0$ such that for any $\varphi$ in $U_{\varphi_\bullet}$,
\[
\inf_{x \in \R^n} \det d_x\varphi \geq M \quad \mbox{and} \quad
 \sup_{x \in \R^n} \left| d_x \varphi^{-1}\right| < C.\footnote{For a linear operator $A:\R^n \to \R^n$, denote by $|A|$ its operator norm, $|A|:=\sup_{|x|=1}|Ax|$}
\]
\end{itemize}
\end{Coro}

\begin{Rem}
The proof shows that Corollary \ref{cor_openness} holds for any real $s$ with $s > n/2+1$.
\end{Rem}

\begin{proof} $(i)$ Introduce
\[
 \mathscr C^1(\R^n) := \big\{ \varphi \in \mbox{Diff}_+^1(\R^n) \,\big|\,
   \varphi - \id \in C_0^1(\R^n) \big\}
\]
where $C_0^1(\R^n) \equiv C_0^1(\R^n,\R^n)$ is the space of $C^1$-maps $f:\R^n
\to \R^n$, vanishing together with their partial derivatives $\dx{i}f
\,\,(1 \leq i \leq n)$ at infinity. By Proposition
\ref{prop_sobolev_imbedding}, $H^s$ continuously embeds into
$C_0^1(\R^n)$ for any integer $s$ with $s > n/2+1$. In particular,
$\Ds^s \hookrightarrow \mathscr C^1(\R^n)$. We now prove that for any $\varphi \in \mathscr C^1(\R^n)$, $d\varphi$ and $d\varphi^{-1}$ are bounded on $\R^n$. Clearly, for any $\varphi \in \mathscr C^1(\R^n)$, $d\varphi$ is bounded on
$\R^n$. To show that $d\varphi^{-1}$ is bounded as well introduce
for any $f \in C_0^1(\R^n)$ the function $F(f):\R^n \to \R$ given by
\begin{eqnarray*}
 F(f)(x)&:=& \det\big( \id + d_xf \big) -1\\
&=& \det\big( (\delta_{i1} + \dx{1}f_i)_{1 \leq i \leq
  n},\ldots,(\delta_{in}+\dx{n}f_i)_{1 \leq i \leq n} \big) -1 
\end{eqnarray*}
where $f(x)=\big(f_1(x),\ldots,f_n(x)\big)$. As 
\[
 \lim_{|x| \to \infty} \dx{k}f_i(x)=0 \quad \mbox{for any} \quad 1
 \leq i,k \leq n
\]
one has 
\begin{equation}\label{vanishing_infinity}
 \lim_{|x| \to \infty} F(f)(x)=0.
\end{equation}
It is then straightforward to verify that $F$ is a continuous map,
\[
 F:C_0^1(\R^n) \to C_0^0(\R^n,\R).
\]
Choose an arbitrary element $\varphi$ in $\mathscr C^1(\R^n)$. Then
(\ref{vanishing_infinity}) implies that
\begin{equation}\label{nonzero_det}
 M_1:= \inf_{x \in \R^n} \det(d_x\varphi) > 0.
\end{equation}
As the differential of the inverse, $d_x\varphi^{-1} =
\big( d_{\varphi^{-1}(x)} \varphi \big)^{-1}$, can be computed in terms
of the cofactors of $d_{\varphi^{-1}(x)} \varphi$ and
$1/\det(d_{\varphi^{-1}(x)}\varphi)$ it follows from
(\ref{nonzero_det}) that
\begin{equation}\label{bounded_inverse}
M_2:= \sup_{x \in \R^n} |d_x\varphi^{-1}| < \infty
\end{equation}
where $|A|$ denotes the operator norm of a linear operator $A:\R^n \to
\R^n$. \\
$(ii)$ Using again that $\Ds^s$ continuously embeds into $\mathscr C^1(\R^n)$ it remains to prove that
$\mathscr C^1(\R^n)-\id$ is an open subset of $C_0^1(\R^n)$. Note that the map $F$ introduced above is continuous. 
Hence there exists a neighborhood
$U_\varphi$ of $f_\varphi:=\varphi -\id$ in $C_0^1(\R^n)$ so that
for any $f \in U_\varphi$
\begin{equation}\label{continuous_jacobian}
\sup_{x \in \R^n} |d_xf - d_xf_\varphi| \leq \frac{1}{2M_2}
\end{equation}
and
\begin{equation}\label{continuous_F}
\sup_{x \in \R^n} \big| F\big(f\big)(x) - F\big(f_\varphi\big)(x)\big|
\leq \frac{M_1}{2}
\end{equation}
with $M_1,M_2$ given as in
(\ref{nonzero_det})-(\ref{bounded_inverse}). We claim that $\id+f \in
\mathscr C^1(\R^n)$ for any $f \in U_\varphi$. As $\varphi \in \mathscr C^1(\R^n)$
was chosen arbitrarily it then would follow that $\mathscr C^1(\R^n) -\id$ is
open in $C_0^1(\R^n)$. First note that by (\ref{continuous_F}),
\[
 0 < M_1/2 \leq \det(\id + d_xf) \quad \forall x \in \R^n, \forall f \in U_\varphi.
\]
Hence $\id+f$ is a local diffeomorphism on $\R^n$ and it remains to show
that $\id +f$ is 1-1 and onto for any $f$ in $U_\varphi$. Choose $f \in
U_\varphi$ arbitrarily. To see that $\id+f$ is 1-1 it suffices to prove
that $\psi:=(\id +f) \circ \varphi^{-1}$ is 1-1. Note that
\[
 \psi = (\id +f_\varphi + f - f_\varphi) \circ \varphi^{-1} = \id +
 (f-f_\varphi)\circ \varphi^{-1}.
\]
For any $x,y \in \R^n$, one therefore has 
\[
 \psi(x)-\psi(y)=x-y + (f-f_\varphi)\circ\varphi^{-1}(x) - (f-f_\varphi)\circ\varphi^{-1}(y).
\]
By (\ref{bounded_inverse}) and (\ref{continuous_jacobian})
\begin{eqnarray*}
 |(f-f_\varphi)\circ\varphi^{-1}(x) -
 (f-f_\varphi)\circ\varphi^{-1}(y)| &\leq& \frac{1}{2M_2}
 |\varphi^{-1}(x) - \varphi^{-1}(y)|\\
&\leq& \frac{1}{2} |x-y|
\end{eqnarray*}
and thus
\[
 |(x-y)-\big(\psi(x)-\psi(y)\big)| \leq \frac{1}{2}|x-y| \quad \forall
 x,y \in \R^n
\]
which implies that $\psi$ is 1-1. To prove that $\id+f$ is onto we show
that $R_f:=\{x+f(x) \,|\, x\in \R^n \}$ is an open and closed subset of
$\R^n$. Being nonempty, one then has $R_f=\R^n$. As $\id+f$ is a local
diffeomorphism on $\R^n$, $R_f$ is open. To see that it is closed,
consider a sequence $(x_k)_{k \geq 1}$ in $\R^n$ so that
$y_k:=x_k+f(x_k)$, $k \geq 1$, converges. Denote the limit by $y$. As
$\lim_{|x| \to \infty} f(x)=0$, the sequence $\big(f(x_k)\big)_{k \geq
1}$ is bounded, hence $x_k=y_k-f(x_k)$ is a bounded sequence and
therefore admits a convergent subsequence $(x_{k_i})_{i \geq 1}$ whose
limit is denoted by $x$. Then
\begin{eqnarray*}
 y&=& \lim_{i \to \infty} x_{k_i} + \lim_{i \to \infty} f(x_{k_i})\\
&=& x +f(x)
\end{eqnarray*}
i.e. $y \in R_f$. This shows that $R_f$ is closed and finishes the proof of item $(ii)$. The proof of $(iii)$ is
straightforward and we leave it to the reader.
\end{proof}

\noindent The following properties of multiplication of functions in Sobolev spaces are well known --
see e.g. \cite{adams}.

\begin{Lemma}\label{lem_imbedding}
Let $s, s'$ be integers with $s > n/2$ and $0 \leq s' \leq s$. Then there exists $K>0$ so that
for any $f \in \Hs^s(\R^n,\R)$, $g \in \Hs^{s'}(\R^n,\R)$, the product $f \cdot g$ is in $\Hs^{s'}(\R^n,\R)$ and
\begin{equation}
\label{multiplication}
 \hnorm{f \cdot g}{s'} \leq K \hnorm{f}{s} \hnorm{g}{s'}.
\end{equation}
In particular, $\Hs^s(\R^n,\R)$ is an algebra.
\end{Lemma}

\begin{Rem}\label{rem_lem_imbedding}
The proof shows that Lemma \ref{lem_imbedding} remains true for any real $s$ and $s'$ with $s > n/2$ and $0 \leq s' \leq s$.
\end{Rem}

\begin{proof}
First we show that
\[
 (1+|\xi|^2)^{s'/2} \widehat{f \cdot g}(\xi)=(1+|\xi|^2)^{s'/2} (\hat f \ast \hat g)(\xi) \in L^2(\R^n,\R)
\]
where $\ast$ denotes the convolution
\[
 (\hat f \ast \hat g)(\xi)= \int_{\R^n} \hat f(\xi-\eta) \hat g(\eta)d\eta.
\]
By assumption,
\[
 \tilde f(\xi) := \hat f(\xi) \ (1+|\xi|^2)^{s/2} \quad \mbox{and} \quad  \tilde g(\xi) = \hat g(\xi) \ (1+|\xi|^2)^{s'/2} 
\]
are in $L^2(\R^n,\R)$. Note that in view of definition (\ref{fourier_hsnorm}), $\|\tilde f\|=\|f\|_s^\sim$ and $\|\tilde g\|=\|g\|_{s'}^\sim$. It is to show that
\[
 \xi \mapsto (1+|\xi|^2)^{s'/2} \int_{\R^n} \frac{|\tilde f(\xi - \eta)|}{(1+|\xi-\eta|^2)^{s/2}} \frac{|\tilde g(\eta)|}{(1+|\eta|^2)^{s'/2}} d\eta
\]
is square-integrable. We split the domain of integration into two subsets $\{|\eta|>|\xi|/2\}$ and $\{|\eta| \leq |\xi|/2 \}$. Then
\begin{eqnarray*}
&& (1+|\xi|^2)^{s'/2} \int_{|\eta| > |\xi|/2} \frac{|\tilde f(\xi - \eta)|}{(1+|\xi-\eta|^2)^{s/2}} \frac{|\tilde g(\eta)|}{(1+|\eta|^2)^{s'/2}} d\eta\\
&\leq& 2^{s'} (1+|\xi|^2)^{s'/2} \int_{|\eta| > |\xi|/2} \frac{|\tilde f(\xi - \eta)|}{(1+|\xi-\eta|^2)^{s/2}} \frac{|\tilde g(\eta)|}{(1+|\xi|^2)^{s'/2}} d\eta\\
&\leq& 2^{s'}  \int_{\R^n} \frac{|\tilde f(\xi - \eta)|}{(1+|\xi-\eta|^2)^{s/2}} |\tilde g(\eta)| d\eta\\
&=& 2^{s'} |\hat f| \ast |\tilde g| (\xi).
\end{eqnarray*}
By Young's inequality (see e.g. Theorem 1.2.1 in \cite{Marti}), 
\[
\big|\big| \, |\hat f| \ast |\tilde g| \, \big|\big| \leq \|\hat f\|_{L^1} \| \tilde g\|
\] 
and
\[
\|\hat f\|_{L^1} \leq \left( \int_{\R^n} (1+|\xi|^2)^{s} |\hat f(\xi)|^2 d\xi \right)^{1/2} 
\left( \int_{\R^n} (1+|\xi|^2)^{-s} d\xi \right)^{1/2}\,.
\]
This implies that
\[
\big|\big| \, |\hat f| \ast |\tilde g| \, \big|\big| \leq C \|f\|_s \|g\|_{s'}^\sim.
\]
Similarly, one argues for the integral over the remaining subset. Note that on the domain $\{|\eta| \leq |\xi|/2 \}$ one has
\[
 (1+|\xi-\eta|^2) \geq (1+|\eta|^2) \quad \mbox{and} \quad (1+|\xi-\eta|^2) \geq \frac{1}{4} (1+|\xi|^2)
\]
and hence
\[
 (1+|\xi-\eta|^2)^{s/2} \geq (1+|\eta|^2)^{(s-s')/2} 2^{-s'} (1+|\xi|^2)^{s'/2}
\]
Hence
\begin{eqnarray*}
&& (1+|\xi|^2)^{s'/2} \int_{|\eta| \leq |\xi|/2} \frac{|\tilde f(\xi - \eta)|}{(1+|\xi-\eta|^2)^{s/2}} \frac{|\tilde g(\eta)|}{(1+|\eta|^2)^{s'/2}} d\eta \\
& \leq &  2^{s'} \int_{|\eta| \leq |\xi|/2} |\tilde f(\xi-\eta)| \frac{|\tilde g(\eta)|}{(1+|\eta|^2)^{s/2}} d\eta 
\end{eqnarray*}
and the $L^2$-norm of the latter convolution is bounded by
\[
\|\tilde f\| \, \|\tilde g(\eta)/(1+|\eta|^2)^{s/2}\|_{L^1} \leq C\|f\|_s \|g\| \leq C \|f\|_s\|g\|_{s'}
\]
with an appropriate constant $C > 0$.
\end{proof}

\noindent The following results concern the chain rule of differentiation for functions in $H^1(\R^n,\R)$.

\begin{Lemma}\label{lemma_chain_rule}
Let $\varphi \in \emph{Diff}_+^{\;1}(\R^n)$ with $d\varphi$ and $d\varphi^{-1}$ bounded on all of $\R^n$. Then the
 following statements hold:
\begin{itemize}
\item[(i)] The right translation by $\varphi$, $f \mapsto R_\varphi(f):=f \circ \varphi$ is a bounded linear map
on $L^2(\R^n,\R)$.
\item[(ii)] For any $f \in H^1(\R^n,\R)$, the composition $f \circ \varphi$ is again in $H^1(\R^n,\R)$ and the
differential $d(f \circ \varphi)$ is given by the map $df \circ \varphi \cdot d\varphi \in L^2(\R^n,\R^n)$,
\begin{equation}\label{differential_of_composition}
d(f\circ \varphi) = (df)\circ \varphi \cdot d\varphi.
\end{equation} 
\end{itemize}
\end{Lemma}

\begin{proof}
$(i)$ For any $f \in L^2(\R^n,\R)$, the composition $f\circ \varphi$ is measurable. As
\[
 M_1:=\inf_{x \in \R^n} \det(d_x\varphi) = \big(\sup_{x \in \R^n} \det d_x\varphi^{-1}\big)^{-1} > 0 
\]
one obtains by the transformation formula
\begin{eqnarray*}
 \int_{\R^n} \big|f\big(\varphi(x)\big)\big|^2 dx &\leq& \frac{1}{M_1} \int_{\R^n} \big|f\big(\varphi(x)\big)\big|^2 \det(d_x\varphi) dx\\
&=& \frac{1}{M_1} \int_{\R^n} |f(x)|^2dx
\end{eqnarray*}
and thus $f \circ \varphi \in L^2(\R^n,\R)$ and the right translation $R_\varphi$ is a bounded linear map on $L^2(\R^n,\R)$.\\
$(ii)$ For any $f \in C_c^\infty(\R^n,\R)$, $f \circ \varphi \in H^1(\R^n,\R)$ and (\ref{differential_of_composition}) holds by the standard chain rule of differentiation. Furthermore for any $f \in H^1(\R^n,\R)$, $df \in L^2(\R^n,\R^n)$ and hence by $(i)$, $(df) \circ \varphi \in L^2(\R^n,\R^n)$. As $d\varphi$ is continuous and bounded by assumption it then follows that for any $1 \leq i \leq n$
\[
 \sum_{k=1}^n \big(\dx{k} f\big) \circ \varphi \cdot \dx{i}\varphi_k \in L^2(\R^n,\R)
\] 
where $\varphi_k(x)$ is the $k$'th component of $\varphi(x)$, $\varphi(x)=\big(\varphi_1(x),\ldots,\varphi_n(x)\big)$. By Lemma \ref{lemma_density}, $f$ can be approximated by $(f_N)_{N \geq 1}$ in $C_c^\infty(\R^n,\R)$. By the chain rule, for any $1 \leq i \leq n$, one has
\[
 \dx{i} (f_N \circ \varphi) = \sum_{k=1}^n (\dx{k} f_N) \circ \varphi \cdot \dx{i}\varphi_k
\]
and in view of $(i)$, in $L^2$,
\begin{equation}\label{sum_converge}
 \sum_{k=1}^n (\dx{k} f_N)\circ \varphi \cdot \dx{i}\varphi_k \underset{N \to \infty}{\longrightarrow}
 \sum_{k=1}^n (\dx{k} f)\circ \varphi \cdot \dx{i}\varphi_k.
\end{equation}
Moreover, for any test function $g \in C_c^\infty(\R^n,\R)$,
\[
- \int_{\R^n} \dx{i} g \cdot f_N \circ \varphi dx = \sum_{k=1}^n \int_{\R^n} g \cdot \big(\dx{k}f_N\big) 
\circ \varphi \cdot \dx{i}\varphi_k dx.
\]
By taking the limit $N \to \infty$ and using \eqref{sum_converge}, one sees that the distributional derivative
$\dx{i} (f\circ \varphi)$ equals $\sum_{k=1}^n (\dx{k} f)\circ \varphi \cdot \dx{i}\varphi_k$ for any
$1 \leq i \leq n$. Therefore, $f \circ \varphi \in H^1(\R^n,\R)$ and
$d(f \circ \varphi)=df \circ \varphi \cdot d\varphi$ as claimed.
\end{proof}

The next result concerns the product rule of differentiation in Sobolev spaces. To state the result, introduce for any integer $s$ with $s >n/2$ and $\varepsilon > 0$ the set
\[
 U_\varepsilon^s:=\big\{ g \in H^s(\R^n,\R) \,\big|\, \inf_{x \in \R^n}\big(1 + g(x)\big) > \varepsilon \big\}.
\]
By Proposition \ref{prop_sobolev_imbedding}, $U_\varepsilon^s$ is an open subset of $H^s(\R^n,\R)$ and so is
\[
 U^s:= \bigcup_{\varepsilon > 0} U_\varepsilon^s.
\]

\noindent Note that $U^s$ is closed under multiplication. More precisely, if $g \in U_\varepsilon^s$ and $h \in U_\delta^s$, then $g +h + gh \in U_{\varepsilon \delta}^s$. Indeed, by Lemma \ref{lem_imbedding}, $gh \in H^s(\R^n,\R)$, and hence so is $g + h + gh$. In addition, $1+g+h+gh=(1+g)(1+h)$ satisfies $\inf_{x \in \R^n} (1+g)(1+h) > \varepsilon \delta$ and thus $g + h +gh$ is in $U_{\varepsilon \delta}^s$.

\begin{Lemma}\label{lemma_division}
Let $s,s'$ be integers with $s > n/2$ and $0 \leq s' \leq s$. Then for any $\varepsilon > 0$ and $K > 0$ there exists a constant $C \equiv C(\varepsilon,K;s,s')>0$ so that for any $f \in H^{s'}(\R^n,\R)$ and $g \in U_\varepsilon^s$ with $\|g\|_s < K$, one has $f/(1+g) \in H^{s'}(\R^n,\R)$ and
\begin{equation}\label{division_ineq}
\|f/(1+g)\|_{s'} \leq C \|f\|_{s'}.
\end{equation}
Moreover, the map 
\begin{equation}\label{division_map} 
 H^{s'}(\R^n,\R) \times U^s \to H^{s'}(\R^n,\R), \quad (f,g) \mapsto f/(1+g)
\end{equation}
is continuous.
\end{Lemma}

\begin{Rem}\label{rem_lemma_division}
The proof shows that Lemma \ref{lemma_division} continues to hold for any $s$ real with $s > n/2$. The case where in addition $s'$ is real is treated in Appendix \ref{appendix B}. 
\end{Rem}

\begin{proof}
We prove the claimed statement by induction with respect to $s'$. For $s'=0$, one has for any $f$ in $L^2(\R^n,\R)$ and $g \in U_\varepsilon^s$
\[
 \left\| \frac{f}{1+g} \right\| \leq \frac{1}{\varepsilon} \|f\|.
\]
Moreover, for any $f_1,f_2 \in L^2(\R^n,\R)$, $g_1,g_2 \in U_\varepsilon^s$
\begin{eqnarray*}
\varepsilon^2 \|\tfrac{f_1}{1+g_1}-\tfrac{f_2}{1+g_2}\| &\leq& \|(f_1-f_2)+f_1(g_2-g_1)+(f_1-f_2)g_1\|\\
&\leq& \big(1+\|g_1\|_{C^0}\big) \|f_1-f_2\| +  \|f_1\| \,\,\|g_2-g_1\|_{C^0}.
\end{eqnarray*}
Hence by Proposition \ref{prop_sobolev_imbedding},
\[
 \varepsilon^2 \|\tfrac{f_1}{1+g_1} -\tfrac{f_2}{1+g_2}\| \leq \big(1 + K_{s,0}\|g_1\|_s\big)\|f_1-f_2\| +  K_{s,0} \|f_1\| \,\, \|g_2-g_1\|_s
\]
and it follows that for any $\varepsilon > 0$
\[
 L^2(\R^n,\R) \times U_\varepsilon^s \to L^2(\R^n,\R), \quad (f,g) \mapsto f/(1+g)
\]
is continuous. As $\varepsilon > 0$ was taken arbitrarily, we see that the map \eqref{division_map} is continuous as
well. Thus the claimed statements are proved in the case $s'=0$. 

Now, assuming that \eqref{division_ineq} and \eqref{division_map} hold for all $1\le s'\le k-1$, we will prove that
they hold also for $s'=k$. Take $f \in H^{s'}(\R^n,\R)$ and $g \in U_\varepsilon^s$. 
First, we will prove that $f/(1+g)\in H^{s'}(\R^n,\R)$ and
\[
\dx{i} \left(\frac{f}{1+g}\right) = \frac{\dx{i}f}{1+g} - \frac{\dx{i}(fg)-g\cdot\dx{i}f}{(1+g)^2}
\quad (1\leq i \leq n).
\]
Indeed, by Lemma \ref{lemma_density}, there exists $(f_N)_{N\geq 1}, (g_N)_{N \geq 1} \subseteq C_c^\infty(\R^n,\R)$
so that $f_N \to f$ in $H^{s'}(\R^n,\R)$ and $g_N \to g$ in $H^s(\R^n,\R)$. As $U_\varepsilon^s$ is open in
$H^s(\R^n,\R)$ we can assume that $(g_N)_{N \geq 1} \subseteq U_\varepsilon^s$. By the product rule of
differentiation, one has for any $N \geq 1$, $1 \leq i \leq n$
\begin{equation}\label{division_derivative}
\dx{i}\left( \frac{f_N}{1+g_N} \right) = \frac{\dx{i}f_N}{1+g_N} - \frac{\dx{i}(f_N g_N)-g_N\cdot\dx{i}f_N}{(1+g_N)^2}.
\end{equation}
As $\dx{i}f_N \underset{N \to \infty}{\longrightarrow} \dx{i}f$ in $H^{s'-1}(\R^n,\R)$ it follows by the induction
hypothesis that $\frac{\dx{i}f}{1+g}\in H^{s'-1}(\R^n,\R)$ and
\begin{equation}\label{convergence_first_part}
\frac{\dx{i}f_N}{1+g_N} \underset{N \to \infty}{\longrightarrow} \frac{\dx{i}f}{1+g} \quad \mbox{in}
\quad H^{s'-1}(\R^n,\R).
\end{equation}
By Lemma \ref{lem_imbedding}, $2g_N + g_N^2 \,(N \geq 1)$ and $2g+g^2$ are in $H^s(\R^n,\R)$ and
\begin{equation}\label{convergence_denominator}
2g_N + g_N^2 \underset{N \to \infty}{\longrightarrow} 2g+g^2 \quad \mbox{in} \quad H^s(\R^n,\R).
\end{equation}
As
\[
\inf_{x \in \R^n} \big( 1+g_N(x)\big)^2 > \varepsilon^2 \quad \mbox{and} \quad \inf_{x \in \R^n} 
\big(1+g(x)\big)^2 > \varepsilon^2
\]
it follows that $2g_N+g_N^2 \,(N \geq 1)$ and $2g+g^2$ are elements in $U^s_{\varepsilon^2}$.
By Lemma \ref{lem_imbedding}, $f_N \cdot g_N \,(N \geq 1), f \cdot g$ are in $H^{s'}(\R^n,\R)$ and
$f_N \cdot g_N \underset{N \to \infty}{\longrightarrow} f \cdot g$ in $H^{s'}(\R^n,\R)$. Therefore 
\begin{equation}\label{convergence_second_part_a}
\dx{i}(f_N \cdot g_N) \to \dx{i} (f \cdot g) \quad \mbox{in} \quad H^{s'-1}(\R^n,\R).
\end{equation}
Similarly, as $\dx{i}f_N \underset{N \to \infty}{\longrightarrow} \dx{i}f$ in $H^{s'-1}(\R^n,\R)$ it follows again
by Lemma \ref{lem_imbedding} that $g_N \cdot \dx{i}f_N \,(N \geq 1), g \cdot \dx{i}f$ are in $H^{s'-1}(\R^n,\R)$ and
\begin{equation}\label{convergence_second_part_b}
g_N \cdot \dx{i}f_N \underset{N \to \infty}{\longrightarrow} g \cdot \dx{i}f \quad \mbox{in} \quad H^{s'-1}(\R^n,\R).
\end{equation}
It follows from \eqref{convergence_denominator}-\eqref{convergence_second_part_b}, and the induction hypothesis that
\begin{equation}\label{convergence_fraction}
\frac{\dx{i}(f_Ng_N)-g_N\cdot\dx{i}f_n}{(1+g_N)^2} \underset{N \to \infty}{\longrightarrow} 
\frac{\dx{i}(fg)-g\cdot\dx{i}f}{(1+g)^2} 
\quad \mbox{in} \quad H^{s'-1}(\R^n,\R).
\end{equation}
In view of \eqref{convergence_first_part} and \eqref{convergence_fraction}, for any test function
$h \in C_c^\infty(\R^n,\R)$, one has for the distributional derivative of $f/(1+g) \in L^2(\R^n,\R)$,
\[
\begin{array}{ccl}
\left< \dx{i} \left(\frac{f}{1+g}\right),h \right> &=& - \int_{\R^n} \dx{i}h\cdot\frac{f}{1+g} \;dx = - 
\lim\limits_{N \to \infty} \int_{\R^n} \dx{i}h \cdot\frac{f_N}{1+g_N} \;dx \\
&=&\lim\limits_{N \to \infty} \int_{\R^n} h\cdot\left[ \frac{\dx{i}f_N}{1+g_N} - 
\frac{\dx{i}(f_Ng_N)-g_N\cdot\dx{i}f_n}{(1+g_N)^2}\right]\,dx\\
&=&\int_{\R^n} h\cdot \left(\frac{\dx{i}f}{1+g} - 
\frac{\dx{i}(fg)-g\cdot\dx{i}f}{(1+g)^2}\right)\,dx.
\end{array}
\]
This shows that for any $1 \leq i \leq n$,
\begin{equation}\label{eq:product_rule}
\dx{i}\left(\frac{f}{1+g}\right) = \frac{\dx{i}f}{1+g} - \frac{\dx{i}(fg)-g\cdot\dx{i}f}{(1+g)^2}
\in H^{s'-1}(\R^n,\R).
\end{equation}
Hence, $f/(1+g) \in H^{s'}(\R^n,\R)$. 
Let us rewrite \eqref{eq:product_rule} in the following form
\begin{equation}\label{eq:product_rule'}
\dx{i}\left(\frac{f}{1+g}\right) = \frac{\dx{i}f}{1+g} -
\frac{\frac{\dx{i}(fg)}{1+g}-\frac{g\cdot\dx{i}f}{1+g}}{1+g}.
\end{equation}
By the induction hypothesis there exists $C_1=C_1(\varepsilon,K;s,s')>0$ such that $\forall f\in H^{s'-1}(\R^n,\R)$,
\[
\|f/(1+g)\|_{s'-1}\le C_1\|f\|_{s'-1}\,.
\]
This together with \eqref{eq:product_rule'} and the triangle inequality imply \eqref{division_ineq}.
The continuity of \eqref{division_map} follows immediately from the induction hypothesis, 
Lemma \ref{lem_imbedding}, and \eqref{eq:product_rule'}.
\end{proof}

\subsection{The topological group $\Ds^s(\R^n)$}

In this subsection we show
\begin{Prop}\label{tg}
For any integer $s$ with $s>n/2+1$, $(\Ds^s,\circ)$ is a topological group.
\end{Prop}
\noindent First we show that the composition map is continuous. Actually we prove the following slightly stronger statement.

\begin{Lemma}\label{continuous_composition}
Let $s,s'$ be integers with $s>n/2+1$ and $0 \leq s' \leq s$. Then 
\[
 \mu^{s'}:H^{s'}(\R^n,\R) \times \Ds^s \to H^{s'}(\R^n,\R), \quad (f,\varphi) \mapsto f \circ \varphi
\]
is continuous. Moreover, given any $0 \leq s' \leq s, M>0$ and $C>0$ there exists a constant $C_{s'}=C_{s'}(M,C)>0$ so that for any $\varphi \in \Ds^s$ satisfying
\[
 \inf_{x \in \R^n} \det(d_x\varphi) \geq M,\qquad \|\varphi-\id\|_s \leq C
\]
and for any $f \in H^{s'}(\R^n,\R)$, one has
\begin{equation}\label{linear_estimate}
\|f \circ \varphi\|_{s'} \leq C_{s'} \|f\|_{s'}.
\end{equation}
\end{Lemma}

\begin{Rem}\label{rem_continuous_composition}
The proof shows that Lemma \ref{continuous_composition} continues to hold for any $s$ real with $s > n/2+1$. The case where in addition $s'$ is real is treated in Appendix \ref{appendix B}.
\end{Rem}

\begin{proof}
We prove the claimed statement by induction with respect to $s'$. First consider the case $s'=0$. By item $(i)$ of Corollary \ref{cor_openness} and item $(i)$ of Lemma \ref{lemma_chain_rule}, the range of $\mu^0$ is contained in $L^2(\R^n,\R)$. To show the continuity of $\mu^{0}$ at $(f_\bullet,\varphi_\bullet) \in L^2(\R^n,\R) \times \Ds^s$ write for $(f,\varphi) \in L^2(\R^n,\R) \times \Ds^s$
\[
 |f \circ \varphi - f_\bullet \circ \varphi_\bullet| \leq |f \circ \varphi - f_\bullet \circ \varphi| + |f_\bullet \circ \varphi - f_\bullet \circ \varphi_\bullet|.
\]
By Corollary \ref{cor_openness} $(iii)$ one can choose a neighborhood $U_{\varphi_\bullet}$ of $\varphi_\bullet$ in $\Ds^s$ so that for any $\varphi \in U_{\varphi_\bullet}$
\[
 \inf_{x \in \R^n} (\det d_x\varphi)  \geq M
\]
for some constant $M > 0$.
The term $|f \circ \varphi - f_\bullet \circ \varphi|$ can then be estimated by
\[
 \int_{\R^n} |f \circ \varphi - f_\bullet \circ \varphi|^2 dx \leq \frac{1}{M} \int_{\R^n} |f - f_\bullet|^2dy
\]
To estimate the term $|f_\bullet \circ \varphi - f_\bullet \circ \varphi_\bullet|$ apply Lemma \ref{lemma_density} to approximate $f_\bullet$ by $\tilde f_\bullet \in C_c^\infty(\R^n,\R)$ and use the triangle inequality
\[
|f_\bullet \circ \varphi - f_\bullet \circ \varphi_\bullet| \leq |f_\bullet \circ \varphi - \tilde f_\bullet \circ \varphi| +
|\tilde f_\bullet \circ \varphi - \tilde f_\bullet \circ \varphi_\bullet| + |\tilde f_\bullet \circ \varphi_\bullet - f_\bullet \circ \varphi_\bullet|.
\]
For any $\varphi \in U_{\varphi_\bullet}$, one has
\[
 \int_{\R^n} |f_\bullet \circ \varphi - \tilde f_\bullet \circ \varphi|^2dx \leq \frac{1}{M} \int_{\R^n} |\tilde f_\bullet - f_\bullet|^2dy
\]
and
\[
 \int_{\R^n} |\tilde f_\bullet \circ \varphi_\bullet - f_\bullet \circ \varphi_\bullet|^2dx \leq \frac{1}{M} \int_{\R^n} |\tilde f_\bullet -f_\bullet|^2dy.
\]
To estimate the term $|\tilde f_\bullet \circ \varphi - \tilde f_\bullet \circ \varphi_\bullet|$ use that $\tilde f_\bullet$ is Lipschitz on $\R^n$, i.e. $|\tilde f_\bullet(x)-\tilde f_\bullet(y)| \leq L|x-y|$ for some constant $L>0$ depending on the choice of $\tilde f_\bullet$, to get
\[
 \int_{\R^n} |\tilde f_\bullet \circ \varphi - \tilde f_\bullet \circ \varphi_\bullet|^2dx \leq L^2 \int_{\R^n} |\varphi - \varphi_\bullet|^2dx.
\]
Combining the estimates obtained so far, one gets for any $\varphi \in U_{\varphi_\bullet}$ 
\begin{eqnarray*}
  \|f \circ \varphi - f_\bullet \circ \varphi_\bullet\| &\leq& M^{-1/2} \|f - f_\bullet\| + 2M^{-1/2} \|\tilde f_\bullet - f_\bullet\|\\
&+& L \|\varphi - \varphi_\bullet\|
\end{eqnarray*}
implying the continuity of $\mu^0$ at $(f_\bullet,\varphi_\bullet)$. Now assume $1\leq s' \leq s$. For any $(f,\varphi) \in H^{s'}(\R^n,\R)\times \Ds^s$ one has by Lemma \ref{lemma_chain_rule} and Corollary \ref{cor_openness} $(i)$
\[
 d(f \circ \varphi)= df \circ \varphi \cdot d\varphi. 
\]
By the induction hypothesis $df \circ \varphi$ is an element in $H^{s'-1}(\R^n,\R^{n})$. Hence Lemma \ref{lem_imbedding} implies that $df \circ \varphi \cdot d\varphi$ is in $H^{s'-1}(\R^n,\R^{n})$ and we thus have shown that the image of $\mu^{s'}$ is contained in $H^{s'}(\R^n,\R)$. The continuity of $\mu^{s'}$ follows from the induction hypothesis, the estimate
\begin{eqnarray*}
 \|df \circ \varphi \cdot d\varphi - df_\bullet \circ \varphi_\bullet \cdot d\varphi_\bullet\|_{s'-1} &\leq& \|df \circ \varphi \cdot (d\varphi - d\varphi_\bullet)\|_{s'-1} \\
&+& \|(df \circ \varphi - d f_\bullet \circ \varphi_\bullet)\cdot d \varphi_\bullet\|_{s'-1}
\end{eqnarray*}
and Lemma \ref{lem_imbedding} on multiplication of functions in Sobolev spaces. The estimate (\ref{linear_estimate}) is obtained in a similar fashion. For $s'=0$,
\[
 \int_{\R^n} |f \circ \varphi|^2 dx \leq \frac{1}{M} \int_{\R^n} |f|^2dy.
\]
For $1 \leq s' \leq s$, we argue by induction. Let $f \in H^{s'}(\R^n,\R)$. Then by the considerations above, $d(f \circ \varphi)=df \circ \varphi \cdot d\varphi$ and $df \circ \varphi \in H^{s'-1}(\R^n,\R^{n})$. By induction, $\|df \circ \varphi\|_{s'-1} \leq C_{s'-1} \|df\|_{s'-1}$. Hence in view of Lemma \ref{lem_imbedding}, $\|d(f \circ \varphi)\|_{s'-1} \leq K C_{s'-1}\|df\|_{s'-1}$ and for appropriate $C_{s'} > 0$ one gets $\|f \circ \varphi\|_{s'}\leq C_{s'} \|f\|_{s'}$.
\end{proof}

\noindent To prove Proposition \ref{tg} it remains to show the following properties of the inverse map.

\begin{Lemma}\label{continuous_inverse}
Let $s$ be an integer with $s>n/2+1$. Then for any $\varphi \in \Ds^s$, its inverse $\varphi^{-1}$ is again in $\Ds^s$ and
\[
 {\tt inv}:\Ds^s \to \Ds^s,\quad \varphi \mapsto \varphi^{-1}
\]
is continuous.
\end{Lemma}

\begin{proof}
First we prove that the inverse $\varphi^{-1}$ of an arbitrary element $\varphi$ in $\Ds^s$ is again in $\Ds^s$.
It is to show that for any multi-index $\alpha \in \mathbb Z^n_{\geq 0}$ with $|\alpha| \leq s$, one has
$\partial^\alpha(\varphi^{-1}-\id) \in L^2(\R^n)$. Clearly, for $\alpha=0$, one has
\[
 \int_{\R^n} |\varphi^{-1}-\id|^2dx = \int_{\R^n} |\id-\varphi|^2 \det(d_y\varphi) dy < \infty
\]
as $\det(d_y\varphi)$ is bounded by Corollary \ref{cor_openness}. In addition we conclude that
\[
 \Ds^s \to L^2(\R^n),\quad \varphi \mapsto \varphi^{-1} - \id
\]
is continuous. Indeed, for any $\varphi,\varphi_\bullet \in \Ds^s$, write
\[
 \varphi^{-1}(x) - \varphi_\bullet^{-1}(x) = \varphi^{-1} \circ \varphi_\bullet
\big(\varphi_\bullet^{-1}(x)\big) - \varphi^{-1} \circ \varphi \big(\varphi_\bullet^{-1}(x)\big).
\]
By Corollary \ref{cor_openness} $(iii)$, it follows that for any $x \in \R^n$,
\begin{equation}\label{eq:infty_close}
\begin{array}{ccl}
\left| \varphi^{-1}(x)- \varphi_\bullet^{-1}(x)\right| &=& 
\left| \varphi^{-1}(x)-\varphi^{-1}\big(\varphi \circ \varphi_\bullet^{-1}(x)\big)\right|\\
&\le& \sup\limits_{x \in \R^n} \left|d_x\varphi^{-1}\right| \cdot 
|x - \varphi \circ \varphi_\bullet^{-1}(x)|\\
&\le& L \left| (\varphi_\bullet - \varphi) \big(\varphi_\bullet^{-1}(x)\big)\right|
\end{array}
\end{equation}
where $L > 0$ can be chosen uniformly for $\varphi$ close to $\varphi_\bullet$. Hence
\begin{equation}\label{L2_continuity}
\int_{\R^n} |\varphi^{-1} - \varphi_\bullet^{-1}|^2 dx \leq L^2 \int_{\R^n} |\varphi - \varphi_\bullet|^2
\det(d_y\varphi_\bullet) dy
\end{equation}
and the claimed continuity follows. Now consider $\alpha \in \mathbb Z^n_{\geq 0}$ with $1 \leq |\alpha| \leq s$.
We claim that $\partial^\alpha(\varphi^{-1}-\id)$ is of the form
\begin{equation}\label{derivative_form}
\partial^\alpha (\varphi^{-1}-\id) = F^{(\alpha)} \circ \varphi^{-1}
\end{equation}
where $F^{(\alpha)}$ is a continuous map from $\Ds^s$ with values in $H^{s-|\alpha|}$.
Then $\partial^\alpha(\varphi^{-1}-\id)$ is in $L^2(\R^n)$ as
\begin{equation}\label{transformed}
\int_{\R^n} \big| \partial^\alpha( \varphi^{-1}-\id)\big|^2dx = \int_{\R^n}|F^{(\alpha)}|^2 \det(d_y\varphi)dy < \infty.
\end{equation}
To prove (\ref{derivative_form}), first note that $\varphi$ and hence $\varphi^{-1}$ are in $\mbox{Diff}_+^1(\R^n)$.
By the chain rule,
\[
d(\varphi^{-1}-\id) = (d\varphi)^{-1} \circ \varphi^{-1} - \id_n=\big((d\varphi)^{-1}-\id_n\big)\circ \varphi^{-1}
\]
where $\id_n$ is the $n \times n$ identity matrix. The expression $(d\varphi)^{-1}-\id_n$ is of the form
\[
(d\varphi)^{-1}-\id_n = \frac{1}{\det (d\varphi)} ( \Phi - \det(d\varphi) \id_n )
\]
where $\Phi(x)$ is the matrix whose entries are the cofactors of $d_x\varphi$. In particular, each entry of $\Phi(x)$
is a polynomial expression of $(\dx{i}\varphi_j)_{1\leq i,j \leq n}$. Hence by Lemma \ref{lem_imbedding} the
off-diagonal entries of $\Phi(x)$ are in $H^{s-1}(\R^n,\R)$. Furthermore, any diagonal entry of $\Phi(x)$ is an
element in $1+H^{s-1}(\R^n,\R)$ and $\det(d_x\varphi)$ is of the form $1+g$ with $g \in H^{s-1}(\R^n,\R)$ and
$\inf_{x \in \R^n}\big(1+g(x)\big) > 0$. We thus conclude that $\Phi(x)-\det(d_x\varphi) \, \id_n$ is in
$H^{s-1}(\R^n,\R^{n \times n})$ and, in turn, by Lemma \ref{lemma_division}
\begin{equation}\label{derivative_hs_minus_1}
(d\varphi)^{-1} - \id_n \in H^{s-1}(\R^n,\R^{n\times n})
\end{equation}
where $\R^{n \times n}$ denotes the space of all $n \times n$ matrices with real coefficients. In particular,
for $e_i=(0,\ldots,1,\ldots,0) \in \mathbb Z^n_{\geq 0}$ with $1 \leq i \leq n$ we have shown that
\[
\dx{i}(\varphi^{-1}-\id) = F^{(e_i)} \circ \varphi^{-1}.
\]
We point out that by Lemma \ref{lem_imbedding} and Lemma \ref{lemma_division}, $F^{(e_i)}$, when viewed as map from
$\Ds^s$ to $H^{s-1}$, is continuous. We now prove formula (\ref{derivative_form}) for any
$\alpha \in \mathbb Z_{\geq 0}^n$ with $1 \leq |\alpha| \leq s$ by induction. The result has already been established
for $|\alpha|=1$. Assume that it has already been proved for any $\beta \in \mathbb Z^n_{\geq 0}$ with $|\beta|\leq s'$
where $0 \leq s' < s$. Choose any $\alpha \in \mathbb Z^n_{\geq 0}$ with $|\alpha|=s'$. Then by induction hypothesis,
$\partial^\alpha(\varphi^{-1}-\id)=F^{(\alpha)} \circ \varphi^{-1}$ with $F^{(\alpha)} \in H^{s-|\alpha|}$. Note that
$s-|\alpha| \geq 1$. Hence by Lemma \ref{lemma_chain_rule},
\begin{eqnarray*}
d(F^{(\alpha)} \circ \varphi^{-1}) &=& dF^{(\alpha)} \circ \varphi^{-1} \cdot (d\varphi)^{-1} \circ \varphi^{-1}\\
&=& (dF^{(\alpha)} \cdot (d\varphi)^{-1}) \circ \varphi^{-1}.
\end{eqnarray*}
As $\dx{i}F^{(\alpha)} \in H^{s-|\alpha|-1}$ for any $1 \leq i \leq n$ and $(d\varphi)^{-1}-\id_n$ is in the space
$H^{s-1}(\R^n,\R^{n \times n})$ it follows by Lemma \ref{lem_imbedding} that 
\[
dF^{(\alpha)} \cdot (d\varphi)^{-1} \in H^{s-|\alpha|-1}(\R^n,\R^{n \times n}). 
\]
This shows that (\ref{derivative_form}) is valid for any $\beta \in \mathbb Z^n_{\geq 0}$ with $|\beta|=s'+1$ and
the induction step is proved. Hence formula \eqref{derivative_form} is proved and by \eqref{transformed},
we see that $\varphi^{-1} \in \D^s$ if $\varphi \in \D^s$. Note that we proved more: It follows from
\eqref{transformed} and the continuity of $F^{(\alpha)}:\D^s \to H^{s-|\alpha|}, |\alpha| \leq s$,
that the map $\D^s \to H^s(\R^n,\R^n)$
\begin{equation}\label{locally_bounded}
\varphi \mapsto \varphi^{-1}-\id
\end{equation}
is locally bounded. It remains to prove that the inverse map $\D^s \to \D^s, \varphi \mapsto \varphi^{-1}$ is
continuous. We have already seen that $\Ds^s \to L^2(\R^n)$, $\varphi \to \varphi^{-1}-\id$ is continuous. 
Now let $\alpha \in \mathbb Z_{\geq 0}^n$ with $1 \leq |\alpha| \leq s$ and $\varphi_\bullet \in \Ds^s$.
Then for any $\varphi \in \Ds^s$ 
\begin{eqnarray*}
|\partial^\alpha(\varphi^{-1}-\varphi^{-1}_\bullet)| &=& |F^{(\alpha)} \circ \varphi^{-1} - 
F^{(\alpha)}_\bullet \circ \varphi^{-1}_\bullet|\\
&\leq& |F^{(\alpha)} \circ \varphi^{-1} - F^{(\alpha)}_\bullet \circ \varphi^{-1}| + |F^{(\alpha)}_\bullet 
\circ \varphi^{-1} - F^{(\alpha)}_\bullet \circ \varphi^{-1}_\bullet|
\end{eqnarray*}
where $F^{(\alpha)}_\bullet=\left. F^{(\alpha)} \right|_{\varphi_\bullet}$. It follows from the local boundedness of
\eqref{locally_bounded}, Corollary \ref{cor_openness} $(iii)$, and Lemma \ref{continuous_composition} with $s'=0$ that
\[
 \| F^{(\alpha)} \circ \varphi^{-1} - F_\bullet^{(\alpha)} \circ \varphi^{-1}\| \leq C_0 \, 
\|F^{(\alpha)}-F_\bullet^{(\alpha)}\|
\]
where $C_0 > 0$ can be chosen uniformly for $\varphi$ near $\varphi_\bullet$. Together with the continuity of $F^{(\alpha)}$ it then follows that $\|F^{(\alpha)} \circ \varphi^{-1} - F_\bullet^{(\alpha)} \circ \varphi^{-1}\| \to 0$ as $\varphi \to \varphi_\bullet$. To analyze the term $|F_\bullet^{(\alpha)} \circ \varphi^{-1} - F_\bullet^{(\alpha)} \circ \varphi_\bullet^{-1}|$ we argue as in the proof of Lemma \ref{continuous_composition}. Using Lemma \ref{lemma_density} one sees that $\varphi_\bullet$ can be approximated by $\tilde \varphi \in \Ds^s$ with $\tilde \varphi - \id \in C_c^\infty(\R^n,\R^n)$. Then
\begin{eqnarray*}
 |F_\bullet^{(\alpha)} \circ \varphi^{-1} &-& F_\bullet^{(\alpha)} \circ \varphi_\bullet^{-1}| \leq |F_\bullet^{(\alpha)} \circ \varphi^{-1} - \tilde F^{(\alpha)} \circ \varphi^{-1}| + \\
 &+& |\tilde F^{(\alpha)} \circ \varphi^{-1} - \tilde F^{(\alpha)} \circ \varphi_\bullet^{-1}| + |\tilde F^{(\alpha)} \circ \varphi_\bullet^{-1} - F_\bullet^{(\alpha)} \circ \varphi_\bullet^{-1}|
\end{eqnarray*}
where $\tilde F^{(\alpha)} = \left. F^{(\alpha)} \right|_{\tilde \varphi}$. For $\varphi$ near $\varphi_\bullet$ one has
\[
\int_{\R^n} |F_\bullet^{(\alpha)} \circ \varphi^{-1} - \tilde F^{(\alpha)} \circ \varphi^{-1}|^2 dx \leq \int_{\R^n} |F_\bullet^{(\alpha)}- \tilde F^{(\alpha)}|^2 \det (d_y \varphi) dy
\]
and 
\[
\int_{\R^n} |F_\bullet^{(\alpha)} \circ \varphi_\bullet^{-1} - \tilde F^{(\alpha)} \circ \varphi_\bullet^{-1}|^2 dx \leq C \int_{\R^n} |F_\bullet^{(\alpha)}- \tilde F^{(\alpha)}|^2  dy
\]
where $C > 0$ satisfies $\sup_{x \in \R^n} (\det d_x \varphi) \leq C$ for $\varphi$ near $\varphi_\bullet$. To estimate the term $|\tilde F^{(\alpha)} \circ \varphi^{-1} - \tilde F^{(\alpha)} \circ \varphi_\bullet^{-1}|$ note that $\tilde F^{(\alpha)} \in C_c^\infty$. In particular, $\tilde F^{(\alpha)}$ is Lipschitz continuous, i.e.
\[
 |\tilde F^{(\alpha)}(x) - \tilde F^{(\alpha)}(y)| \leq L_1 |x - y| \quad \forall x,y \in \R^n
\]
for some constant $L_1 > 0$ depending on the choice of $\tilde \varphi$. Thus
\[
 \int_{\R^n} |\tilde F^{(\alpha)} \circ \varphi^{-1} - \tilde F^{(\alpha)} \circ \varphi_\bullet^{-1}|^2 dx \leq L_1^2 \int_{\R^n} |\varphi^{-1} - \varphi_\bullet^{-1}|^2 dx 
\]
and in view of (\ref{L2_continuity}) it then follows that $\|\tilde F^{(\alpha)} \circ \varphi^{-1} - \tilde F^{(\alpha)} \circ \varphi_\bullet^{-1}\| \to 0$ as $\varphi \to \varphi_\bullet$. Altogether we have shown that $\|F_\bullet^{(\alpha)} \circ \varphi^{-1} - F_\bullet^{(\alpha)} \circ \varphi_\bullet^{-1}\| \to 0$ as $\varphi \to \varphi_\bullet$.
\end{proof}

\begin{proof}[Proof of Proposition \ref{tg}] The claimed statement follows from Lemma \ref{continuous_composition} and Lemma \ref{continuous_inverse}.
\end{proof}

\subsection{Proof of Theorem \ref{thm1}}\label{subsection_proof_of_thm1}
As a first step we will prove the following
\begin{Prop}\label{prop1}
For any $r\in\Z$ and any integer $s$ with $s > n/2+1$
\begin{equation}
\mu : \Hs^{s+r}(\R^n,\R^d) \times \Ds^s \to \Hs^s(\R^n,\R^d), \quad  (u,\varphi) \mapsto u \circ \varphi
\end{equation}
is a $C^r$-map.
\end{Prop}
\noindent The main ingredient of the proof of Proposition \ref{prop1} is the converse to Taylor's theorem. To state it we first need to introduce some more
notation. Given arbitrary Banach spaces $Y,X_1,\ldots,X_k, k \geq 1$, we denote by $L(X_1,\ldots,X_k;Y)$ the space of
continuous $k$-linear forms on $X_1 \times \ldots \times X_k$ with values in $Y$. In case where $X_i=X$ for
any $1 \leq i \leq k$ we write $L^k(X;Y)$ instead of
$L(X,\ldots,X;Y)$ and set $L^0(X;Y)=Y$. Note that the spaces
$L(X;L^{k-1}(X;Y))$ and $L^k(X;Y)$ can be identified in a canonical
way. The subspace of $L^k(X;Y)$ of symmetric continuous $k$-linear forms is denoted by $L^k_{sym}(X;Y)$. The converse to Taylor's theorem can then be formulated as follows -- see \cite{abraham}, p.6.

\begin{Th}\label{thm_converse_taylor}
Let $U \subseteq X$ be a convex set and $F:U \to Y$, $f_k:U \to L^k_{sym}(X;Y), k=0,\ldots,r$. For any $x \in U$ and $h \in X$ so that $x+h \in U$, define $R(x,h) \in Y$ by
\[
 F(x+h)=F(x)+\sum_{k=1}^r \frac{f_k(x)(h,\ldots,h)}{k!} + R(x,h).
\] 
If for any $0 \leq k \leq r$, $f_k$ is continuous and for any $x \in U$, $\|R(x,h)\|/\|h\|^r \to 0$ as $h \to 0$ then $F$ is of class $C^r$ on $U$ and $d^kF=f_k$ for any $0 \leq k \leq r$. 
\end{Th}

\noindent To prove Proposition \ref{prop1} we first need to establish some auxiliary results.

\begin{Lemma}\label{rho}
Let $s$ be an integer with $s > n/2+1$. To shorten notation, for this lemma and its proof we write $H^s$ instead of $\Hs^s(\R^n,\R)$. Then for any $k \geq 1$, the map $\rho_k$ given by
\begin{eqnarray*}
\rho_k: \Hs^s \times \Ds^s &\to& L^k_{sym}(\Hs^s;\Hs^s) \\
(u,\varphi) & \mapsto & \left[(h_1,\ldots,h_k) \mapsto (u \circ \varphi)\cdot \prod_{i=1}^k h_i
\right]
\end{eqnarray*}
is continuous.
\end{Lemma}
\begin{proof}[Proof of Lemma \ref{rho}] First we note that the map $\rho_k$ is well defined. Indeed for any $(u,\varphi) \in H^s \times \Ds^s$, the function $u \circ \varphi$ is in $H^s$ by Lemma \ref{continuous_composition}. Hence by Lemma \ref{lem_imbedding}, for any $(h_i)_{1 \leq i \leq k} \subseteq H^s$ the function $u \circ \varphi \cdot \prod_{i=1}^k h_i$ is in $H^s$. It follows that $\rho_k(u,\varphi) \in L^k_{sym}(H^s;H^s)$.
To show that $\rho_k$ is continuous consider arbitrary sequences
$(\varphi_l)_{l \geq 1} \subseteq \Ds^s$ and $(u_l)_{l \geq 1} \subseteq \Hs^s$
with $\varphi_l \to \varphi$ in $\Ds^s$ and $u_l \to u$ in $\Hs^s$. By Lemma \ref{lem_imbedding}, one has for any $(h_i)_{1 \leq i \leq k} \subseteq H^s$,
\[
\hnorm{(u \circ \varphi) \cdot \prod_{i=1}^k h_i  - (u_l \circ \varphi_l) \cdot \prod_{i=1}^k h_i}{s} \leq K^{k+1}
\hnorm{u \circ \varphi - u_l \circ \varphi_l}{s} \cdot \prod_{i=1}^k \hnorm{h_i}{s}.
\]
As $ \hnorm{u \circ \varphi - u_l \circ \varphi_l}{s} \to 0$ for $l \to \infty$ by Lemma
\ref{continuous_composition}, the claimed continuity follows.
\end{proof}

\begin{Lemma}\label{lem_poincare}
Let $s$ be an integer with $s > n/2+1$. Given $\varphi_\bullet \in \Ds^s$ choose $\varepsilon > 0$ so
small that $\infR \det (d_x\varphi_\bullet) > \varepsilon$. Then there exists a
convex neighborhood $U \subseteq \Ds^s$ of $\varphi_\bullet$ and a constant $C>0$ with the
property that
\[
\infR \det (d_x \varphi) > \varepsilon \quad \mathrm{and} \quad
\hnorm{\varphi - \id}{s} < C \quad \forall \varphi \in U.
\]
Furthermore, there is a constant $C_s=C_s(\varepsilon,C)$, depending on
$\varepsilon$ and $C$ so that for any $f \in \Hs^{s+1}(\R^n,\R)$ and
$\varphi \in U$
\begin{equation}\label{poincare}
\hnorm{f \circ  \varphi - f \circ \varphi_\bullet}{s} \leq C_s \hnorm{f}{s+1} \hnorm{\varphi - \varphi_\bullet}{s}.
\end{equation}
\end{Lemma}

\begin{proof}[Proof of Lemma \ref{lem_poincare}]
The first statement follows from Corollary \ref{cor_openness} $(iii)$. With regard to the second part note that by Lemma \ref{continuous_composition} it suffices to prove estimate (\ref{poincare}) for $f \in \tspace$ as
$\tspace$ is dense in $\Hs^{s+1}(\R^n,\R)$ by Lemma \ref{lemma_density}. Introduce $\delta \varphi(x)=\varphi(x) - \varphi_\bullet(x)$ and note that $\varphi_\bullet + t\delta \varphi$ is in $U$ for any
$0 \leq t \leq 1$ as $U$ is assumed to be convex. By Proposition \ref{prop_sobolev_imbedding}, $\varphi \in \mbox{Diff}_+^1(\R^n)$. For any $x \in \R^n$ consider the $C^1$-curve, 
\[
[0,1] \to \R^n, \quad  t \mapsto f \circ \big(\varphi_\bullet + t \delta \varphi\big)(x).
\]
Clearly, for any $x \in \R^n$,
\begin{eqnarray}
\nonumber
f \circ \varphi(x) - f \circ \varphi_\bullet(x) &=& \int_0^1 \frac{d}{dt} \left( f \circ \big(\varphi_\bullet +t \delta \varphi\big)(x)\right) \;dt \\
\label{fundamental_lemma}
&=& \int_0^1 \left( d_{(\varphi_\bullet + t \delta \varphi)(x)} f \right) \cdot \delta \varphi(x) \;dt.
\end{eqnarray}
By Lemma \ref{continuous_composition},
\[
 t \mapsto d_{\varphi_\bullet + t \delta \varphi}f \cdot \delta \varphi = df \circ (\varphi_\bullet + t \delta \varphi) \cdot \delta \varphi
\]
 is a continuous path in $\Hs^s$, hence it is Riemann integrable in $\Hs^s$ and we have that equality \eqref{fundamental_lemma} is valid in $H^s$. Hence,
\[
 \hnorm{f \circ \varphi - f \circ \varphi_\bullet}{s} \leq \int_0^1 \hnorm{
   d_{\varphi_\bullet + t \delta \varphi}f  \cdot \delta \varphi}{s} dt.
\]
Estimate (\ref{poincare}) then follows
using Lemma \ref{lem_imbedding} and Lemma \ref{continuous_composition}.
\end{proof}

\begin{Lemma}
\label{nu}
Let $s$ be an integer satisfying $s > n/2+1$. To shorten notation, for
the course of this lemma and its proof, we write again $\Hs^s$
instead of $\Hs^s(\R^n,\R)$. Then for any $k \geq 1$, the map $\nu_k$ given by
\begin{eqnarray*}
 \nu_k: \Ds^s &\to& L(H^{s+1};L^{k-1}_{sym}(H^s;H^s))\\
        \varphi & \mapsto & \left[(h,h_1,\ldots,h_{k-1}) \mapsto (h \circ \varphi)
        \cdot \prod_{i=1}^{k-1} h_i \right]
\end{eqnarray*}
is continuous.
\end{Lemma}

\begin{Rem}\label{canonical_identification}
Note that $L\left(H^{s+1};L_{sym}^{k-1}(H^s;H^s)\right)$ isometrically embeds into $L_{sym}^k(H^{s+1} \times H^s;H^s)$ in a canonical way.
\end{Rem}

\begin{proof}[Proof of Lemma \ref{nu}]
For any $h \in H^{s+1}$, $(h_i)_{1\leq i \leq k-1} \subseteq \Hs^s$ and $\varphi, \varphi_\bullet
\in \Ds^s$, we have in view of Lemma \ref{lem_imbedding},
\begin{eqnarray*}
 \hnorm{ (h \circ \varphi) \cdot \prod_{i=1}^{k-1} h_i  -  (h \circ \varphi_\bullet) \cdot \prod_{i=1}^{k-1} h_i}{s}
&\leq& K^{k-1} \hnorm{h \circ \varphi - h \circ \varphi_\bullet}{s} \cdot \prod_{i=1}^{k-1} \hnorm{h_i}{s}.
\end{eqnarray*}
By Lemma \ref{lem_poincare}, there exists $C_s>0$ so that for $\varphi$ in a sufficiently small neighborhood of $\varphi_\bullet$,
\[
 \|h \circ \varphi - h \circ \varphi_\bullet\|_s \leq C_s \|\varphi - \varphi_\bullet\|_s \; \|h\|_{s+1}.
\]
This shows the claimed continuity.
\end{proof}

\begin{proof}[Proof of Proposition \ref{prop1}]
To keep notation as simple as possible we present the proof in the case where $d=n$. The case $r=0$ is treated in Lemma \ref{continuous_composition}, hence it remains to consider the case $r \geq 1$. We want to apply the converse of Taylor's theorem with $U=H^{s+r} \times \Ds^s$, viewed as subset of $X:=H^{s+r} \times H^s$ and $Y:=H^s$. Let $u,\delta u \in \Hs^{s+r}$ and $\varphi \in \Ds^s$, $\delta \varphi
\in \Hs^s$ be given. By Proposition \ref{prop_sobolev_imbedding}, $u, \delta u \in C^r(\R^n,\R^n)$. Hence by Taylor's theorem, for any $x \in \R^n$, $u(\varphi(x)+\delta \varphi(x))$ is given by
\[
u\big(\varphi(x)\big) + \sum_{k=1}^r \sum_{|\alpha|=k}
\frac{1}{\alpha!} \Big(\partial^{\alpha} u \Big)\big(\varphi(x)\big) \cdot
\delta \varphi(x)^{\alpha} + R_1(u,\varphi,\delta \varphi)(x)
\]
where $\delta \varphi(x)^{\alpha}=\delta \varphi_1(x)^{\alpha_1} \cdots \delta
\varphi_n(x)^{\alpha_n}$ and $R_1(u,\varphi,\delta \varphi)(x)$ is defined by
\[
 \sum_{|\alpha|=r} \left\{
 \frac{r}{\alpha!} \int_0^1 (1-t)^{r-1} \Big(\big(
 \partial^{\alpha} u \big) \big( \varphi(x) + t \delta \varphi(x) \big) - \partial^{\alpha}
 u \big(\varphi(x) \big) \Big) \cdot \delta \varphi(x)^\alpha dt \right\}.
\]
Similarly, $ \delta u(\varphi(x) + \delta \varphi(x))$ is given by
\[
 \delta u\big(\varphi(x)\big) +
 \sum_{k=1}^{r-1} \sum_{|\alpha|=k} \frac{1}{\alpha!}
 \Big(\partial^{\alpha}\delta u \Big) \big( \varphi(x)\big) \cdot \delta
 \varphi(x)^{\alpha} + R_2(\delta u,\varphi,\delta \varphi)(x)
\]
with $R_2(\delta u,\varphi,\delta \varphi)(x)$ defined by
\[
 \sum_{|\alpha|=r} \left\{
 \frac{r}{\alpha!} \int_0^1 (1-t)^{r-1}
 \big(\partial^{\alpha} \delta u \big) \big( \varphi(x) + t \delta \varphi(x) \big) \cdot \delta \varphi(x)^\alpha dt \right\}.
\]
Note that for any $x \in \R^n$ the integrals appearing in the definition of the remainder terms $R_1$ and $R_2$ are
well-defined as Riemann integrals. Indeed, as $u, \delta u \in C^r(\R^n,\R^n)$ we see that for any $x \in \R^n$ 
these integrands are continuous functions of $t \in [0,1]$. By Lemma \ref{continuous_composition} (continuity of
composition) and Lemma \ref{lem_imbedding} (continuity of product), the integrands appearing in
the remainder terms $R_1$ and $R_2$ can be viewed as continuous curves in $\Hs^s$, 
parametrized by $t$ and hence are Riemann integrable in $H^s$. Hence the pointwise
integrals are functions in  $\Hs^s$. Furthermore, when viewed as $\Hs^s$-valued curves, the
integrands depend continuously on the parameters 
$(u,\varphi,\delta u,\delta \varphi) \in\Hs^{s+r}\times\Ds^s\times\Hs^{s+r}\times\Hs^s$ by
Lemma \ref{lem_imbedding} and Lemma \ref{continuous_composition}.

In the following we denote by $B_{\varepsilon}^{s+r}(u_\bullet)$ the ball in $H^{s+r}$ of radius $\varepsilon$, 
centered at $u_\bullet \in H^{s+r}$,
\[
B_{\varepsilon}^{s+r}(u_\bullet) = \{ u \in H^{s+r} \ | \ \hnorm{u-u_\bullet}{s+r} < \varepsilon \}.
\]
For $(u_\bullet,\varphi_\bullet) \in \Hs^{s+r} \times \Ds^s$, set $U_1=B_{\varepsilon}^{s+r}(u_\bullet)
\subseteq \Hs^{s+r}$ and $U_2=B_{\varepsilon}^s(\varphi_\bullet-\id) \subseteq \Hs^s$,
where we choose $\varepsilon$ small enough to ensure that $\id + U_2 \subseteq \Ds^s$.
Furthermore, define the subset $V \subseteq \Hs^{s+r} \times \Ds^s \times \Hs^{s+r} \times \Hs^s$ by
\[
V\!=\!\{(u,\varphi,\delta u,\delta \varphi) \in \Hs^{s+r} \times \Ds^s \times \Hs^{s+r} \times \Hs^s \,|
\, (u+\delta u,\varphi + \delta \varphi) \in U_1 \times (\id+U_2)\}.
\]
In view of the considerations above, we get for $(u,\varphi,\delta u, \delta \varphi) \in V$  the following
identity in $\Hs^s$
\[
 (u + \delta u) \circ (\varphi + \delta \varphi) = u \circ \varphi + \sum_{k=1}^r \frac{\eta_k(u,\varphi)}{k!} 
(\delta u,\delta \varphi)^k + R(u,\varphi,\delta u,\delta \varphi)
\]
where $(\delta u, \delta \varphi)^k$ stands for $\big( (\delta u,\delta
\varphi), \ldots, (\delta u,\delta \varphi) \big)$ and  for any $1 \leq k \leq r$, $
 \eta_k(u,\varphi)$ is an element in $L^k_{sym}(\Hs^{s+r} \times \Hs^s;\Hs^s)$, given by
\[
 \eta_k(u,\varphi) (\delta u, \delta \varphi)^k =  \sum_{|\alpha|=k} \frac{k!}{\alpha
  !} ( \partial^{\alpha} u ) \circ \varphi  \cdot \delta \varphi^{\alpha}
+  \sum_{|\alpha|=k-1} \frac{k!}{\alpha!} (\partial^{\alpha} \delta u) \circ \varphi \cdot \delta \varphi^{\alpha}.
\]
The remainder term $R(u,\varphi,\delta u,\delta \varphi)$ is given by
\[
 R(u,\varphi,\delta u,\delta \varphi) =
 R_1(u,\varphi,\delta \varphi) + R_2(\delta u,\varphi,\delta \varphi).
\]
Lemma \ref{rho} and Lemma \ref{nu} together with Remark \ref{canonical_identification} show that for any $k=1,\ldots,r$,
\[
 \eta_k:H^{s+r} \times \Ds^s \to L^k_{sym}(H^{s+r}\times H^s;H^s), \quad (u,\varphi) \mapsto \eta_k(u,\varphi)
\]
is continuous. Moreover, by Lemma \ref{continuous_composition} and Lemma \ref{lem_imbedding},
\[
 \frac{\|R_1(u,\varphi,\delta \varphi)\|_s}{(\|\delta u\|_{s+r} + \|\delta \varphi\|_s)^r} \leq \sum_{|\alpha|=r} \frac{1}{\alpha!} \sup_{0 \leq t \leq 1} \|\partial^\alpha u \circ (\varphi + t \delta \varphi)-\partial^\alpha u \circ \varphi\|_s \to 0
\]
and
\[
 \frac{\|R_2(\delta u,\varphi,\delta \varphi)\|_s}{(\|\delta u\|_{s+r} + \|\delta \varphi\|_s)^r} \leq \sum_{|\alpha|=r} \frac{1}{\alpha!} \sup_{0 \leq t \leq 1} \|(\partial^\alpha \delta u)\circ (\varphi+ t \delta \varphi) \|_s \to 0.
\]
as $\|\delta \varphi\|_s+\|\delta u\|_{s+r} \to 0$. By Theorem \ref{thm_converse_taylor}, it then follows that $\mu$ is a $C^r$ map.
\end{proof}

\noindent Proposition \ref{prop1} together with the implicit function theorem can be used to prove
the following result on the inverse map.
\begin{Prop}\label{prop2}
For any $r \in \Z$ and any integer $s$ with $s > n/2+1$
\begin{equation}
{\tt inv} : \Ds^{s+r} \to \Ds^s, \quad \varphi \mapsto \varphi^{-1}
\end{equation}
is a $C^r$-map.
\end{Prop}
\begin{proof}
The case $r=0$ has been established in Lemma \ref{continuous_inverse}. In particular we know that for any 
$\varphi \in \Ds^s$, its inverse $\varphi^{-1}$ is again in $\Ds^s$. So let $r
\geq 1$. By Proposition \ref{prop1},
\[
 \mu : \Ds^{s+r} \times \Ds^s \to \Ds^s, \quad (\varphi,\psi) \mapsto \varphi \circ \psi
\]
is a $C^r$-map. For any $\varphi \in \Ds^{s+r}$, consider the differential of $\psi \mapsto \mu(\varphi,\psi)$ at
$\psi=\varphi^{-1}$
\[
\left. d_{\psi}\mu(\varphi,\cdot) \right|_{\psi=\varphi^{-1}}: \Hs^s \to
\Hs^s , \quad \delta \psi \mapsto
d \varphi \circ \varphi^{-1} \cdot \delta \psi.
\]
As $r \geq 1$, we get that $d\varphi, d\varphi \circ \varphi^{-1}\in H^s(\R^n,\R^{n \times n})$. In fact, 
$\left. d_\psi \mu(\varphi,\cdot) \right|_{\psi=\varphi^{-1}}$ is a linear isomorphism on $H^s$ whose inverse is given by
$\delta \psi \mapsto (d\varphi)^{-1} \circ \varphi^{-1} \cdot \delta \psi$. Note that by Lemma 
\ref{continuous_composition}, $(d\varphi)^{-1} \circ \varphi^{-1} \in H^s(\R^n,\R^{n \times n})$ and by Lemma 
\ref{lem_imbedding}, $\delta \psi \mapsto (d\varphi)^{-1} \circ \varphi^{-1} \cdot \delta \psi$ is a bounded linear
map $H^s \to H^s$. Furthermore the equation
\[
\mu(\varphi,\psi)=\id
\]
has the unique solution $\psi=\varphi^{-1} \in \Ds^s$. Hence by the
implicit function theorem (see e.g. \cite{Lang}), the map ${\tt inv}:\Ds^{s+r} \to
\Ds^s$, $\varphi \mapsto \varphi^{-1}$ is $C^r$.
\end{proof}

\begin{proof}[Proof of Theorem \ref{thm1}]
Theorem \ref{thm1} follows from Proposition \ref{prop1} and Proposition \ref{prop2}.
\end{proof}

\subsection{Sobolev spaces $H^s(U,\R)$}\label{sobolev_spaces_on_open_sets}
In Section \ref{sec:diff_structure} we need a version of Proposition \ref{prop1} involving the Sobolev spaces
$H^s(U,\R)$ where $U \subseteq \R^n$ is an open nonempty subset with Lipschitz boundary. It means that locally,
the boundary $\partial U$ can be represented as the graph of a Lipschitz function -- see Definition 3.4.2 in 
\cite{Marti}\footnote{cf. \S4.5 in \cite{adams}}. Let $s\in\Z$.
By definition, $H^s(U,\R)$ is the Hilbert space of elements $f$ in $L^2(U,\R)$, having the property that
their distributional derivatives $\partial^\alpha f$ up to order $|\alpha|\leq s$ are $L^2$-integrable on $U$,
endowed with the norm $\|f\|_s$ where 
$\|f\|_s={\langle f,f \rangle}_s^{1/2}$ and ${\langle \cdot,\cdot \rangle}_s$ denotes the inner product defined for
$f,g \in H^s(U,\R)$ by
\[
{\langle f,g \rangle}_s= \sum_{|\alpha| \leq s} \int_U \partial^\alpha f(x) \partial^\alpha g(x) dx.
\]
Further we introduce $H^s(U,\R^m):=H^s(U,\R)^m$. The spaces $H^s(U,\R)$ and $H^s(\R^n,\R)$ are closely related. 
Recall that a function $f:\overline U \to \R$ is said to be $C^r$-differentiable, $r \geq 1$, if there exists an open
neighborhood $V$ of $\overline U$ in $\R^n$ and a $C^r$-function $g:V \to \R$ so that
$f=\left. g \right|_{\overline U}$. We denote by $C^r(\overline U,\R)$ the space of $C^r$-differentiable functions
$f:\overline U \to \R$ and by $C_0^r(\overline U,\R)$ the subspace of $C^r(\overline U,\R)$ consisting of functions
$f:\overline U \to \R$, vanishing at $\infty$, i.e. having the property that for any $\varepsilon >0$,
there exists $M \equiv M_\varepsilon >0$ so that 
\[
 \sup_{x \in \overline U, |x| \geq M} \sup_{|\alpha| \leq r} |\partial^\alpha f(x)| < \varepsilon.
\]
Furthermore, we denote by $C_b^r(\overline{U},\R)$ the subspace of $C^r$-differentiable functions
$f:\overline{U} \to \R$ so that $f$ and all its derivatives up to order $r$ are bounded,
\[
\sup_{x \in \overline{U}} \sup_{|\alpha| \leq r} \big| \partial^\alpha f(x)\big| < \infty.
\]
In a similar fashion one defines $C^\infty(\overline{U},\R)$, $C_0^\infty(\overline{U},\R)$, and
$C_b^\infty(\overline{U},\R)$ and corresponding spaces of vector valued functions $f:\overline{U} \to \R^m$.\\ \\
\noindent
The following result describes how $H^s(U,\R)$ and $H^s(\R^n,\R)$ are related -- see e.g. \cite{Marti}, 
Theorem 3.4.5, Theorem 5.3.1, and Theorem 6.1.1, for these well known results.

\begin{Prop}\label{prop_extension}
Assume that the open set $U \subseteq \R^n$ has a Lipschitz boundary and $s\in\Z$. Then the following statements hold.
\begin{itemize}
\item[(i)] $\big\{ \left. f \right|_U \,\big|\, f \in C_c^\infty(\R^n,\R) \big\}$ is dense in $H^s(U,\R)$.
\item[(ii)] The restriction operator, $H^s(\R^n,\R) \to H^s(U,\R)$, $f \mapsto \left. f \right|_U$, 
is continuous with norm $\leq 1$\footnote{This statement holds for any open set $U \subseteq \R^n$ with $\partial U$
not necessarily Lipschitz.}. Moreover, there is a bounded linear operator $E:H^s(U,\R) \to H^s(\R^n,\R)$, so that
$f=\left. (Ef) \right|_U$ for any $f$ in $H^s(U,\R)$. $E$ is referred to as extension operator.
\item[(iii)] For any integers $s,r \in \mathbb Z_{\geq 0}$ with $s > n/2$,
\[
H^{s+r}(U,\R) \hookrightarrow C_0^r(\overline U,\R)
\]
and the embedding is a bounded linear operator.
\end{itemize}
\end{Prop}
\noindent The following result is needed for the proof of Lemma \ref{continuous_multilinear} below. As usual,
we denote by $L^q(U,\R)$ the Banach space of $L^q$-integrable functions $f:U \to \R$. For a proof of the proposition
see e.g. Theorem 5.4 in \cite{adams}.

\begin{Prop}\label{imbedding_theorems}
Assume that the open set $U \subseteq \R^n$ has a Lipschitz boundary and let $s \in \Z$.
Then the following statements hold:
\begin{itemize}
\item[(i)] If $0 \leq s < n/2$, then for any $2 \leq q \leq \frac{2n}{n-2s}$,
\[
H^s(U,\R) \hookrightarrow L^q(U,\R)
\]
is continuous.
\item[(ii)] If $s=n/2$, then for any $2 \leq q < \infty$,
\[
H^s(U,\R) \hookrightarrow L^q(U,\R)
\]
is continuous.
\end{itemize}
\end{Prop}
\noindent Combining Proposition \ref{prop_extension} and Lemma \ref{lem_imbedding} one obtains the following
\begin{Lemma}\label{lem_imbedding_general}
Assume that the open set $U \subseteq \R^n$ has a Lipschitz boundary. Let $s,s'$ be integers with $s>n/2$ and
$0 \leq s' \leq s$. Then there exists $K >0$ so that for any $f \in H^s(U,\R)$ and $g \in H^{s'}(U,\R)$,
the product $f \cdot g$ is in $H^{s'}(U,\R)$ and
\[
 \|f \cdot g\|_{s'} \leq K \|f\|_s \|g\|_{s'}.
\]
In particular, $H^s(U,\R)$ is an algebra.
\end{Lemma}

\noindent We will also need the following variant of Lemma \ref{lem_imbedding_general}.

\begin{Lemma}\label{continuous_multilinear}
Let $U \subseteq \R^n$ be a non-empty, open, bounded set with Lipschitz boundary and let $s > n/2$, $s\in\Z$.
Then for any $r \geq 2$ and any $k=(k_1,\ldots,k_r) \in \mathbb Z_{\geq 0}^r$ with $\sum_{j=1}^r k_j \leq s$,
the $r$-linear map,
\begin{equation}\label{multilinear_multiplication}
H^{s-k_1}(U,\R) \times \cdots \times H^{s-k_r}(U,\R) \to L^2(U,\R),\quad (f_1,\ldots,f_r) \mapsto f_1 \cdots f_r
\end{equation}
is well-defined and continuous.
\end{Lemma}

\begin{proof}
First note that the map
\[
C_b^0(U,\R) \times L^2(U,\R) \to L^2(U,\R),\quad (f_1,f_2) \mapsto f_1 \cdot f_2
\]
is continuous. Combining this with Proposition \ref{prop_extension} $(iii)$, one sees that it remains to prove that the map \eqref{multilinear_multiplication} is well-defined and continuous for any $r \geq 2$ and any $k=(k_1,\ldots,k_r) \in \mathbb Z_{\geq 0}^r$ with $\sum_{j=1}^r k_j \leq s$ and 
\begin{equation}\label{reduced_case}
s-k_j-\frac{n}{2} \leq 0,\quad 1 \leq j \leq r.
\end{equation}
In what follows we assume that \eqref{reduced_case} holds. Divide the set $I:=\{ j \in \N \;|\; 1 \leq j \leq r\}$ into two subsets, $I=I_< \cup I_0$,
\[
 I_< := \{ j \in I \;|\; s-k_j - \frac{n}{2} < 0\}
\]
and
\[
 I_0 :=\{ j \in I \;|\; s-k_j-\frac{n}{2} = 0\}.
\]
By Proposition \ref{imbedding_theorems}, for any $j \in I_<$,
\begin{equation}\label{imbedding_a}
H^{s-k_j}(U,\R) \hookrightarrow L^{q_j}(U,\R),\quad q_j = \frac{2n}{n-2(s-k_j)}
\end{equation}
and for any $j \in I_0$,
\begin{equation}\label{imbedding_b}
H^{s-k_j}(U,\R) \hookrightarrow L^{q_j}(U,\R), \quad \forall q_j \geq 2.
\end{equation}
We choose $q_j$ as follows: If $I=I_0$ then choose $q_j \geq 2, j \in I$, so that
\begin{equation}\label{reciproc_exponents}
\frac{1}{q_1} + \cdots + \frac{1}{q_r} < \frac{1}{2}.
\end{equation}
If $I_< \neq \emptyset$ one has by \eqref{imbedding_a}
\[
\sum_{j \in I_<} \frac{1}{q_j} = \sum_{j \in I_<} \left(\frac{1}{2}-\frac{s-k_j}{n}\right) \leq \frac{r}{2}-\frac{rs}{n}+\frac{1}{n} \sum_{j=1}^r k_j.
\]
As by assumption, $\sum_{j=1}^r k_j \leq s$ and $s >n/2$ one gets
\[
\sum_{j \in I_<} \frac{1}{q_j} \leq \frac{1}{2}+(r-1)\frac{1}{2} - (r-1)\frac{s}{n} = \frac{1}{2} + \frac{r-1}{n}\left(\frac{n}{2}-s\right) < \frac{1}{2}.
\]
Hence by choosing for any $j \in I_0$ $q_j \geq 2$ large enough we can ensure that also in the case where $I_< \neq \emptyset$ \eqref{reciproc_exponents} holds. Altogether we have shown that there exist $q_j \geq 2, j \in I$ so that \eqref{imbedding_a},\eqref{imbedding_b}, and
\begin{equation}\label{reciproc_condition}
\frac{1}{q_1} + \cdots \frac{1}{q_r} \leq \frac{1}{2} 
\end{equation}
hold. Thus $q=\left(\frac{1}{q_1} + \cdots \frac{1}{q_r}\right)^{-1} \geq 2$.
It follows from the generalized H\"older inequality that the $r$-linear map
\[
 L^{q_1}(U,\R) \times \cdots \times L^{q_r}(U,\R) \to L^q(U,\R),\quad (f_1,\ldots,f_r) \mapsto f_1 \cdots f_r
\]
is continuous. As $U \subseteq \R^n$ is bounded and $q \geq 2$,
\[
 L^q(U,\R) \hookrightarrow L^2(U,\R)
\]
and the inclusion is continuous. Hence, the $r$-linear map
\[
 L^{q_1}(U,\R) \times \cdots \times L^{q_r}(U,\R) \to L^2(U,\R),\; (f_1,\ldots,f_r) \mapsto f_1 \cdots f_r
\]
is continuous as well. This together with the continuity of the embeddings \eqref{imbedding_a} and 
\eqref{imbedding_b} implies that the map \eqref{multilinear_multiplication} is well-defined and continuous for any 
$k \in \mathbb Z_{\geq 0}^r$ satisfying $\sum_{j=1}^r k_j \leq s$ and \eqref{reduced_case}.
\end{proof}

\noindent Let $U \subseteq \R^n$ be a bounded open set with Lipschitz boundary and $s>n/2+1$, $s\in\Z$.
Denote by $\Ds^s(U,\R^n)$ the subset of $H^s(U,\R^n)\big(\subseteq C^1(\overline U,\R^n)\big)$ consisting of orientation
preserving local diffeomorphisms  $\varphi : U\to\R^n$ that extend to bijective maps $\varphi : {\overline U}\to
\varphi(\overline U)\subseteq\R^n$ and such that 
\begin{equation}\label{eq:det>0}
\inf\limits_{x \in{\overline U}}\det(d_x\varphi)>0\,.
\end{equation}
More precisely,
\[
\Ds^s(U,\R^n):= \big\{ \varphi \in H^s(U,\R^n)\;\big|\;\varphi : {\overline U}\to\R^n\,\,\mbox{ is 1-1}\,\,\mbox{and}\,
\inf_{x \in{\overline U}} \det(d_x\varphi) > 0 \big\}.
\] 
\begin{Lemma}\label{lem:openness}
$\Ds^s(U,\R^n)$ is an open subset in $H^s(U,\R^n)$.
\end{Lemma}

\begin{proof}
In view of Proposition \ref{prop_sobolev_imbedding} and Proposition \ref{prop_extension} $(ii)$,
$\Ds^s(U,\R^n)$ can be continuously embedded into $C^1(\R^n,\R^n)$,
\[
\Ds^s(U,\R^n)\subseteq C^1(\R^n,\R^n)\,.
\]
Take an arbitrary element $\varphi\in\Ds^s(U,\R^n)$. For $\varepsilon>0$ denote by $B_\varepsilon$ the open 
$\varepsilon$-ball  centered at zero in $H^s(U,\R^n)$. As $\overline U$ is compact one gets from
\eqref{eq:det>0} and the inverse function theorem that there exists $\varepsilon>0$ such that for any
$f\in B_\varepsilon$, the map
\begin{equation}
\psi : {\overline U}\to\R^n,\quad\psi:=\varphi+f
\end{equation}
is a {\em local diffeomorphism}. 
Strengthening these arguments one sees that there exist $\varepsilon>0$ and $\delta>0$ such that
for any $f\in B_\varepsilon$ and $\forall\,x, y\in{\overline U}$, $x\ne y$,
\begin{equation}\label{eq:relation1}
|x-y|<\delta\quad\Longrightarrow\quad\psi(x)\ne\psi(y)\,.
\end{equation}
In fact, following the arguments of the proof of the inverse function theorem one sees that
for any $x\in{\overline U}$ there exist $\varepsilon_x>0$ and an open neighborhood $U_x$ of $x$ in $\R^n$
such that for any $f\in B_{\varepsilon_x}$ the map
\[
\psi\big|_{U_x} : U_x\to\R^n 
\]
is injective. Using the compactness of $\overline U$ we find
$x_1,...,x_n\in{\overline U}$ such that $\cup_{j=1}^nU_{x_j}\supseteq{\overline U}$.
Take, $\varepsilon:=\min\limits_{1\le j\le n}\varepsilon_{x_j}$.
Then, assuming that \eqref{eq:relation1} does {\em not} hold, we can construct two sequences 
$(p_j)_{1\le j\le n}$ and $(q_j)_{1\le j\le n}$ of points in $\overline U$ and
$(f_j)_{1\le j\le n}\subseteq B_\varepsilon$ such that
\begin{equation}\label{eq:pq}
0<|p_j-q_j|<1/j\quad\quad\mbox{and}\quad\quad \psi_j(p_j)=\psi_j(q_j)
\end{equation}
where $\psi_j:=\varphi+f_j$. By the compactness of $\overline U$, we can assume that 
there exists $p\in{\overline U}$ such that $\lim\limits_{j\to\infty}p_j=\lim\limits_{j\to\infty}q_j=p$.
Taking $j\ge 1$ sufficiently large we obtain that $p_j, q_j\in U_p$, and therefore
$\psi_j(p_j)\ne\psi_j(q_j)$. As this contradicts \eqref{eq:pq}, we see that implication
\eqref{eq:relation1} holds.

Further, we argue as follows. Consider the sets
\[
\Delta_\delta:=\{x,y\in{\overline U}\,\big|\,|x-y|<\delta\}
\]
and
\[
{\cal K}_\delta:={\overline U}\times{\overline U}\setminus\Delta_\delta\,.
\]
As ${\cal K}_\delta$ is compact and $\varphi : {\overline U}\to\R^n$ is injective,
\[
m:=\min_{(x,y)\in{\cal K}_\delta}|\varphi(x)-\varphi(y)|>0\,.
\]
This implies that $\forall\,x, y\in{\overline U}$, $x\ne y$,
\begin{equation}\label{eq:relation2}
|\varphi(x)-\varphi(y)|< m\quad\Longrightarrow\quad |x-y|<\delta\,.
\end{equation}
By taking $\varepsilon>0$ smaller if necessary, we can ensure that for any $f\in B_\varepsilon$,
\begin{equation}\label{eq:m-close}
\|\psi-\varphi\|_{C^0}<m/2\,.
\end{equation}
Finally, assume that there exists $f\in B_\varepsilon$ so that the map $\psi : {\overline U}\to\R^n$, 
$\psi=\varphi+f$, is {\em not} injective.
Then there exist $x,y\in{\overline U}$, $x\ne y$, so that
\[
\psi(x)=\psi(y)\,.
\]
This together with \eqref{eq:m-close} implies that 
\[
|\varphi(x)-\varphi(y)|<m\,.
\]
In view \eqref{eq:relation1} and \eqref{eq:relation2} we get that $\psi(x)\ne\psi(y)$.
This contradiction shows that $\psi$ is injective.
\end{proof}

\begin{Prop}\label{prop_composition_general}
Let $U$ be an open bounded subset in $\R^n$ with Lipschitz boundary. Then for any $d,r,s \in \mathbb Z_{\geq 0}$ with $s>n/2+1$
\[
 \mu:H^{s+r}(\R^n,\R^d) \times \Ds^s(U,\R^n) \to H^s(U,\R^d), \quad (f,\varphi) \mapsto f \circ \varphi
\]
is a $C^r$-map.
\end{Prop}

\noindent In view of Proposition \ref{prop_extension}, the proof of Proposition \ref{prop1} can be easily adapted to show Proposition \ref{prop_composition_general}. We leave the details to the reader.

\begin{Coro}\label{coro_right_translation}
Under the assumption of Proposition \ref{prop_composition_general}, the right translation by an arbitrary element $\varphi \in \Ds^s(U,\R^n)$, 
\[
 R_\varphi:H^s(\R^n,\R^d) \to H^s(U,\R^d), \quad f \mapsto f \circ \varphi
\]
is a $C^\infty$-map.
\end{Coro} 

\begin{proof}
By Proposition \ref{prop_composition_general}, $R_\varphi$ is well-defined and continuous. As $R_\varphi$ is a linear operator it then follows that $R_\varphi$ is a $C^\infty$-map. 
\end{proof}

\noindent As an application of Corollary \ref{coro_right_translation} we get the following result. 

\begin{Coro}\label{th_transformation}
Let $U,\,V \subseteq \R^n$ be open and bounded sets with Lipschitz boundary and let $\varphi:U \to V$ be a
$C^\infty$-diffeomorphism with $\varphi \in C^\infty(\overline{U},\R^n)$ and
$\varphi^{-1} \in C^\infty(\overline{V},\R^n)$.  Then for any given $s \geq 0$, $s\in\Z$,
the right translation by $\varphi$,
\[
R_\varphi:H^s(V,\R) \to H^s(U,\R), \quad f \mapsto f \circ \varphi
\]
is a continuous linear isomorphism.
\end{Coro}

\noindent Finally, we include the following result concerning the left translation.
Recall that for any given open subset $U \subseteq \R^n$, we denote by $C_b^\infty(\overline{U},\R^n)$ the subspace of $C^\infty(\overline{U},\R^n)$ consisting of all elements $f \in C^\infty(\overline{U},\R^n)$ so that $f$ and all its derivatives are bounded on $\overline{U}$.

\begin{Prop}\label{prop_left_translation}
Let $m,d,s \in \mathbb Z_{\geq 0}$ with $m,d \geq 1$ and $s > n/2$ and let $U$ be an open bounded subset of $\R^n$ with Lipschitz boundary. Then for any $g$ in $C_b^\infty(\R^m,\R^d)$, the left translation by $g$,
\[
 L_g: H^s(U,\R^m) \to H^s(U,\R^d),\quad f \mapsto g \circ f
\]
is a $C^\infty$-map.
\end{Prop}

\begin{proof}
We begin by showing that $L_g$ is continuous. Note that by Proposition \ref{prop_extension}, the extension operator 
\[
 E:H^s(U,\R^m) \to H^s(\R^n,\R^m)
\]
is a bounded linear operator, $\|E\| < \infty$. By Proposition \ref{prop_sobolev_imbedding} the embedding $H^s(\R^n,\R^m) \hookrightarrow C_0^0(\R^n,\R^m)$ is continuous and for any $f \in H^s(U,\R^m)$,
\begin{equation}\label{imbedding_ineq}
\|E(f)\|_{C^0} \leq K_{s,0} \|E(f)\|_s \leq K_{s,0} \|E\| \; \|f\|_s.
\end{equation}
As $g$ is continuous and bounded, $g \circ E(f)$ is in $C_b^0(\R^n,\R^d)$ and hence $g \circ f$ in $C_b^0(U,\R^d)$. Furthermore
\[
 C_b^0(\R^n,\R^m) \to C_b^0(\R^n,\R^d),\quad h \mapsto g \circ h
\]
is continuous. More precisely, for $h_1,h_2 \in C_b^0(\R^n,\R^m)$
\begin{equation}\label{lipschitz_ineq}
\|g \circ h_1 - g \circ h_2\|_{C^0} \leq L \|h_1 - h_2\|_{C^0}
\end{equation}
where $L:=\sup_{x \in \R^m} |d_x g| < \infty$. As for any $f_1,f_2 \in H^s(U,\R^m)$,
\[
 g \circ f_1 - g \circ f_2 = \left. \left(g \circ E(f_1) - g \circ E(f_2)\right) \right|_U
\] 
it follows from the boundedness of the restriction map, \eqref{imbedding_ineq} and \eqref{lipschitz_ineq}, that
\begin{multline}\label{lipschitz_property}
\|g \circ f_1 - g \circ f_2\|_{C^0(U)} \leq \|g \circ E(f_1) - g \circ E(f_2)\|_{C^0}\\
 \leq L \, \|E(f_1)-E(f_2)\|_{C^0} \leq L K_{s,0} \|E\| \; \|f_1-f_2\|_s.
\end{multline}
In particular, $H^s(U,\R^m) \to C_b^0(U,\R^d),f \mapsto g \circ f$ is Lipschitz continuous.
Take $f \in H^s(U,\R^m)$. By Proposition \ref{prop_extension} $(i)$, there exists a sequence $(f^{(k)})_{k \geq 1}$, $f^{(k)} \in C_c^\infty(\R^n,\R^m)$, such that
\begin{equation}\label{approx_seq}
\left. f^{(k)} \right|_U \to f \quad \mbox{as } k \to \infty
\end{equation}
in $H^s(U,\R^m)$. Using the chain and the Leibniz rule we see that for any $k \geq 1$, $1 \leq i \leq d$, and any multi-index $\alpha \in \mathbb Z_{\geq 0}^n$ with $|\alpha| \leq s$, $\partial^\alpha (g_i \circ f^{(k)})$ is a linear combination of products of the form
\begin{equation}\label{product_form}
\partial^\beta g_i \circ f^{(k)} \cdot \partial^{\gamma_1}f_{j_1}^{(k)} \cdots \partial^{\gamma_r}f_{j_r}^{(k)}
\end{equation}
where $\beta \in \mathbb Z_{\geq 0}^m$ with $|\beta| \leq |\alpha|$, $r \in \mathbb Z_{\geq 0}$ with $r \leq |\alpha|$ and $\gamma_1,\ldots,\gamma_r \in \mathbb Z_{\geq 0}^n$ with $\gamma_1 + \ldots \gamma_r = \alpha$. It follows from \eqref{lipschitz_property} and \eqref{approx_seq} that for any $|\beta| \leq |\alpha|$, and for any $1 \leq i \leq d$,
\begin{equation}\label{composition_convergence}
\left. \partial^\beta g_i \circ f^{(k)} \right|_U \to \partial^\beta g_i \circ f \quad \mbox{in } C_b^0(U,\R)
\end{equation}
as $k \to \infty$. Moreover, by \eqref{approx_seq}, for any $1 \leq p \leq r$,
\begin{equation}\label{factor_convergence}
 \left. \partial^{\gamma_p} f_{j_p}^{(k)} \right|_U \to \partial^{\gamma_p}f_{j_p} \quad \mbox{in } H^{s-|\gamma_p|}(U,\R).
\end{equation}
As $\sum_{j=1}^r |\gamma_j| = |\alpha| \leq s$, we get from \eqref{composition_convergence}, \eqref{factor_convergence}, and Lemma \ref{continuous_multilinear} that
\[
 \left. \left( \partial^\beta g_i \circ f^{(k)} \cdot \partial^{\gamma_1}f_{j_1}^{(k)} \cdots \partial^{\gamma_r} f_{j_r}^{(k)}\right)\right|_U \to \partial^\beta g_i \circ f \cdot \partial^{\gamma_1}f_{j_1} \cdots \partial^{\gamma_r}f_{j_r}
\]
in $L^2(U,\R)$ as $k \to \infty$. In particular, for any test function $\varphi \in C_c^\infty(U)$,
\begin{multline*}
 \lim_{k \to \infty} \int_{\R^n} \left[ \partial^\beta g_i \circ f^{(k)} \cdot \partial^{\gamma_1}f_{j_1}^{(k)} \cdots \partial^{\gamma_r} f_{j_r}^{(k)} \right] \cdot \varphi \;dx\\
= \int_{\R^n} \left[\partial^\beta g_i \circ f \cdot \partial^{\gamma_1} f_{j_1} \cdots \partial^{\gamma_r} f_{j_r}\right] \cdot \varphi \;dx.
\end{multline*}
Furthermore, by \eqref{lipschitz_property},
\begin{eqnarray}
\nonumber
 \langle \partial^\alpha (g_i \circ f),\varphi \rangle &=& (-1)^{|\alpha|} \int_{\R^n} \big(g_i \circ f\big)(x) \partial^\alpha \varphi(x) \;dx\\
\nonumber
&=& \lim_{k \to \infty} (-1)^{|\alpha|} \int_{\R^n} \big(g_i \circ f^{(k)}\big)(x) \partial^\alpha \varphi(x)\;dx\\
\label{test_approx}
&=& \lim_{k \to \infty} \int_{\R^n} \partial^\alpha \big(g_i \circ f^{(k)}\big)(x) \varphi(x) \;dx.
\end{eqnarray}
Combining this with \eqref{product_form} and \eqref{test_approx} we see that for any $\alpha$ in $\mathbb Z_{\geq 0}^n$, $|\alpha| \leq s$, the weak derivative $\partial^\alpha(g_i \circ f)$ is in $L^2(U,\R)$. As $\partial^\alpha(g_i \circ f)$ is a linear combination of terms of the form
\[
 \partial^\beta g_i \circ f \cdot \partial^{\gamma_1}f_{j_1} \cdots \partial^{\gamma_r}f_{j_r} \in L^2(U,\R)
\]
with $\sum_{j=1}^r \gamma_j = \alpha$ it follows from \eqref{lipschitz_property} and Lemma \ref{continuous_multilinear} that the map $H^s(U,\R^m) \to L^2(U,\R)$,
\[
 f \mapsto \partial^\beta g_i \circ f \cdot \partial^{\gamma_1}f_{j_1} \cdots \partial^{\gamma_r}f_{j_r},
\]
is continuous. This shows that
\begin{equation}\label{continuity_of_composition}
H^s(U,\R^m) \to H^s(U,\R^d),\quad f \mapsto g \circ f,
\end{equation}
is continuous.  To see that $L_g$ is $C^r$-smooth for any $r \geq 1$ we again apply Theorem \ref{thm_converse_taylor}. Let $f,\delta f$ be elements in $H^s(U,\R^m)$. Expanding $g$ at $f(x)$, $x \in U$ arbitrary, up to order $r \geq 1$, one gets 
\begin{eqnarray*}
 g\big(f(x)+\delta f(x)\big) &=& g\big(f(x)\big) + \sum_{i=1}^r \sum_{|\alpha|=i} \frac{1}{\alpha!} \big(\partial^\alpha g\big) \big(f(x)\big) \cdot \delta f^\alpha(x)\\
&+& R(f,\delta f)(x) 
\end{eqnarray*}
where $\delta f^\alpha(x)= \prod_{i=1}^m \delta f_i(x)^{\alpha_i}$ and the remainder term $R(f,\delta f)$ is given by
\begin{eqnarray*}
R(f,\delta f)(x) &=& \sum_{|\alpha|=r} \frac{r}{\alpha!} \int_0^1 (1-t)^{r-1} \Big( \big(\partial^\alpha g\big)\big(f(x)+t\delta f(x)\big)\\
&-& \big(\partial^\alpha g\big) \big(f(x)\big) \Big) \delta f^\alpha(x) dt.
\end{eqnarray*}
By \eqref{continuity_of_composition}, for any $\alpha \in \mathbb Z_{\geq 0}^n$,
\begin{equation}\label{continuous_derivative_composition2}
H^s(U,\R^m) \to H^s(U,\R^d),\quad f \mapsto \partial^\alpha g \circ f
\end{equation}
is continuous. In view of Lemma \ref{lem_imbedding_general} (cf. also Lemma \ref{rho}), $\partial^\alpha g \circ f$ can be viewed as an element in $L^{|\alpha|}_{sym}\big(H^s(U,\R),H^s(U,\R^d)\big)$, defined by
\[
 (\delta h_j)_{1\leq j \leq |\alpha|} \mapsto \partial^\alpha g \circ f \cdot \prod_{j=1}^{|\alpha|} \delta h_j
\]
and the map
\[
 H^s(U,\R^m) \to L^{|\alpha|}_{sym}\big( H^s(U,\R),H^s(U,\R^d)\big), \quad f \mapsto \partial^\alpha g \circ f
\]
is continuous. Similarly one sees that $R(f,\delta f)$ is in $H^s(U,\R^d)$ and by Lemma \ref{lem_imbedding_general},
\[
 \frac{\|R(f,\delta f)\|_s}{\|\delta f\|_s^r} \leq K^{r+1} \sum_{|\alpha|=r} \frac{1}{\alpha!} \sup_{0 \leq t \leq 1} \|\partial^\alpha g \circ (f+t \delta f)-\partial^\alpha g \circ f\|_s. 
\]
By the continuity of the map (\ref{continuous_derivative_composition2}) it then follows that
\[
 \lim_{\|\delta f\|_s \to 0} \frac{\|R(f,\delta f)\|_s}{\|\delta f\|_s^r} = 0.
\]
Hence Theorem \ref{thm_converse_taylor} applies and it follows that $L_g$ is $C^r$-smooth for any $r \geq 1$.
\end{proof}

\noindent When applying Proposition \ref{prop_left_translation} we will need the following simple Lemma.

\begin{Lemma}\label{lemma_smooth_extension}
Let $U \subseteq \R^n$ be a bounded domain. If $g \in C^\infty(\overline{U},\R)$ then there exists $\tilde g \in C_c^\infty(\R^n,\R)$ such that $\left. \tilde g \right|_{\overline U} = g$.
\end{Lemma}

\noindent The following result easily follows from Proposition \ref{prop_extension} $(ii)$.

\begin{Lemma}\label{lem:cut-off}
Let $U\subseteq\R^n$ be an open subset in $\R^n$ with Lipschitz boundary and let $s>n/2$. Then for any 
$f\in H^s(U,\R^d)$ and $\varphi\in C^\infty_c(\R^n)$, $\varphi\cdot f \in H^s(U,\R^d)$.
\end{Lemma}

\section{Diffeomorphisms of a closed manifold}\label{Section 3}

In this section we prove Theorem \ref{thm2}. The main results used for the proof -- in addition to the ones of
Proposition \ref{prop_composition_general}, Proposition \ref{prop_left_translation}, and Lemma \ref{lemma_smooth_extension} -- are summarized in Section \ref{subsec3:prelim} and will be proved in Section \ref{sec:diff_structure}.

\subsection{Preliminaries}\label{subsec3:prelim}

Let $M$ be a closed manifold of dimension $n$ and $N$ a manifold of dimension $d$. Further let $s$ be an integer, $s>n/2$. Recall that a continuous map $f:M \to N$ is said to be an element in $H^s(M,N)$ if for any point $x \in M$, there exist a chart $\chi:\U \to U \subseteq \R^n$ of $M$, $x \in \U$, and a chart $\eta:\V \to V \subseteq \R^d$ of $N$, $f(x) \in \V$, such that $f(\U) \subseteq \V$ and
\[
 \eta \circ f \circ \chi^{-1}:U \to V \subseteq \R^d
\]
is an element in $H^s(U,\R^d)$. Note that if $\widetilde \chi:\widetilde \U \to \widetilde U$ and $\widetilde \eta:\widetilde \V \to \widetilde V$ are two other charts such that $x \in \widetilde \U$ and $f(\widetilde \U) \subseteq \widetilde \V$, then $\widetilde \eta \circ f \circ \widetilde \chi^{-1}$ is not necessarily an element in $H^s(\widetilde U,\R^d)$. As an example consider $M=\mathbb T=\R/\mathbb Z$, $N=\R$ and let $f:(-1/2,1/2) \to \R$ be the function
\[
f(x) :=
\left\{
\begin{array}{rl}
x^{2/3},& x \in [0,1/2)\\
(-x)^{2/3},& x \in [-1/2,0)
\end{array}
\right..
\]
Extending $f$ periodically to $\R$ we get a function on $\mathbb T$ that we denote by the same letter. It is not hard to see that $f \in H^1(\mathbb T,\R)$. Now, introduce a new coordinate $y=x^2$ on the open set $(0,1/2) \subseteq \mathbb T$. Then $\tilde{f}(y) := f(x(y))=y^{1/3}$, $y \in (0,1/4)$. We have, $\tilde{f}'(y)=1/(3y^{2/3})$, and hence, $\tilde{f}' \notin L^2((0,1/4),\R)$. This shows that $\tilde{f} \notin H^1((0,1/4),\R)$. 

With this in mind we define

\begin{Def}\label{def_bdd_cover}
An open cover $(\U_i)_{i \in I}$ of $M$ by coordinate charts $\chi_i:\U_i \to U_i \subseteq \R^n$, $i \in I$, is called a \emph{cover of bounded type}, if for any $i,j \in I$ with $\U_i \cap \U_j \neq \emptyset$,
\[
\chi_j \circ \chi_i^{-1} \in C_b^\infty \big(\overline{\chi_i(\U_i \cap \U_j)},\R^n\big).
\]
\end{Def}

\begin{Def}\label{def_fine_cover}
Assume that $\U_I=(\U_i)_{i \in I}$ is a cover of $M$ and $\V_I=(\V_i)_{i \in I}$ is a collection of charts of $N$. The pair $(\U_I,\V_I)$ is said to be a \emph{fine cover} if the following conditions are satisfied:
\begin{itemize}
\item[(C1)] $I$ is finite and for any $i \in I$, $\chi_i:\U_i \to U_i \subseteq \R^n$ and $\eta_i:\V_i \to V_i \subseteq \R^d$ are coordinate charts of $M$ respectively $N$; $U_i$ and $V_i$ are bounded and have a Lipschitz boundary.
\item[(C2)] $\U_I \; [\V_I]$ is a cover of $M \; [\cup_{i \in I} \V_i]$ of bounded type.
\item[(C3)] For any $i,j \in I$, the boundaries of $\chi_i(\U_i \cap \U_j)$ and $\eta_i(\V_i \cap \V_j)$ are piecewise $C^\infty$-smooth, i.e. they are given by a finite (possibly empty) union of transversally intersecting
$C^\infty$-embedded hypersurfaces in $\R^n$ respectively $\R^d$. In particular, $\chi_i(\U_i \cap \U_j)$ and $\eta_i(\V_i \cap \V_j)$ have a Lipschitz boundary. 
\end{itemize}
\end{Def}

\noindent Fine covers $(\U_I,\V_I)$ will be used to construct a $C^\infty$-differentiable structure of $H^s(M,N)$.
To make this construction independent of any choice of metrics on $M$ and $N$, the notion of a fine cover does not involve any metric.

\begin{Def}\label{def_fine_f}
A triple $(\U_I,\V_I,f)$ consisting  of $f \in H^s(M,N)$ with $s > n/2$ and a fine cover $(\U_I,\V_I)$ is said to be
a {\rm fine cover with respect to $f$} if $f(\U_i) \Subset \V_i$ for any $i \in I$, i.e.,
$\overline{f(\U_i)}$ is compact and contained in $\V_i$.
\end{Def}

\begin{Lemma}\label{lem:fine_cover_existence}
Let $f \in H^s(M,N)$ and $s>n/2$. Then there exists a fine cover $(\U_I,\V_I)$ with respect to $f$. 
\end{Lemma}
\begin{proof}
To construct such a fine cover choose a Riemannian metric $g_M$ on $M$, a Riemannian metric $g_N$  on $N$, and
$\rho > 0$, so that $2\rho$ is smaller than the injectivity radius of the compact subset $f(M) \subseteq N$ with
respect to the Riemannian metric $g_N$. Note that $f(M)$ is compact as $M$ is compact and $f$ is continuous.
Furthermore, $f:M \to N$ is uniformly continuous. Hence there exists $r > 0$ with $2r$ smaller than the injectivity
radius of $(M,g_M)$ so that $\mathop{\rm dist}\nolimits_{g_N}(f(x),f(x')) < \rho$ for any $x,x' \in M$ with
$\mathop{\rm dist}\nolimits_{g_M}(x,x')<r$.\footnote{Here $\mathop{\rm dist}\nolimits_{g_M}$ and
$\mathop{\rm dist}\nolimits_{g_N}$  denote the geodesic distances on $(M,g_M)$ and $(N,g_N)$ respectively.}
For any $x \in M$ define
\[
\U_x:=\exp_x(B_r) \quad \mbox{and} \quad U_x:=B_r \subseteq T_xM \cong \R^n
\]
where $B_r$ denotes the open ball in $T_x M$ of radius $r$ with respect to the inner product $g_M(x)$ and
$\exp_x:T_xM \to M$ denotes the Riemannian exponential map at $x$. The map $\chi_x:\U_x \to U_x$ is then defined to be
the restriction of the inverse of $\exp_x{}$ to $\U_x$, which is well defined as $2r$ is smaller than the injectivity 
radius. Hence $\chi_x$ is a chart of $M$. Assume that there exist points $x,x' \in M, x \neq x'$ and
$p \in \partial \U_x \cap \partial \U_{x'}$, so that the boundaries of the geodesic balls $\U_x$ and $\U_{x'}$ do not
intersect transversally at $p$. We claim that in this case $\U_x \cap \U_{x'} = \emptyset$. Indeed, as
$\mathop{\rm dist}\nolimits_{g_M}(x,p)=r$, $\mathop{\rm dist}\nolimits_{g_M}(x',p)=r$, and as $2r$ is smaller
than the injectivity radius of $(M,g_M)$ there exists a minimal geodesic connecting the points $x$ and $x'$.
In view of the assumptions that $x \neq x'$ and $\partial \U_x$ and $\partial \U_{x'}$ do not intersect
transversally in $p$ it then follows that $p$ lies on the above geodesic between $x$ and $x'$ and
$\mathop{\rm dist}\nolimits_{g_M}(x,x')=2r$, hence $\U_x \cap \U_{x'} = \emptyset$. 
Therefore, for any $x,x' \in M$, $x \neq x'$, $\partial \U_x$ and
$\partial \U_{x'}$ either do not intersect at all or intersect transversally. In a similar way we construct charts 
$\eta_{f(x)}:\V_{f(x)} \to V_{f(x)} \subseteq \R^d$, $x \in M$, where now $V_{f(x)} \subseteq T_{f(x)}N \cong \R^d$ is
the open ball of radius $\rho$ in $T_{f(x)}N$ centered at $0$ and
$\eta_{f(x)} = (\left. \exp_{f(x)} \right|_{V_{f(x)}})^{-1}$.
Here $\exp_{f(x)}$ denotes the Riemannian exponential map of $(N,g_N)$ at $f(x)$. As $M$ is compact there exist 
finitely many points $(x_i)_{i \in I} \subseteq M$ so that $\U_I = (\U_i)_{i \in I}$ with $\U_i \equiv \U_{x_i}$ 
covers $M$. By construction $\V_I=(\V_i)_{i \in I}$ with $\V_i=\V_{f(x_i)}$ is then a cover of $f(M)$ and one verifies
that $(\U_I,\V_I,f)$ is a fine cover with respect to $f$.
\end{proof}

\begin{Lemma}\label{lemma_local_pt}
Let $(\U_I,\V_I,h)$ be fine cover with respect to $h\in H^s(M,N)$. Then for any $i \in I$, the map
$h_i:=\eta_i \circ h \circ \chi_i^{-1}:U_i \to V_i \subseteq \R^d$ is in $H^s(U_i,\R^d)$.
\end{Lemma} 

\begin{proof}
By the definition of $H^s(M,N)$ and the compactness of $M$ there exist a finite open cover $(\W_j)_{j \in J}$ of 
$M$ by coordinate charts
\[
\mu_j:\W_j \to W_j \subseteq \R^n
\]
and for any $j \in J$ an open coordinate chart $\nu_j:\Zj \to Z_j \subseteq \R^d$ of $N$ with
$h(\W_j) \Subset \Zj$ and $W_j,Z_j$ bounded so that for any $j \in J$
\[
\nu_j \circ h \circ \mu_j^{-1} \in H^s(W_j,\R^d).
\]
Without loss of generality we may assume that $I \cap J = \emptyset$. In a first step we show that for any open subset
$\U \Subset \W_j \cap \U_i$ with Lipschitz boundary $\partial \U$, the function 
$\left. \eta_i \circ h \circ \chi_i^{-1} \right|_U$ is in $H^s(U,\R^d)$. 
Here $U$ is given by $\chi_i(\U) \subseteq \R^n$. Indeed, note that as $U=\chi_i(\U) \Subset U_i$ and
$\mu_j(\U) \Subset W_j$ it follows that 
\[
\mu_j \circ \chi_i^{-1}:U \to \mu_j(\U) \mbox{ is in } C_b^\infty(\overline{U},\R^n)
\]
and
\[
\chi_i \circ \mu_j^{-1}: \mu_j(\U) \to U \mbox{ is in } C_b^\infty\big(\overline{\mu_j(\U)}, \R^n\big).
\]
Hence by Corollary \ref{th_transformation},
\[
 \left.(\nu_j \circ h \circ \mu_j^{-1}) \circ (\mu_j \circ \chi_i^{-1})\right|_U \in H^s(U,\R^d).
\]
Furthermore, one can choose $\V \subseteq N$ open so that
\[
h(\U) \Subset \V \Subset \Zj \cap \V_i.
\]
Hence $\left. \eta_i \circ \nu_j^{-1} \right|_{\nu_j(\V)}:\nu_j(\V) \to \eta_i(\V)$ is
in $C_b^\infty\big(\overline{\nu_j(\V)},\R^d\big)$.
One then can apply Proposition \ref{prop_left_translation} 
and Lemma \ref{lemma_smooth_extension} to conclude that
\[
\left. \eta_i \circ h \circ \chi_i^{-1} \right|_U = (\eta_i \circ \nu_j^{-1}) \circ (\nu_j \circ h \circ \mu_j^{-1})
\circ \left.(\mu_j \circ \chi_i^{-1})\right|_U \in H^s(U,\R^d).
\]
In view of this we can assume that the cover $(\W_j)_{j\in J}$ is a refinement of $(\U_i)_{i\in I}$,
i.e., for any $j\in J$ there exists $\sigma(j)\in I$ such that
\[
\W_j\subseteq\U_{\sigma(j)}\,,
\]
that satisfies the following additional properties: for any $j\in J$, $\W_j\Subset\U_{\sigma(j)}$,
\begin{equation}\label{eq:mu_j}
\mu_j\equiv\chi_{\sigma(j)}|_{\W_j} : \W_j\to W_j\Subset U_{\sigma(j)}\subseteq\R^n
\end{equation}
\begin{equation}\label{eq:nu_j}
\nu_j\equiv\eta_{\sigma(j)} : \Zj\equiv\V_{\sigma(j)}\to Z_j\equiv V_{\sigma(j)}\subseteq\R^d
\end{equation}
and
\begin{equation}\label{eq:smooth}
\nu_j\circ h\circ\mu_j^{-1}\in H^s(W_j,\R^d)\,.
\end{equation}
Now, choose an arbitrary $i\in I$ and consider the closure $\overline{\U_i}$ of $\U_i$ in $M$.
Let
\[
J_i:=\{j\in J\;|\;\W_j\cap\overline{\U_i}\ne\emptyset\}.
\]
Then $(\W_j)_{j\in J_i}$ is a open cover of
$\overline{\U_i}$. We can choose $(\W_j)_{j\in J_i}$ so that for any $j\in J_i$,
$\chi_i(\W_j\cap\U_i)\subseteq\R^n$ has Lipschitz boundary. Let $(\varphi_j)_{j\in J}$ be a partition of unity
on $M$ subordinate to the open cover $(\W_j)_{j\in J}$. By construction,
\begin{equation}\label{eq:1}
\Big(\sum_{j\in J_i}\varphi_j\Big)\Big|_{\overline{\U_i}}\equiv 1\,.
\end{equation}
Take an arbitrary $j\in J_i$. As the cover $(\U_l)_{l\in I}$ is of bounded type,
\begin{equation}
\chi_{\sigma(j)}\circ\chi_i^{-1}\in C^\infty_b(\overline{\chi_i(\U_{\sigma(j)}\cap\U_i)},\R^n)
\end{equation}
and
\begin{equation}
\eta_i\circ\eta_{\sigma(j)}^{-1}\in C^\infty_b(\overline{\eta_{\sigma(j)}(\V_{\sigma(j)}\cap\V_i)},\R^d)\,.
\end{equation}
In view of \eqref{eq:mu_j} and \eqref{eq:nu_j}
\begin{equation}\label{eq:mu_j'}
\mu_j\circ\chi_i^{-1}|_{\chi_i(\W_j\cap\U_i)}=\chi_{\sigma(j)}\circ\chi_i^{-1}|_{\chi_i(\W_j\cap\U_i)}\in
C^\infty_b(\overline{\chi_i(\W_j\cap\U_i)},\R^n)
\end{equation}
and
\begin{equation}\label{eq:nu_j'}
\eta_i\circ\nu_j^{-1}=\eta_i\circ\eta_{\sigma(j)}^{-1}|_{\eta_{\sigma(j)}(\V_{\sigma(j)}\cap\V_i)}\in
C^\infty_b(\overline{\eta_{\sigma(j)}(\V_{\sigma(j)}\cap\V_i)},\R^d)\,.
\end{equation}
We have
\begin{equation}\label{eq:comp}
(\eta_i\circ h\circ\chi_i^{-1})|_{\chi_i(\W_j\cap\U_i)}=
(\eta_i\circ\nu_j^{-1})\circ(\nu_j\circ h\circ\mu_j^{-1})\circ(\mu_j\circ\chi_i^{-1})|_{\chi_i(\W_j\cap\U_i)}\,.
\end{equation}
Then, in view of \eqref{eq:smooth}, \eqref{eq:mu_j'}, \eqref{eq:nu_j'}, and \eqref{eq:comp}, as well as
Corollary \ref{th_transformation}, Proposition \ref{prop_left_translation}, Lemma \ref{lemma_smooth_extension},
and Lemma \ref{lem:cut-off} one concludes that
\begin{equation}\label{eq:prod}
(\varphi_j\circ\chi_i^{-1})\cdot(\eta_i\circ h\circ\chi_i^{-1})\in H^s(U_i,\R^d)
\end{equation}
where the mapping above is extended from $\chi_i(\W_j\cap\U_i)$ to the whole neighborhood $U_i$ by
zero. Finally, in view of \eqref{eq:1} we get
\[
\eta_i\circ h\circ\chi_i^{-1}=\sum_{j\in J_i}(\varphi_j\circ\chi_i^{-1})\cdot(\eta_i\circ h\circ\chi_i^{-1})
\in H^s(U_i,\R^d)\,.
\]
This completes the proof of Lemma \ref{lemma_local_pt}.
\end{proof}

\vspace{0.5cm}

For a given {\em fine cover} $(\U_I,\V_I)$, introduce the subset $\0^s \equiv \0^s(\U_I,\V_I)$ of $H^s(M,N)$
\[
 \0^s:= \big\{ h \in H^s(M,N) \; \big| \; h(\U_i) \Subset \V_i \; \forall i \in I \big\}
\]
and the map
\[
\imath \equiv \imath_{\U_I,\V_I}:\0^s \to \oplus_{i \in I} H^s(U_i,\R^d), h \mapsto (h_i)_{i \in I}
\]
where for any $i \in I$
\[
h_i :=\eta_i \circ h \circ \chi_i^{-1}:U_i \to V_i \subseteq \R^d.
\]
By Lemma \ref{lemma_local_pt}, the map $\imath$ is well-defined and we say that $h_I:=(h_i)_{i \in I}$ is the
restriction of $h$ to $U_I:=(U_i)_{i \in I}$.

\begin{Def}\label{def_submanifold}
A subset $S$ of a Hilbert space $H$ is called a $C^\infty$-submanifold of $H$ if for any $p \in S$, there exist an
open neighborhood $V$ of $p$ in $H$, open neighborhoods $W_i \subseteq H_i$ of $0$ of the Hilbert spaces $H_i$,
$i=1,2$, and a $C^\infty$-diffeomorphism $\psi:V \to W_1 \times W_2$, with $\psi(p)=(0,0)$ so that,
\[
\psi(V \cap S) = W_1 \times \{0\}.
\]
\end{Def}

The following result will be proved in Section \ref{sec:diff_structure}.

\begin{Prop}\label{Th:submanifold}
Let $(\U_I,\V_I)$ be a fine cover and $\0^s \equiv \0^s(\U_I,\V_I)$ with $s > n/2$, and
$\imath \equiv \imath_{\U_I,\V_I}$ be defined as above. Then the range $\imath(\0^s)$ of $\imath$ is a
$C^\infty$-submanifold of $\oplus_{i \in I} H^s(U_i,\R^d)$.
\end{Prop}

To continue, let us recall the notion of a $C^\infty$-Hilbert manifold. Let $\mathcal M$ be a topological space.
A pair $(\U, \chi:\U \to U)$ consisting of an open subset $\U \subseteq \mathcal M$ and a homeomorphism
$\chi:\U \to U \subseteq H$ of $\U$ onto an open subset $U$ of a Hilbert space is said to be a chart of
$\mathcal M$. Occasionally, we also refer to $\U$ or to $\chi:\U\to U$ as a chart. For any $x \in \U$ we say that
 $(\U,\chi)$ is a chart at $x$. Two charts $\chi_i:\U_i \to U_i \subseteq H$ of $\mathcal M$ are said to be
\emph{compatible} if
$\chi_2 \circ \chi_1^{-1}:\chi_1(\U_1 \cap \U_2) \to \chi_2(\U_1 \cap \U_2)$ is a $C^\infty$-map between the open
sets $\chi_i(\U_1 \cap \U_2) \subseteq H$. An atlas of $\mathcal M$ is a cover $\mathcal A$ of $\mathcal M$ by 
compatible charts.
A maximal atlas of $M$ (maximality means that any chart that is compatible with the charts in $\mathcal A$ belongs
to $\mathcal A$) is said to be a $C^\infty$-differentiable structure of $\mathcal M$.
Clearly any atlas of $\mathcal M$ induces precisely one $C^\infty$-differentiable structure.
Assume that $(\U_I,\V_I)$ is a fine cover. The following result says that the $C^\infty$-differentiable structure on
the subset $\0^s \equiv \0^s(\U_I,\V_I)$ of $H^s(M,N)$ obtained by pulling back
the one of the submanifold $\imath(\0^s)$ does not depend on the choice of $(\U_I,\V_I)$.
More precisely, let $(\U_J,\V_J)$ be a fine cover. For convenience we choose the index sets $I,J$ so that
$I \cap J = \emptyset$. As above, introduce the subset ${\0}^s \equiv \0^s({\U}_J,{\V}_J)$ of $H^s(M,N)$ together
with the restriction map,
\[
{\imath} \equiv \imath_{{\U}_J,{\V}_J}: \0^s({\U}_J,{\V}_J)\to
\oplus_{j \in J} H^s({U}_j,\R^d), \quad f \mapsto ({f}_j)_{j \in J}
\]
where for any $j \in J$, ${f}_j$ is given by 
\[
 {f}_j:= {\eta}_j \circ f \circ {\chi}_j^{-1}: {U}_j \to {V}_j \subseteq \R^d.
\]
By Proposition \ref{Th:submanifold}, $\0^s({\U}_J,{\V}_J)$ admits a $C^\infty$-differentiable structure obtained by pulling back
the one of the submanifold 
\[
\imath \big( \0^s(\U_J,\V_J) \big) \subseteq \oplus_{j \in J} H^s({U}_j,\R^d).
\]
In Section \ref{sec:diff_structure} we prove the following statements:

\begin{Lemma}\label{lemma_opensubset}
Let $s$ be an integer, $s > n/2$, and let $(\U_I,\V_I)$ and $(\U_J,\V_J)$  be fine covers. 
Then $\0^s(\U_I,\V_I) \cap \0^s(\U_J,\V_J)$ is open in $\0^s(\U_I,\V_I)$. 
\end{Lemma}

\begin{Prop}\label{Th:diff_structure}
Let $s$ be an integer with $s>n/2$ and let $(\U_I,\V_I)$  and $({\U}_J,{\V}_J)$ be fine covers.
Then the $C^\infty$-differentiable structures on the intersection $\0^s(\U_I,\V_I) \cap \0^s({\U}_J,{\V}_J)$ induced
from $\0^s(\U_I,\V_I)$ and
$\0^s({\U}_J,{\V}_J)$ respectively, coincide. 
\end{Prop}

It follows from Lemma \ref{lem:fine_cover_existence} that the sets $\0^s(\U_I,\V_I), \0^s(\U_J,\V_J),\ldots$ 
constructed above, with $I,J,\ldots$ finite and pairwise
disjoint, form a cover $\C$ of $H^s(M,N)$. By Lemma \ref{lemma_opensubset}, the set $\mathcal T$ of subsets
$S \subseteq H^s(M,N)$, having the property that
\begin{equation}\label{topology_T}
S \cap \0^s(\U_I,\V_I) \mbox{ is open in } \0^s(\U_I,\V_I) \quad \forall \, \0^s(\U_I,\V_I) \in \C 
\end{equation}
defines a topology of $H^s(M,N)$. In particular, $\mathcal C$ is an open cover of $H^s(M,N)$ in the topology 
$\mathcal T$. Note that

\begin{Lemma}\label{lem:hausdorff}
The topology $\mathcal T$ of $H^s(M,N)$ is Hausdorff.
\end{Lemma}

\begin{proof}
Take $f,g\in H^s(M,N)$ so that $f\ne g$. Then there exists $x\in M$ such that $f(x)\ne g(x)$. 
Using that $f$ and $g$ are assumed continuous one constructs, as in Lemma \ref{lem:fine_cover_existence},
a fine cover $(\U_I,\V_I)$ with respect to $f$ and a fine cover $(\U_J,\V_J)$ with respect to $g$,
$I\cap J=\emptyset$, such that there exist $i\in I$ and $j\in J$ so that
\[
x\in \U_i,\quad\quad\U_i=\U_j,\quad\quad\mbox{and}\quad\quad \V_i\cap\V_j=\emptyset.
\]
Then, $\0^s(\U_I,\V_I)\cap \0^s(\U_I,\V_I)=\emptyset$.
As by Lemma \ref{lemma_opensubset} the sets $\0^s(\U_I,\V_I)$ and $\0^s(\U_I,\V_I)$ are open in $\mathcal T$ we
see that $\mathcal T$ is Hausdorff.
\end{proof}

Combining Proposition \ref{Th:diff_structure} with Lemma
\ref{lemma_opensubset} it follows that the cover $\C$ defines a $C^\infty$-differentiable structure on $H^s(M,N)$.

\begin{Coro}\label{Coro:hilbert_mfl}
Let $M$ be a closed manifold of
dimension $n$, $N$ a $C^\infty$-manifold of dimension $d$ and $s$ an integer with $s >
n/2$. Then the cover $\C$ induces a $C^\infty$-differentiable structure
$\mathcal A^s$ on $H^s(M,N)$ so that $H^s(M,N)$ is a Hilbert manifold. 
\end{Coro}

\begin{proof}
By Lemma \ref{lem:fine_cover_existence} and Lemma \ref{lemma_opensubset}, $\C$ is an open cover of $H^s(M,N)$.
The claimed statement then follows from Proposition \ref{Th:diff_structure}.
\end{proof}

Ebin and Marsden introduced a $C^\infty$-differentiable structure of $H^s(M,N)$ in terms of a Riemannian metric
$g \equiv g_N$ of $N$ -- see \cite{EM} or \cite{EMF}. More precisely, given any $f:M \to N$ in $H^s(M,N)$
introduce the linear space
\[
 T_f H^s(M,N):=\{ X \in H^s(M,TN) \, | \, \pi_N \circ X = f \}
\] 
where $\pi_N:TN \to N$ is the canonical projection of the tangent bundle $TN$ of $N$ to the base manifold $N$. 
Elements in $T_f H^s(M,N)$ are referred to as vector fields along $f$. On the linear space $T_f H^s(M,N)$ we define
an inner product as follows. Choose a fine cover $(\U_I,\V_I)$ so that $f \in \0^s(\U_I,\V_I)$.
In particular, $\U_I$ is an open cover of $M$ by coordinate charts of $M$, 
$\chi_i:\U_i \to U_i \subseteq \R^n$ and $\V_I$ is a set of coordinate charts of $N$,
$\eta_i:\V_i \to V_i \subseteq \R^d$. The restriction of an arbitrary element $X \in T_f H^s(M,N)$ to $\U_i$ induces
a continuous map $X_i:U_i \to \R^d$,
\[
X_i(x) = \left(X^k\big(\chi^{-1}_i(x) \big)\right)_{k=1}^d,\quad x \in U_i
\]
where $X^k$ are the coordinates of $X\big(\chi_i^{-1}(x)\big)$ in the chart $V_i \subseteq \R^d$.
Using that $(\U_I,\V_I)$ is a fine cover one concludes from Lemma \ref{lemma_local_pt} and the compactness of
$X(M)\subseteq TN$ that
\[
X_i\in H^s(U_i,\R^d).
\]
The family $(X_i)_{i \in I}$ is referred to as the restriction of $X$ to
$U_I=(U_i)_{i \in I}$. For $X,Y \in T_f H^s(M,N)$, define
\begin{equation}\label{e:scalar_product}
{\langle X , Y \rangle}_s := \sum_{i \in I, |\alpha| \leq s} \int_{U_i} \langle \partial^\alpha  X_i,
\partial^\alpha Y_i \rangle dx
\end{equation}
where $\langle \cdot , \cdot \rangle$ denotes the Euclidean inner product in $\R^d$. Then 
${\langle \cdot, \cdot \rangle}_s$ is a inner product, making $T_f H^s(M,N)$ into a Hilbert space.
Another choice of $\U_I,\V_I$ will lead to a possibly different inner product, but the two Hilbert norms can be
shown to be equivalent. In this way one obtains a differential structure of $T_fH^s(M,N)$. With the help of the 
exponential maps $exp_y:T_y N \to N$, $y \in N$, defined in terms of the Riemannian metric $g$ of $N$,
Ebin and Marsden (\cite{EM}) show that $H^s(M,N)$ is a $C^\infty$-Hilbert manifold.\footnote{Note that
our arguments will give an independent proof of this fact.}
More specifically, charts on $H^s(M,N)$ are defined with the help of the exponential map
\[
\overline{\exp}: O^s \to H^s(M,N),\quad X \mapsto \left[ x \mapsto \exp_{f(x)}\big(X(x)\big)\right],
\]
where $O^s \subseteq T_f H^s(M,N)$ is a sufficiently small neighborhood of zero in $T_f H^s(M,N)$ --
see Section \ref{sec:diff_structure} for more details. We denote the $C^\infty$-differentiable structure of
$H^s(M,N)$ defined in this way by $\mathcal A_{g}^s$. In Section \ref{sec:diff_structure} we prove

\begin{Prop}\label{Th:same_diff_structure}
Let $M$ be a closed manifold of dimension $n$, $N$ a $C^\infty$-manifold endowed with a Riemannian metric $g$,
and $s$ an integer with $s>n/2$. Then 
\[
\mathcal A^s = \mathcal A_{g}^s.
\]
\end{Prop}

Now let $M$ be a closed oriented $n$-dimensional manifold and let $s$ be an integer with $s>n/2+1$.
From Proposition \ref{prop_extension} and the assumption $s>n/2+1$ it follows that $H^s(M,M)$ can be continuously
 embedded into $C^1(M,M)$. As in Lemma \ref{lem:openness} one sees that
\[
\Ds^s(M):=\{ \varphi \in \mbox{Diff}_+^1(M) \;|\; \varphi \in H^s(M,M) \}
\]
is open in $H^s(M,M)$. Hence $\Ds^s(M)$ is a $C^\infty$-Hilbert manifold. 

\begin{Lemma}\label{Lemma:inverse_continuous}
Let $M$ be a closed oriented manifold of dimension $n$ and $s$ be an integer with
$s>n/2+1$. Then for any $\varphi \in \Ds^s(M)$, the inverse $\varphi^{-1}$ is
in $\Ds^s(M)$ and the map
\[
{\tt inv} : \Ds^s(M) \to \Ds^s(M),\quad \varphi \mapsto \varphi^{-1}
\]
is continuous.
\end{Lemma}

For the convenience of the reader we include a proof of Lemma
\ref{Lemma:inverse_continuous} in Appendix \ref{appendix A}.

\subsection{Proof of Theorem \ref{thm2}}\label{subsec:proof1.2}
To prove Theorem \ref{thm2}, we first need to introduce some more notation.
Let $M$ be a closed oriented manifold of dimension $n$ and $N$ a $C^\infty$-manifold of dimension $d$.
Consider open covers $\U_I:=(\U_i)_{i \in I}$ and $\V_I=(\V_i)_{i \in I}$ of $M$ where $I \subseteq \N$
is finite and a set of open subsets $\W_I:=(\W_i)_{i \in I}$ of $N$ so that for any $i \in I$, $\U_i$ and $\V_i$
are coordinate charts of $M$, $\chi_i:\U_i \to U_i \subseteq \R^n$, $\eta_i:\V_i \to V_i \subseteq \R^n$ and $\W_i$
is a coordinate chart of $N$, $\xi_i:\W_i \to W_i \subseteq \R^d$ where $U_i$ and $V_i$ are bounded, open subsets of $\R^n$ with Lipschitz boundaries. Let $U_I=(U_i)_{i \in I}$, $V_I=(V_i)_{i \in I}$, and $W_I=(W_i)_{i \in I}$.
For such data we introduce the subsets
\[ 
 \Ps^s(U_I,V_I) \subseteq \oplus_{i \in I} H^s(U_i,\R^n)
\]
and
\[
 \Ps^s(V_I,W_I) \subseteq \oplus_{i \in I} H^s(V_i,\R^d)
\]
consisting of elements $(h_i)_{i \in I}\in\oplus_{i \in I} H^s(U_i,\R^n)$ and
$(f_i)_{i \in I}\in\oplus_{i \in I} H^s(V_i,\R^d)$ respectively such that for any $i \in I$,
\begin{equation}\label{compact_image}
h_i(U_i) \Subset V_i \quad \mathrm{and} \quad f_i(V_i) \Subset W_i.
\end{equation}
Further, for any integer $s$ with $s > n/2+1$, introduce the subset $\Ds^s(U_I,V_I)$ consisting of elements $(\varphi_i)_{i \in I}$ in $\Ps^s(U_I,V_I)$ so that for any $i \in I$, $\varphi_i:{\overline U}_i \to V_i$ is 1-1 and
\[
0<\inf_{x \in{\overline U}_i}\det(d_x \varphi_i).
\]
By Proposition \ref{prop_extension}, $\Ps^s(V_I,W_I)$ is open in
$\oplus_{i \in I} H^s(V_i,\R^d)$. Moreover, one concludes from Lemma \ref{lem:openness} and
Proposition \ref{prop_extension} that $\Ds^s(U_I,V_I)$ is open in $\oplus_{i \in I} H^s(U_i,\R^n)$.
For any integers $r$, $s$ with $r \geq 0$ and $s > n/2+1$ define the map 
\begin{eqnarray*}
\tilde{\mu}_I: \Ps^{s+r}(V_I,W_I) \times \Ds^s(U_I,V_I) &\to& \Ps^s(U_I,W_I) \\
((f_i)_{i \in I},(\varphi_i)_{i \in I}) & \mapsto & (f_i \circ \varphi_i)_{i \in I} 
\end{eqnarray*}

By Proposition \ref{prop_composition_general}, $\tilde{\mu}_I$ is well--defined and has the following property.

\begin{Lemma}\label{Lemma:local_regularity}
$\tilde{\mu}_I$ is a $C^r$-map.
\end{Lemma}

\begin{Prop}\label{Prop:composition_regularity}
Let $M$ be a closed oriented manifold of dimension $n$, $N$ a $C^\infty$-manifold of dimension $d$, and $r$, $s$
integers with
$r \geq 0$ and $s > n/2+1$. Then 
\[
 \mu: H^{s+r}(M,N) \times \Ds^s(M) \to H^s(M,N), \quad (f,\varphi) \mapsto f \circ \varphi
\]
is a $C^r$-map.
\end{Prop}

\begin{proof}
Let $\varphi \in \Ds^s(M)$ and $f \in H^{s+r}(M,N)$ be arbitrary.  Arguing as in the proof of
Lemma \ref{lem:fine_cover_existence} one constructs open covers $(\U_i)_{i\in I}$ and $(\V_i)_{i\in I}$ on $M$
as well as an open cover $(\W_i)_{i\in I}$ of $f(M)$ in $N$ such that
$(\U_I,\V_I)$ is a fine cover with respect to $\varphi$ and $(\V_I,\W_I)$ is a fine cover with respect to $f$.
Denote by $\0^s(\U_I,\V_I)$ and $\0^{s+r}(\V_I,\W_I)$ the open subsets of $H^s(M,M)$ respectively $H^{s+r}(M,N)$,
introduced in Section \ref{subsec3:prelim}.
Then $\Ds^s(M) \cap \0^s(\U_I,\V_I)$ is an open neighborhood of $\varphi$ in $\Ds^s(M)$ and $\0^{s+r}(\V_I,\W_I)$ is
an open neighborhood of $f$ in $H^{s+r}(M,N)$. Furthermore, note that
\[
 \imath_{\U_I,\V_I}\big(\Ds^s(M) \cap \0^s(\U_I,\V_I)\big) \subseteq \Ds^s(U_I,V_I) 
\]
and
\[
 \imath_{\V_I,\W_I}\big(\0^{s+r}(\V_I,\W_I)\big) \subseteq \Ps^{s+r}(V_I,W_I)
\]
where $\imath_{\U_I,\V_I}$ and $\imath_{\V_I,\W_I}$ are the embeddings introduced in Section \ref{subsec3:prelim}. One has the following commutative diagram:

\[
\begin{array}{ccc}
\0^{s+r}(\V_I,\W_I) \times \left( \Ds^s(M) \cap \0^s(\U_I,\V_I)\right) &\stackrel{\mu_I}{\longrightarrow}&\0^s(\U_I,\W_I)\\ 
\Big\downarrow \rlap{\footnotesize{$\imath_{\V_I,\W_I} \times \imath_{\U_I,\V_I}$}}&&\Big\downarrow\rlap{\footnotesize{$\imath_{\U_I,\W_I}$}}\\
\mathcal P^{s+r}(V_I,W_I) \times \Ds^s(U_I,V_I)&\stackrel{\tilde \mu_I}{\longrightarrow}&\mathcal P^s(U_I,W_I)
\end{array}
\]
where $\mu_I$ is the restriction of the composition
\[
 \mu:H^{s+r}(M,N) \times \Ds^s(M) \to H^s(M,N)
\]
to $\0^{s+r}(\V_I,\W_I) \times \left(\Ds^s(M)\cap \0^s(\U_I,\V_I)\right)$. In view of Lemma \ref{Lemma:local_regularity} 
\[
 \tilde \mu_I:\Ps^{s+r}(V_I,W_I) \times \Ds^s(U_I,V_I) \to \Ps^s(U_I,W_I)
\]
is $C^r$-smooth. By the definition of the differential structure on $\0^{s+r}(\V_I,\W_I)$ and $\0^s(\U_I,\V_I)$ (see \S \ref{subsec3:prelim}) we get from the commutative diagram above that $\mu_I$ is $C^r$-smooth. As $\varphi$, $f$ are arbitrary, it follows that $\mu$ is $C^r$-smooth.
\end{proof}

Next we consider the inverse map, associating to any
$C^1$-diffeomorphism $\varphi:M \to M$ of a given closed manifold $M$ its
inverse. Following the arguments of the proof of Proposition \ref{prop2} and using Proposition \ref{Prop:composition_regularity} we obtain

\begin{Prop}\label{Prop:inverse_regularity}
For any closed oriented manifold $M$ of dimension $n$ and any integers $r$, $s$ with $r \geq 1$ and $s > n/2+1$
\[
 {\tt inv}:\Ds^{s+r}(M) \to \Ds^s(M), \quad \varphi \mapsto \varphi^{-1}
\]
is a $C^r$-map.
\end{Prop}

\begin{proof}[Proof of Theorem \ref{thm2}]
The claimed results are established by Proposition
\ref{Prop:composition_regularity}, Lemma \ref{Lemma:inverse_continuous} and Proposition \ref{Prop:inverse_regularity}.
\end{proof}

As an immediate consequence of Proposition \ref{Prop:composition_regularity} and Lemma \ref{Lemma:inverse_continuous} we obtain the following

\begin{Coro}\label{Coro:topological_group}
For any closed oriented manifold $M$ of dimension $n$ and any integer $s > n/2+1$, $\Ds^s(M)$ is a topological group.
\end{Coro}

\section{Differentiable structure of $H^s(M,N)$}\label{sec:diff_structure}
In Section \ref{subsec3:prelim} we outlined the construction of a $C^\infty$-differentiable structure of $H^s(M,N)$ for any integer $s$ with $s>n/2$. In this section we prove the auxiliary results stated in Subsection \ref{subsec3:prelim}, which were needed for this construction. Throughout this section we assume that $M$ is a closed manifold of dimension $n$, $s \in \Z$ with $s>n/2$, $N$ is a $C^\infty$-manifold of dimension $d$, and $g \equiv g_N$ is a $C^\infty$-Riemannian metric on $N$.

\subsection{Submanifolds}\label{subsec4:submanifolds}
The main purpose of this subsection is to prove Proposition \ref{Th:submanifold}. Let us begin by recalling the set-up. Choose a fine cover $(\U_I,\V_I)$ as defined in Subsection \ref{subsec3:prelim}. In particular, $\U_I=(\U_i)_{i \in I}$ is a finite cover of $M$ and $\V_I=(\V_i)_{i \in I}$ one of $\cup_{i \in I} \V_i$ and for any $i \in I$, $\U_i,\V_i$ are coordinate charts $\chi_i:\U_i \to U_i \subseteq \R^n$ respectively $\eta_i:\V_i \to V_i \subseteq \R^d$. Recall that $\0^s(\U_I,\V_I)$, introduced in subsection \ref{subsec3:prelim}, is given by
\begin{equation}\label{subbase}
\0^s(\U_I,\V_I) = \big\{ h \in H^s(M,N) \ \big| \ h(\U_i) \Subset \V_i \quad \forall i \in I \big\}
\end{equation}
and the map
\begin{equation}\label{submanifold_embedding}
\imath \equiv \imath_{\U_I,\V_I} : \0^s(\U_I,\V_I) \to \oplus_{i \in I} H^s(U_i,\R^d),
\end{equation}
defined by $\imath(h):=(h_i)_{i \in I}$ and $h_i=\eta_i \circ h \circ \chi_i^{-1}:U_i \to V_i \subseteq \R^d$ is 
injective. Proposition \ref{Th:submanifold} states that $\imath\big( \0^s(\U_I,\V_I)\big)$ is a submanifold of 
$\oplus_{i \in I} H^s(U_i,\R^d)$. We will prove this by showing that for any $f \in \0^s(\U_I,\V_I)$ there exists a 
neighborhood $Q^s$ of $(f_i)_{i \in I}$ in $\oplus_{i \in I} H^s(U_i,\R^d)$ so that
$Q^s \cap \imath\big(\0^s(\U_I,\V_I)\big)$ coincides with $\imath \circ \overline{\exp}_f (O^s)$ where
$\overline{\exp}_f$ is the exponential map $\overline{\exp}_f:T_fH^s(M,N) \to H^s(M,N)$ defined below
(see also the discussion of the differential structure $\mathcal A_g^s$ of $H^s(M,N)$ in Subsection
\ref{subsec3:prelim}) and $O^s$ is a (small) neighborhood of $0$ in $T_fH^s(M,N)$.
By proving that $d_0 (\imath \circ \overline{\exp}_f)$ splits (Lemma \ref{lem_closed_range} below)
we then conclude that $\imath\big(\0^s(\U_I,\V_I)\big)$ is a submanifold of
$\oplus_{i \in I} H^s(U_i,\R^d)$. Let us now look at the Hilbert space $T_fH^s(M,N)$ and the map
$\overline{\exp}_f$ in more detail. For any $y \in N$, denote by $T_yN$ the tangent space of $N$ at $y$ and by
$\exp_y$ the exponential map of the Riemannian metric $g$ on $N$. It maps a (sufficiently small) element
$v \in T_yN$ to the point in $N$ on the geodesic issuing at $y$ in direction $v$ at time $t=1$. For any $y \in N$
the exponential map $\exp_y$ is defined in a neighborhood of $0_y$ in $T_yN$. Furthermore, for any
$X \in T_fH^s(M,N)$, with $f \in H^s(M,N)$, and $x \in M$, $X(x)$ is an element in $T_{f(x)}N$,
hence if $\|X(x)\|$ is sufficiently small, $\exp_{f(x)}\big(X(x)\big)\in N$ is well defined and,
for $X$ sufficiently small, we can introduce the map
\[
\overline{\exp}_f(X):= M \to N, \quad x \mapsto \exp_{f(x)}\big(X(x)\big).
\]
Note that for $X=0$, $\overline{\exp}_f(0)=f$. To analyze the map $\overline{\exp}_f$ further let us express it
in local coordinates provided by the fine cover $(\U_I,\V_I)$. The restriction of an arbitrary element
$X \in T_fH^s(M,N)$ to $U_i$ is given by the map
\begin{equation}\label{vectorfield_restriction}
X_i:U_i \to \R^d,\quad x \mapsto X_i(x).
\end{equation}
As $X\in T_fH^s(M,N)$, $X_i$ is an element in $H^s(U_i,\R^d)$. Recall that $T_fH^s(M,N)$ is a Hilbert
space. Without loss of generality we assume that the inner product (\ref{e:scalar_product}) is defined in terms
of $\U_I$ and $\V_I$. It is then immediate that the linear map
\begin{equation}\label{tangentspace_in_coordinates}
\rho:T_fH^s(M,N) \to \oplus_{i \in I} H^s(U_i,\R^d),\quad X \mapsto (X_i)_{i \in I}
\end{equation}
is an isomorphism onto its image. For $X$ (sufficiently) close to $0$ we want to describe the restriction of 
$\overline{\exp}_f(X)$ to $U_I=(U_i)_{i \in I}$. To this end, let us express $\exp_y(v)$ for $y \in \V_i$, $i \in I$,
and $v$ sufficiently close to $0$ in $T_yN$ in local coordinates provided by $\eta_i:\V_i \to V_i$. For any small
$v \in T_yN$, $\eta_i(\exp_yv)$ is given by $\gamma_i(1;y_i,v_i)$ where $t \mapsto \gamma_i(t;y_i,v_i) \in \R^d$ is
the geodesic issuing at $y_i:=\eta_i(y)$ in direction given by the coordinate representation $v_i$ of the vector $v$.
The geodesic $\gamma_i(t;y_i,v_i)$ satisfies the ODE on $V_i$,
\begin{equation}\label{ode_eq}
\ddot \gamma_i + \Gamma(\gamma_i)(\dot \gamma_i,\dot \gamma_i) = 0
\end{equation}
with initial data
\begin{equation}\label{ode_initial_data}
\gamma_i(0;y_i,v_i)=y_i \quad \mbox{and} \quad \dot \gamma_i(0;y_i,v_i)=v_i.
\end{equation}
Here $\, \dot{} \,$ stays for $\frac{d}{dt}$ and for any $z_i \in V_i$ and $w_i=(w_i^p)_{1 \leq p \leq d} \in \R^d$,
\begin{equation}\label{christoffel_extended}
\Gamma(z_i)(w_i,w_i) = \left( \sum_{1 \leq p,q \leq d} \Gamma_{pq}^k(z_i) w_i^p w_i^q \right)_{1 \leq k \leq d}
\end{equation}
with $\Gamma_{pq}^k$ denoting the Christoffel symbols of the Riemannian metric $g$, expressed in the local coordinates of the chart $\eta_i:\V_i \to V_i$,
\begin{equation}\label{christoffel_coordinates}
\Gamma_{pq}^k = \frac{g^{kl}}{2} \big( \partial_{z_i^q} g_{pl} - \partial_{z_i^l} g_{pq} + \partial_{z_i^p} g_{lq} \big)
\end{equation}
where $g_{pl}$ are the coefficients of the metric tensor and $g^{lk} \cdot g_{km}=\delta_m^k$ where $\delta_m^k$ is the Kronecker delta. Note that $\Gamma_{pq}^k$ is a $C^\infty$-function on $V_i$. The velocity vector $v_i \in \R^d$ in (\ref{ode_initial_data}) is chosen close to zero so that the solution $\gamma_i(t;y_i,v_i)$ exists and stays in $V_i$ for any $|t|<2$. Now let us return to the map $X \mapsto \overline{\exp}_f(X)$. Its restriction to $U_i$ is given by the time one map of the flow $X_i \mapsto \alpha_i(t;X_i)$, where for any $Y_i \in H^s(U_i,\R^d)$, $\alpha_i(t;Y_i)$ solves the ODE
\begin{equation}\label{ode_first_order}
(\dot \alpha_i,\dot Z_i) = \big(Z_i, - \Gamma(\alpha_i)(Z_i,Z_i)\big)
\end{equation}
with initial data
\begin{equation}\label{ode_first_order_initial_data}
\big( \alpha_i(0;Y_i),Z_i(0;Y_i)\big) = (f_i,Y_i).
\end{equation}
As above, $f_i$ is given by $f_i=\eta_i \circ f \circ \chi_i^{-1}$ and satisfies $f_i(U_i) \Subset V_i$.

\begin{Lemma}\label{lem_ode_solution}
For any $f \in \0^s(\U_I,\V_I)$ and $i \in I$, there exists a neighborhood $O_i^s$ of $0$ in $H^s(U_i,\R^d)$ so that for any $Y_i \in O_i^s$, the initial value problem (\ref{ode_first_order})-(\ref{ode_first_order_initial_data}) has a unique $C^\infty$-solution
\[
 (-2,2) \to H^s(U_i,\R^d) \times H^s(U_i,\R^d), \quad t \mapsto \big( \alpha_i(t;Y_i),Z_i(t;Y_i)\big)
\]
satisfying 
\[
\alpha_i(t;Y_i)(U_i) \Subset V_i\,.
\]
In fact, 
\[
 (\alpha_i,Z_i) \in C^\infty\left( (-2,2)\times O_i^s, H^s(U_i,\R^d) \times H^s(U_i,\R^d)\right).
\]
\end{Lemma}

\begin{proof}
The claimed result follows from the classical theorem for ODE's in Banach spaces on the existence, uniqueness, and 
$C^\infty$-smooth dependence on initial data of solutions (cf. e.g. \cite{Lang}). Indeed, denote by $H^s(U_i,V_i)$
the subset of the Hilbert space $H^s(U_i,\R^d)$,
\[
 H^s(U_i,V_i):=\{ h \in H^s(U_i,\R^d) \ | \ h(U_i) \Subset V_i \}.
\]
By the Sobolev embedding theorem (Proposition \ref{prop_extension} $(iii)$) $H^s(U_i,V_i)$ is open in $H^s(U_i,\R^d)$.
We claim that the vector field
\begin{eqnarray*}
H^s(U_i,V_i) \times H^s(U_i,\R^d) &\to& H^s(U_i,\R^d) \times H^s(U_i,\R^d)\\
(h_i,Y_i) &\mapsto& \big(Y_i,-\Gamma(h_i)(Y_i,Y_i) \big)
\end{eqnarray*}
is well-defined and $C^\infty$-smooth. Indeed, as $h_i \in H^s(U_i,V_i)$, one has that $h_i(U_i) \Subset V_i$,
thus the composition $\Gamma \circ h_i$ is well-defined. Furthermore, by Proposition \ref{prop_left_translation},
Lemma \ref{lemma_smooth_extension}, \eqref{christoffel_extended} and \eqref{christoffel_coordinates}
\[
H^s(U_i,V_i) \to H^s(U_i,\R), \quad h_i \mapsto \Gamma_{pq}^k(h_i)
\]
is $C^\infty$-smooth. By Lemma \ref{lem_imbedding_general}, $H^s(U_i,\R)$ is an algebra and multiplication of elements
of $H^s(U_i,\R)$ is $C^\infty$-smooth. Hence the map
\[
H^s(U_i,V_i) \times H^s(U_i,\R^d) \to H^s(U_i,\R^d),\quad (h_i,Y_i) \mapsto \Gamma(h_i)(Y_i,Y_i)
\] 
is $C^\infty$-smooth. Summarizing our considerations we have proved that the vector field
\begin{eqnarray*}
H^s(U_i,V_i) \times H^s(U_i,\R^d) &\to& H^s(U_i,\R^d) \times H^s(U_i,\R^d)\\
(h_i,Y_i) &\mapsto& \big(Y_i, -\Gamma(h_i)(Y_i,Y_i)\big)
\end{eqnarray*}
is $C^\infty$-smooth. Further note that for $Y_i \equiv 0$, $\big(\alpha_i(t,0),Z_i(t,0)\big)=(f_i,0)$ is a stationary
solution of (\ref{ode_first_order})-(\ref{ode_first_order_initial_data}). Hence by the classical local in time
existence and uniqueness theorem for solutions of ODE's in Banach spaces we conclude that there exists a 
(small) neighborhood $O_i^s$ of $0$ in $H^s(U_i,\R^d)$ so that for any $Y_i \in O_i^s$, the initial value problem
(\ref{ode_first_order})-(\ref{ode_first_order_initial_data}) has a unique solution
$\big(\alpha_i(t,Y_i),Z_i(t,Y_i)\big)$ in $C^\infty\big((-2,2),H^s(U_i,V_i) \times H^s(U_i,\R^d)\big)$. 
As the solution depends $C^\infty$-smoothly on the initial data one concludes that 
$(\alpha_i,Z_i) \in C^\infty\big((-2,2) \times \0_i^s,H^s(U_i,V_i) \times H^s(U_i,\R^d)\big)$.   
\end{proof}

\begin{Coro}\label{coro_smooth_exp}
For any $f \in \0^s(\U_I,\V_I)$, there exists a neighborhood $O^s$ of $0$ in $T_fH^s(M,N)$ so that for any 
$X \in O^s$, $\overline{\exp}_f(X)$ is in $\0^s(\U_I,\V_I)$ and the composition 
$\imath_f:= \imath \circ \overline{\exp}_f(X)$,
\[
O^s \overset{\overline{\exp}_f}{\longrightarrow} \0^s(\U_I,\V_I) \overset{\imath}{\longrightarrow} 
\oplus_{i \in I} H^s(U_i,\R^d)
\]
is $C^\infty$-smooth.
\end{Coro}

\begin{proof}
For any $i \in I$, the $i$-th component of the restriction map
\[
\rho_i:T_fH^s(M,N) \to H^s(U_i,\R^d),\quad X \mapsto X_i(x)
\]
is linear and bounded by the definition of $T_fH^s(M,N)$, hence it is $C^\infty$-smooth. As a consequence
\begin{equation}\label{eq:O}
O^s:= \bigcap_{i \in I}\,\rho_i^{-1}(O_i^s) \subseteq T_fH^s(M,N)
\end{equation}
is an open neighborhood of $0$ in $T_fH^s(M,N)$ with $O_i^s$ being the neighborhood of $0$ in $H^s(U_i,\R^d)$ of
Lemma \ref{lem_ode_solution}. The latter implies that for any $i \in I$, the composition
\[
O^s \overset{\rho_i}{\longrightarrow} H^s(U_i,\R^d) \overset{\alpha_i(1;\cdot)}{\longrightarrow} H^s(U_i,V_i)
\]
is $C^\infty$-smooth. Recall that the restriction of $\overline{\exp}_f(X)$ to $U_i$ is given by $\alpha_i(1;X_i)$. Hence $\overline{\exp}_f(X) \in \0^s(\U_I,\V_I)$ and 
\begin{equation}\label{smooth_i}
\imath_f(X)=\big(\alpha_i(1;\rho_i(X))\big)_{i \in I}
\end{equation}
showing that $\imath_f$ is $C^\infty$-smooth as $\rho_i : T_fH^s(M,N)\to H^s(U_i,\R^d)$ is a bounded linear map.
\end{proof}

Next we want to analyze the map $\imath_f$ further.

\begin{Lemma}\label{lem_closed_range}
For any $f \in \0^s(\U_I,\V_I)$, the differential $d_0\imath_f:T_fH^s(M,N) \to \oplus_{i \in I}H^s(U_i,\R^d)$ of
$\imath_f$ at $X=0$ is 1-1 and has closed range.
\end{Lemma}

\begin{proof}
We claim that for any $X \in T_fH^s(M,N)$,
\[
 d_0\imath_f(X)=\big(\rho_i(X)\big)_{i \in I}
\]
where for any $x \in U_i$, $\rho_i(X)(x)=X_i(x)$ is the $i$-th component of the restriction map. Indeed, for any $\lambda \in \R$ with $|\lambda| < 1$ and $|t|<2$, any solution of the initial value problem (\ref{ode_first_order})-(\ref{ode_first_order_initial_data}) with $Y_i$ and $\lambda Y_i$ in $O_i^s$ satisfies
\begin{equation}\label{scaling_exp}
\alpha_i(\lambda t;Y_i)=\alpha_i(t;\lambda Y_i).
\end{equation}
As $\rho_i(\lambda X)=\lambda \rho_i(X)$ by the linearity of the map $\rho_i$ it then follows from \eqref{smooth_i} and \eqref{scaling_exp} that for any $X \in O^s$ with $\lambda X \in O^s$,
\[
 \imath_f(\lambda X)=\big(\alpha_i(\lambda;X_i)\big)_{i \in I}
\]
and hence
\[
 \left. \frac{d}{d\lambda} \right|_{\lambda=0} \imath_f(\lambda X) = \big(\dot \alpha_i(0;X_i)\big)_{i \in I}=(X_i)_{i \in I}.
\]
As a consequence, $d_0\imath_f(X)=\big(\rho_i(X)\big)_{i \in I}$ for any $X \in T_fH^s(M,N)$ and $d_0\imath_f$ is 1-1. It remains to show that $d_0\imath_f$ has closed range. Note that for any given $X \in O^s$ and $x \in \chi_j(\U_i \cap \U_j)$ with $i,j \in I$, the restrictions $X_i$ and $X_j$ are related by
\begin{equation}\label{closedness_relation}
 d_{f_j(x)} (\eta_i \circ \eta_j^{-1}) \cdot X_j(x) = X_i\big(\chi_i \circ \chi_j^{-1}(x)\big).
\end{equation}
Conversely, if $(Y_i)_{i \in I} \in \oplus_{i \in I} H^s(U_i,\R^d)$ satisfies the relations
\eqref{closedness_relation} for any $x \in \chi_j(\U_i \cap \U_j)$ and $i,j \in I$, there exists
$X \in T_f H^s(M,N)$ so that
\begin{equation}\label{vector_field_along_f}
\rho_i(X)=Y_i
\end{equation}
for any $i \in I$. As $s>n/2$, it then follows from Lemma \ref{lem_imbedding_general},
Corollary \ref{th_transformation}, Proposition \ref{prop_extension}(ii), as well as
Proposition \ref{prop_left_translation} and Lemma \ref{lemma_smooth_extension}, that for any $i,j \in I$,
the linear map
\begin{eqnarray*}
R_{ij}:\oplus_{i \in I} H^s(U_i,\R^d) &\to& H^s\left(\chi_j(\U_i \cap \U_j),\R^d\right),\\
(X_i)_{i \in I} &\mapsto& d_{f_j(x)} \big(\eta_i \circ \eta_j^{-1}\big)
\cdot X_j(x)-X_i\big(\chi_i \circ \chi_j^{-1}(x)\big)
\end{eqnarray*}
is bounded. Hence, the relations (\ref{closedness_relation}) define a closed linear subspace of
$\oplus_{i \in I} H^s(U_i,\R^d)$.
\end{proof}

Lemma \ref{lem_closed_range} will be used to show that $\imath\big(\0^s(\U_I,\V_I)\big)$ is a submanifold of
 $\oplus_{i \in I} H^s(U_i,\R^d)$ by applying the following corollary of the inverse function theorem.

\begin{Lemma}\label{lem_embedded_submanifold}
Let $E$ and $H$ be Hilbert spaces and let $H_1$ be a closed subspace of $H$. Furthermore let $V$ be an open
neighborhood of $0$ in $E$ and $\Phi:V \to H$ a $C^\infty$-map so that $d_0\Phi(E)=H_1$ and
$\mbox{Ker}\ d_0 \Phi = \{0\}$. Then there exist a $C^\infty$-diffeomorphism $\Psi$ of some open neighborhood of
$\Phi(0) \in H$ to an open neighborhood of $0 \in H$ and an open neighborhood $V_1 \subseteq V$ of $0$ in $E$ so that
$\left.\Psi \circ \Phi\right|_{V_1}$ is a $C^\infty$-diffeomorphism onto an open neighborhood of $0$ in $H_1$.
\end{Lemma}

See e.g. \cite{Lang}, Chapter I, Corollary 5.5 for a proof.

\begin{proof}[Proof of Proposition \ref{Th:submanifold}]
We will show that for any $f\in\0^s(\U_i,\V_I)$ there exists an open neighborhood $Q^s$ of $\imath(f)$ in
$\oplus_{i \in I} H^s(U_i,\R^d)$ such that 
\[
\imath_f(O^s)=Q^s \cap \imath\big(\0^s(\U_I,\V_I)\big)
\]
where $O^s$ is an open neighborhood of zero in $T_fH^s(M,N)$ such that
$\imath_f(O^s)$ is a submanifold in $\oplus_{i \in I} H^s(U_i,\R^d)$.
Recall that the differential of the map $Y_i \mapsto \alpha_i(1;Y_i)$ of Lemma \ref{lem_ode_solution} at $Y_i=0$ is the identity (cf. the proof of Lemma \ref{lem_closed_range}),
\[
d_0\alpha_i(1;\cdot) = \id_{H^s(U_i,\R^d)}.
\]
It thus follows by the inverse function theorem that for any $i \in I$, there exists an open neighborhood $Q_i^s$ of $f_i$ contained in $H^s(U_i,V_i)$ such that, after shrinking $O_i^s$, if necessary 
\[
(P1)\quad\quad
\left\{
\begin{array}{l}
\alpha_i(1;\cdot):O_i^s \to Q_i^s\,\,\mbox{ia a $C^\infty$-diffeomorphism}\\
\forall\,\,Y_i\in O^s_i,\,\,\alpha_i(1;Y_i)(U_i)\Subset V_i
\end{array}
\right.
\]
By shrinking the neighborhood $O^s_i$ of zero in $H^s(U_i,\R^d)$ once more one can ensure that the open
neighborhood $O^s$ of zero in $T_fH^s(M,N)$ given by \eqref{eq:O} satisfies the following two additional
properties:
\[
\begin{array}{l}
(P2)\quad\quad\quad\, \imath_f(O^s)\,\,\mbox{is a submanifold in}\,\,\oplus_{i \in I} H^s(U_i,\R^d)\\
(P3)\quad\quad\quad\, \forall\,\xi\in O^s,\,\,g(\xi,\xi)<\varepsilon\,.
\end{array}
\]
where $\varepsilon > 0$ is chosen as in Lemma \ref{global_exp} below.
Our candidate for the open neighborhood $Q^s$ of $\imath(f)=(f_i)_{i\in I}$ in $\oplus_{i \in I}H^s(U_i,\R^d)$ is
\[
Q^s:= \oplus_{i \in I} Q_i^s.
\]
Take $h \in \0^s(\U_I,\V_I)$ with $\imath(h)=(h_i)_{i \in I} \in Q^s$. By the definition of $Q^s$ and $Q_i^s$, 
there exists $(Y_i)_{i \in I} \in \oplus_{i \in I}O_i^s$ such that for any $i \in I$, $\alpha_i(1;Y_i)=h_i$. 
We now have to show that $(Y_i)_{i \in I}$ is the restriction of a global vector field along $f$. 
In view of \eqref{closedness_relation} and \eqref{vector_field_along_f} it is to prove that for any 
$x \in \chi_j(\U_i \cap \U_j)$, $i,j \in I$, the identity (\ref{closedness_relation}) is satisfied. 
Assume the contrary. Then there exists $k,l \in I$ and $x \in \U_k \cap \U_l$ so that, 
with $x_k:=\chi_k(x)$, $x_l:=\chi_l(x)$ and $y=f(x) \in \V_k \cap \V_l$, the vectors $\xi \in T_yN$ and
$\bar\xi \in T_yN$ corresponding to $Y_k(x_k)$ and $Y_l(x_l)$ respectively do not coincide,
\begin{equation}\label{different_pushforward}
\xi \neq \bar \xi.
\end{equation}
On the other hand, by the definition of $h_k$ and $\alpha_k$
\[
h_k(x_k)=\alpha_k(1;Y_k)(x_k)=\eta_k(\exp_y \xi)
\]
and, similarly,
\[
h_l(x_l)=\alpha_l(1;Y_l)(x_l)=\eta_l(\exp_y \bar \xi).
\]
As $\imath(h)=(h_i)_{i \in I}$ it then follows that
\[
\exp_y \xi = h(x) = \exp_y \bar \xi.
\]
However, in view of the choice of $\varepsilon$ in $(P3)$ and Lemma \ref{global_exp} below,
the latter identity contradicts (\ref{different_pushforward}). Hence $(Y_i)_{i \in I}$ satisfies 
\eqref{closedness_relation} and $\imath_f(X)=(Y_i)_{i \in I}$ where $X \in O^s$ is the vector field along $f$
defined by (\ref{vector_field_along_f}).
\end{proof}

It remains to state and prove Lemma \ref{global_exp} used in the proof of Proposition \ref{Th:submanifold}. 
For any $\varepsilon > 0$ and any subset $A \subseteq N$ denote by $B_g^\varepsilon A$ the 
$\varepsilon$-ball bundle of $N$ restricted to $A$
\[
B_g^\varepsilon A = \big\{ \xi \in \cup_{y \in A} T_y N \ \big| \ g(\xi,\xi)^{1/2} < \varepsilon \big\}
\]
where $g$ is the Riemannian metric on $N$. Denote by $\pi:TN \to N$ the canonical projection. Recall that $f \in H^s(M,N)$ implies that $f$ is continuous. As $M$ is assumed to be closed, $f(M)$ is compact. By the classical ODE theorem and the compactness of $f(M)$ there exists a neighborhood $\V$ of $f(M)$ in $N$ and $\varepsilon >0$ so that
\[
\Phi:B_g^\varepsilon \V \to N \times N,\quad \xi \mapsto \big(\pi(\xi),\exp_{\pi(\xi)}\xi\big)
\]
is well-defined and $C^\infty$-smooth. 

\begin{Lemma}\label{global_exp}
For any $f \in \0^s(\U_I,\V_I)$, there exists $\varepsilon > 0$ and an open neighborhood $\V$ of $f(M)$ so that
\[
\Phi:B_g^\varepsilon \V \to \W\subseteq N\times N,\quad \xi \mapsto \big(\pi(\xi), \exp_{\pi(\xi)}\xi \big)
\]
is a $C^\infty$-diffeomorphism onto an open neighborhood $\W$ of $\{ (y,y) \,|\, y \in \V \}$ in $N \times N$.
\end{Lemma}

\begin{proof}
Note that for any $\xi \in TN$ of the form $0_y \in T_y N$ with $y \in f(M)$, $\Phi(0_y)=(y,y)$ and
$d_{0_y}\Phi:T_{0_y}(TN) \to T_y N \times T_y N$ is a linear isomorphism. By the inverse function theorem and
the compactness of $f(M)$ it then follows that there exist an open neighborhood $\V$ of $f(M)$,
an open neighborhood $\W$ of the diagonal $\{(y,y)\,|\, y \in \V \}$ in $N\times N$, and $\varepsilon > 0$ so that
\begin{equation}\label{eq:Phi}
\Phi: B_g^\varepsilon \V \to\W\subseteq N\times N, \quad \xi \mapsto \big(\pi(\xi),\exp_{\pi(\xi)}\xi \big)
\end{equation}
is a local diffeomorphism that is {\em onto} and that for any $x\in\V$
\[
\Phi\Big|_{B_g^\varepsilon\V\cap T_x N} : B_g^\varepsilon\V\cap T_x N\to N
\]
is a diffeomorphism onto its image. The last statement and the formula for $\Phi$ in \eqref{eq:Phi} imply that
$\Phi$ is is injective. Hence, $\Phi$ is a bijection. As it is also a local diffeomorphism,
$\Phi$ is a diffeomorphism.
\end{proof}

\begin{Rem}
Note that we did not use the Ebin-Marsden differential structure on $N^s(M,N)$. 
In consequence, our construction gives an independent proof of Ebin-Marsden's result.
\end{Rem}
\noindent As a by-product, the proof of Proposition \ref{Th:submanifold} leads to the following
\begin{Coro}\label{coro_same_diff_structure}
For any set of the form $\0^s(\U_I,\V_I)$,
\[
\mathcal A^s \cap \0^s(\U_I,\V_I) = \mathcal A_g^s \cap \0^s(\U_I,\V_I)
\]
i.e. the $C^\infty$-differentiable structure induced from $\oplus_{i \in I} H^s(U_i,\R^d)$ coincides on $\0^s(\U_I,\V_I)$ with the one of Ebin-Marsden, introduced in \cite{EM}.
\end{Coro}

\subsection{Differentiable structure}
In this subsection we prove Proposition \ref{Th:diff_structure} and Proposition \ref{Th:same_diff_structure} as well
as Lemma \ref{lemma_opensubset}. Recall that the map 
\begin{equation}\label{e:imath2}
\imath \equiv \imath_{\U_I,\V_I} : \0^s(\U_I,\V_I)\to\oplus_{i\in I}H^s(U_i,\R^d)
\end{equation}
is injective and by Proposition \ref{Th:submanifold}, the image of $\imath$ is a $C^\infty$-submanifold
in $\oplus_{i\in I}H^s(U_i,\R^d)$. Hence, by pulling back the $C^\infty$-differentiable structure of the image of
$\imath$, we get a $C^\infty$-differentiable structure on $\0^s(\U_I,\V_I)$. First we prove Lemma 
\ref{lemma_opensubset}.

\begin{proof}[Proof of Lemma \ref{lemma_opensubset}]
Let $(\U_I,\V_I)$ and $(\U_J,\V_J)$ be fine covers. For convenience assume that the index sets $I,J$ are chosen in
such a way that $I \cap J = \emptyset$. It is to show that $\0^s(\U_I,\V_I) \cap \0^s(\U_J,\V_J)$ is open in
$\0^s(\U_I,\V_I)$. Given $h \in \0^s(\U_I,\V_I) \cap \0^s(\U_J,\V_J)$ consider its restriction 
$(h_i)_{i \in I} = \imath_{\U_I,\V_I}(h)$ in $\oplus_{i \in I} H^s(U_i,\R^d)$ and choose a Riemannian metric $g$ on 
$N$. In view of Proposition \ref{prop_extension} (iii), for any $\varepsilon > 0$, there exists an open neighborhood 
$W$ of $(h_i)_{i \in I}$ in $\oplus_{i \in I} H^s(U_i,\R^d)$ such that for any 
$(p_i)_{i \in I} \in W \cap \imath_{\U_I,\V_I}\big(\0^s(\U_I,\V_I)\big)$ and any $x \in M$
\begin{equation}\label{small_distance}
\mathop{\rm dist}\nolimits_g\big(p(x),h(x)\big) < \varepsilon
\end{equation}
where $\mathop{\rm dist}\nolimits_g$ is the geodesic distance function on $(N,g)$ and $p \in H^s(M,N)$ is the unique element of 
$\0^s(\U_I,\V_I)$ such that $\imath_{\U_I,\V_I}(p) =(p_i)_{i \in I}$. It follows from \eqref{small_distance} and the
definition of $\0^s(\U_I,\V_I)$ that the neighborhood $W$ of $(h_i)_{i \in I}$ in $\oplus_{i \in I} H^s(U_i,\R^d)$ 
can be chosen so that
\begin{equation}\label{inverse_W}
\W := \imath^{-1}_{\U_I,\V_I} \left(W \cap \imath_{\U_I,\V_I}\big(\0^s(\U_I,\V_I)\big)\right) \subseteq \0^s(\U_J,\V_J).
\end{equation}
In view of the definition of the topology on $\0^s(\U_I,\V_I)$, $\W$ is an open neighborhood of $h$ in 
$\0^s(\U_I,\V_I)$. As $h \in \0^s(\U_I,\V_I) \cap \0^s(\U_J,\V_J)$ was chosen arbitrarily, formula 
\eqref{inverse_W} implies that $\0^s(\U_I,\V_I) \cap \0^s(\U_J,\V_J)$ is open in $\0^s(\U_I,\V_I)$.
\end{proof}

Next we prove Proposition \ref{Th:diff_structure} which says that the $C^\infty$-differentiable structures of 
$\0^s(\U_I,\V_I) \cap \0^s(\U_J,\V_J)$ induced by the ones of $\0^s(\U_I,\V_I)$ and $\0^s(\U_J,\V_J)$ coincide.

\begin{proof}[Proof of Proposition \ref{Th:diff_structure}]
Let $(\U_I,\V_I)$ and $(\U_J,\V_J)$ be fine covers. For convenience we choose $I,J$ such that $I \cap J \neq \emptyset$
and assume that $\0^s_{IJ} := \0^s(\U_I,\V_I) \cap \0^s(\U_J,\V_J) \neq \emptyset$. Note that the boundary 
$\partial \chi_i(\U_i \cap \U_j), i \in I, j \in J$, might not be Lipschitz. To address this issue we 
refine the covers $(\U_I,\V_I)$ and $(\U_J,\V_J)$. For any $h \in \0^s_{IJ}$ there exist fine covers $(\U_K,\V_K)$,
$(\U_L,\V_L)$ with $I,J,K,L$ pairwise disjoint such that (i) $h \in \0^s(\U_K,\V_K) \cap \0^s(\U_L,\V_L)$, 
(ii) there exist maps  $\sigma:K \to I$ and $\tau:L \to J$ so that for any $k \in K$, $\ell \in L$
\[
\U_k \Subset  \U_{\sigma(k)},\; \V_k \Subset \V_{\sigma(k)} \quad \mbox{and} \quad \U_\ell \Subset
\U_{\tau(\ell)},\; \V_\ell \Subset \V_{\tau(\ell)},
\]
and (iii) for any $k \in K$ and $\ell \in L$, $\U_k \cap \U_\ell \subseteq M$ and $\V_k \cap \V_\ell \subseteq N$
have piecewise smooth boundary and $\U_K \cup \U_L := \{\U_k,\U_\ell\}_{k \in K,\ell \in L}$ is a cover of $M$ of
bounded type. \\
Fine covers $(\U_K,\V_K)$  and $(\U_L,\V_L)$ with properties (i)-(iii) can be constructed by choosing for 
$\U_k, \V_k \; (k \in K)$ and $\U_\ell,\V_\ell \; (\ell \in L)$ appropriate geodesic balls defined in terms of
Riemannian metrics on $M$ and $N$ respectively and arguing as in the proof of Lemma \ref{lem:fine_cover_existence}.
Moreover, we choose for any $k \in K$ the coordinate chart $\chi_k:\U_k \to U_k \subseteq \R^n$ to be the restriction
of the coordinate chart $\chi_{\sigma(k)}:\U_{\sigma(k)} \to U_{\sigma(k)} \subseteq \R^n$ to $\U_k$. 
In a similar way we choose the coordinate charts $\eta_k \; (k \in K)$ and $\chi_\ell,\eta_\ell \;(\ell \in L)$.
Let $\0^s_{KL} := \0^s(\U_K,\V_K) \cap \0^s(\U_L,\V_L)$ and $\0^s_{IJKL} := \0^s_{IJ} \cap \0^s_{KL}$ and define 
\begin{eqnarray*}
\mathcal F_I := \oplus_{i \in I} H^s(U_i,\R^d),& \mathcal F_J := \oplus_{j \in J} H^s(U_j,\R^d),\\
\mathcal F_K := \oplus_{k \in K} H^s(U_k,\R^d),& \mathcal F_L := \oplus_{\ell \in L} H^s(U_\ell,\R^d).
\end{eqnarray*}
By Lemma \ref{lemma_opensubset}, the sets $\0^s_{KL}, \0^s_{IJ}$, and $\0^s_{IJKL}$ are open sets in the topology $\mathcal T$, defined by \eqref{topology_T}. To prove Proposition \ref{Th:diff_structure} it suffices to show that the $C^\infty$-differentiable structures on $\0^s_{IJKL}$ induced from the ones of $\0^s_I$ and $\0^s_J$ coincide. For this purpose, consider the following diagram
\begin{equation}\label{figure1}
\begin{array}{c}
\begin{array}{ccccc}
\mathcal F_I & \hspace{1.3cm}& \0^s_{IJ} &\hspace{1.3cm}& \mathcal F_J
\vspace{0.2cm}\\
\rotatebox{90}{$\subseteq$} && \rotatebox{90}{$\subseteq$} && \rotatebox{90}{$\subseteq$}\\
\end{array}\\
\begin{diagram}
\imath_I(\0^s_{IJKL}) & \lTo^{\imath_I} & \0^s_{IJKL} & \rTo^{\imath_J} & \imath_J(\0^s_{IJKL})\\
\dTo^{\mathcal P_I} & \ldTo^{\imath_K} && \rdTo^{\imath_L} & \dTo_{\mathcal P_J}\\
\imath_K(\0^s_{IJKL}) && \rTo^{\mathcal R} && \imath_L(\0^s_{IJKL})
\end{diagram}\\
\begin{array}{ccccc}
\rotatebox{-90}{$\subseteq$} && && \rotatebox{-90}{$\subseteq$}
\vspace{0.2cm}\\
\mathcal F_K & \hspace{1.6cm}&  &\hspace{1.6cm}& \mathcal F_L\\
\end{array}\\
\end{array}
\end{equation}
where $\imath_I,\imath_J,\imath_K$, and $\imath_L$ denote the corresponding restrictions of $\imath_{\U_I,\V_I}$, 
$\imath_{\U_J,\V_J}$, $\imath_{\U_K,\V_K}$, and $\imath_{\U_L,\V_L}$, to $\0^s_{IJKL}$ and $\mathcal P_I,\mathcal P_J$
 are the maps
\begin{eqnarray}\label{p_ij}
\nonumber
\mathcal P_I:\imath_I(\0^s_{IJKL}) \to \imath_K(\0^s_{IJKL}),&\; (f_i)_{i \in I} &\mapsto (\left. f_{\sigma(k)} 
\right|_{U_k} )_{k \in K},\\
\mathcal P_J:\imath_J(\0^s_{IJKL}) \to \imath_L(\0^s_{IJKL}),&\; (f_j)_{j \in J} &\mapsto (\left. f_{\tau(\ell)} 
\right|_{U_\ell})_{\ell \in L}\,.
\end{eqnarray}
Finally, the map $\mathcal R:\imath_K(\0^s_{IJKL}) \to \imath_L(\0^s_{IJKL})$ is defined in such a way that the
central sub-diagram in \eqref{figure1} is commutative. Note that by the definition of the charts
$\chi_k,\eta_k \;(k \in K)$ and $\chi_\ell, \eta_\ell \; (\ell \in L)$, the left and right sub-diagrams in 
\eqref{figure1} are commutative. By Lemma \ref{lemma_R_diffeo} below the map $\mathcal R$ is a diffeomorphism. 
Proposition \ref{Th:diff_structure} then follows once we show that the maps $\mathcal P_I$ and $\mathcal P_J$
are diffeomorphisms, as in this case, $\mathcal P^{-1}_J \circ \mathcal R \circ \mathcal P_I$ is a diffeomorphism.
 Consider the map $\mathcal P_I$. As $\mathcal P_I$ is the restriction of the bounded linear map
\[
\widetilde{\mathcal P}_I:\mathcal F_I \to \mathcal F_K,\; (f_i)_{i \in I} \mapsto (\left. f_{\sigma(k)}
\right|_{U_k})_{k \in K}
\]
to the submanifold $\imath_I(\0^s_{IJKL}) \subseteq \mathcal F_I$, $\mathcal P_I$ is smooth.
Take an arbitrary element $f_I \equiv \imath_I(f) \in \imath_I(\0^s_{IJKL})$ and consider the differential of 
$\mathcal P_I$ at $f_I$,
\[
d_{f_I} \mathcal P_I:\rho_I\big(T_f H^s(M,N)\big) \to \rho_K\big(T_f H^s(M,N)\big)
\]
where $\rho_I$ is the restriction map \eqref{tangentspace_in_coordinates} corresponding to $(\U_I,\V_I)$ and $\rho_K$
is the restriction map corresponding to $(\U_K,\V_K)$. In view of the choice of the coordinate charts
$(\chi_k)_{k \in K}$, $d_{f_I} \mathcal P_I$ is given by
\begin{equation}\label{differential_P_I}
d_{f_I} \mathcal P_I:(X_i)_{i \in I} \mapsto \left( \left. X_{\sigma(k)} \right|_{U_k} \right)_{k \in K}.
\end{equation}
In particular it follows from \eqref{differential_P_I} that $d_{f_I} \mathcal P_I$ is injective and onto.
Hence, by the open mapping theorem $d_{f_I} P_I$ is a linear isomorphism.
As $f_I \in \imath_I(\0^s_{IJKL})$ is arbitrary, $\mathcal P_I:\imath_I(\0^s_{IJKL}) \to \imath_K(\0^s_{IJKL})$ is a
local diffeomorphism. As by the commutativity of the left sub-diagram of \eqref{figure1}, $\mathcal P_I$ is
a homeomorphism we get that it is a diffeomorphism. Similarly, one proves that $\mathcal P_J$ is a diffeomorphism.
\end{proof}

Next we prove Lemma \ref{lemma_R_diffeo} used in the proof of Proposition \ref{Th:diff_structure}.
Let $\mathcal R$ be the map introduced there.

\begin{Lemma}\label{lemma_R_diffeo}
$\mathcal R$ is a diffeomorphism.
\end{Lemma}

\begin{proof}
Throughout the proof we use the notation introduced in the proof of Proposition \ref{Th:diff_structure} without
 further reference. Consider the following diagram
\begin{equation}\label{figure2}
\begin{diagram}
&& \0^s_{KL} &&\\
&\ldTo^{\imath_K} && \rdTo^{\imath_L}&\\
\mathcal F_K \supseteq \imath_K(\0^s_{KL}) &&\rTo^{\widetilde{\mathcal R}} && \imath_L(\0^s_{KL}) \subseteq \mathcal F_L
\end{diagram}
\end{equation}
where $\widetilde{\mathcal R}:\imath_K(\0^s_{KL}) \to \imath_L(\0^s_{KL})$ is the map defined by
$\widetilde{\mathcal R}\big(\imath_K(f)\big)=\imath_L(f)$ for any $f \in \0^s_{KL}$. Clearly, the diagram 
\eqref{figure2} is commutative and $\mathcal R$ is the restriction of $\widetilde{\mathcal R}$ to
$\imath_K(\0^s_{IJKL})$. It suffices to show that $\widetilde{\mathcal R}$ is a diffeomorphism. Note that
\[
 \0_{KL}^s = \0^s(\U_K \cap \U_L,\V_K \cap \V_L)
\]
where
\[ 
\U_K \cap \U_L = (\U_k \cap \U_\ell)_{k \in K,\ell \in L} \quad \mbox{and} \quad  \V_K \cap \V_L =
(\V_k \cap \V_\ell)_{k \in K,\ell \in L}.
\]
On $\U_K\cap \U_L$ and $\V_K\cap{\V}_L$ one can introduce
two families of coordinate charts. For any given $k \in K$ and $\ell \in L$ define

\hspace{0.1cm}

\[
\alpha_{k\ell}:=\chi_k|_{\U_k\cap{\U}_\ell} : \U_k\cap{\U}_\ell\to
\chi_k(\U_k\cap{\U}_\ell) \subseteq U_k \subseteq\R^n\,,
\]
\[
\beta_{k\ell}:=\eta_k|_{\V_k\cap{\V}_\ell} : \V_k\cap{\V}_\ell\to
\eta_k(\V_k\cap{\V}_\ell) \subseteq V_k \subseteq\R^d\,.
\]
and, alternatively,
\[
{\gamma}_{k\ell}:={\chi}_\ell|_{{\U}_k\cap \U_\ell} :
{\U}_k \cap \U_\ell\to {\chi}_\ell({\U}_k\cap \U_\ell) \subseteq
U_\ell \subseteq\R^n\,,
\]
\[
{\delta}_{k\ell}:={\eta}_\ell|_{{\V}_k\cap \V_\ell} :
{\V}_k \cap \V_\ell \to {\eta}_\ell({\V}_k\cap \V_\ell)\subseteq
V_\ell \subseteq\R^d\,.
\]

\hspace{0.1cm}

These two choices of coordinate charts lead to the two embeddings $\imath_1$ and $\imath_2$
\begin{eqnarray}\label{e:i_1}
\imath_1 : \0^s(\U_K\cap{\U}_L,\V_K\cap{\V}_L) &\to& \oplus_{k \in K, \ell \in L}H^s(\chi_k(\U_k\cap{\U_\ell}),\R^d)\\
\nonumber
f &\mapsto& (f_{k\ell})_{k \in K, \ell \in L}
\end{eqnarray}
where
\[
f_{k\ell}:=\beta_{k\ell}\circ f\circ\alpha_{k\ell}^{-1} : \chi_k(\U_k\cap{\U_\ell})\to
\eta_k(\V_k\cap{\V}_\ell)\subseteq\R^d
\]
and 
\begin{eqnarray}\label{e:i_2}
\imath_2 : \0^s(\U_K\cap{\U}_L,\V_K\cap{\V}_L) &\to& \oplus_{k \in K, \ell 
\in L}H^s({\chi}_\ell({\U}_k\cap \U_\ell),\R^d)\\
\nonumber
f &\mapsto& ({g}_{k\ell})_{k \in K, \ell \in L}
\end{eqnarray}
where
\[
{g}_{k\ell}:= \delta_{k\ell} \circ f \circ \gamma_{k\ell}^{-1} : 
\chi_\ell(\U_k \cap \U_\ell) \to \eta_\ell(\V_k \cap \V_\ell) \subseteq \R^d.
\]

\noindent Let 
\[
\mathcal G_K:=\oplus_{k \in K, \ell \in L}H^s(\chi_k(\U_k \cap \U_\ell),\R^d),
\]
\[
\mathcal G_L:=\oplus_{k \in K, \ell \in L}H^s({\chi}_\ell({\U}_k\cap \U_\ell),\R^d)
\]
and consider the following diagram
\begin{equation}\label{figure3}
\begin{diagram}
\mathcal F_K \supseteq \imath_K(\0^s_{KL}) & \lTo^{\imath_K} & \0^s_{KL} & \rTo^{\imath_L} & \imath_L(\0^s_{KL})
\subseteq \mathcal F_L \\
\dTo^{\mathcal R_K} & \ldTo^{\imath_1} && \rdTo^{\imath_2} & \dTo_{\mathcal R_L} \\
\mathcal G_K \supseteq \imath_1(\0^s_{KL}) && \rTo^T && \imath_2(\0^s_{KL}) \subseteq \mathcal G_L
\end{diagram}
\end{equation}
\noindent where $\imath_K$ is the restriction of 
\[
\imath_{\U_K,\V_K}:\0^s(\U_K,\V_K) \to \mathcal F_K
\] 
to $\0_{KL}^s \subseteq \0^s(\U_K,\V_K)$, $\imath_L$ is defined similarly, and the maps $R_K$, $R_L$, and $T$ are defined by
\begin{eqnarray*}
R_K &:& \mathcal F_K\to \mathcal G_K,\,\,\,(f_k)_{k \in K}\mapsto(f_{k\ell})_{k \in K,\ell \in L},\;\,f_{k\ell}:=f_k|_{\chi_k(\U_k\cap{\U}_\ell)}\,,\\
R_L &:& \mathcal F_L \to \mathcal G_L,\,\,\,({f}_\ell)_{\ell \in L}\mapsto({g}_{k\ell})_{k \in K,\ell \in L},\;\,
{g}_{k\ell}:={f}_\ell|_{{\chi}_\ell({\U}_k\cap \U_\ell)},\\
T &:& \mathcal G_K\to \mathcal G_L,\,(f_{k\ell})_{k \in K,\ell \in L}\mapsto({g}_{k\ell})_{k \in K,\ell \in L},\;\,
\end{eqnarray*}
with 
\[
 g_{k\ell}:=\left(\eta_\ell \circ \eta_k^{-1}\right) \circ f_{k\ell} \circ \left(\chi_k \circ \chi_\ell^{-1}\right).
\]
Note that the diagram \eqref{figure3} commutes. The arguments used to prove that $\mathcal P_I$ in \eqref{figure1} is a diffeomorphism show that $\mathcal R_K$ and $\mathcal R_L$ are diffeomorphisms. We claim that $T$ is a diffeomorphism. First note that $T$ is bijective and its inverse $T^{-1}$ is given by 
\[
 T^{-1}: \mathcal G_L \to \mathcal G_K,\quad (g_{k\ell})_{k \in K,\ell \in L} \mapsto (f_{k\ell})_{k \in K,\ell \in L}
\]
with
\[
 f_{k\ell} = \left.\left(\eta_k \circ \eta_\ell^{-1}\right) \circ g_{k\ell} \circ \left(\chi_\ell \circ
\chi_k^{-1}\right) \right|_{\chi_k(\U_k \cap \U_\ell)}.
\]
In view of the boundedness of the extension operator of Proposition \ref{prop_extension}$(ii)$ the smoothness of $T$
and $T^{-1}$ then follows from Corollary \ref{coro_right_translation}, Proposition \ref{prop_left_translation}.
and Lemma \ref{lemma_smooth_extension}. Comparing the diagrams \eqref{figure2} and \eqref{figure3} we conclude that
$\widetilde{\mathcal R} = \mathcal R_K \circ T \circ \mathcal R_L^{-1}$. Hence $\widetilde{\mathcal R}$ is a
diffeomorphism.
\end{proof}

\begin{proof}[Proof of Proposition \ref{Th:same_diff_structure}]
The claim that the $C^\infty$-differentiable structure on $H^s(M,N)$, introduced by Ebin-Marsden and the one
introduced in this paper coincide follows from Corollary \ref{coro_same_diff_structure} and Proposition
\ref{Th:diff_structure}.
\end{proof}

As a consequence of Proposition \ref{Th:same_diff_structure} we obtain the following corollary.
\begin{Coro}\label{Coro:independence}
The $C^\infty$-differentiable structure on $H^s(M,N)$ introduced in \cite{EM}, is independent of the choice of the
Riemannian metric on $N$.
\end{Coro}

\appendix
\section{Appendix}\label{appendix A}
In this appendix we prove Lemma \ref{Lemma:inverse_continuous}. First we need to establish an auxiliary result. 
Throughout this appendix, we will use the notation introduced in Section \ref{Section 3}. For bounded open subsets
$U,W \subseteq \R^n$ with $C^\infty$-boundaries and $s>n/2+1$, denote by $\Ds^s_{U,W}$ the following subset of
$\Ds^s(U,\R^n)$,
\[
\Ds^s_{U,W}:=\big\{ \varphi \in \Ds^s(U,\R^n) \ \big| \ \overline{W} \subseteq \varphi(U) \big\}.
\]
Arguing as in Lemma \ref{lem:openness} one can prove that $\Ds^s_{U,W}$ is an open subset of $\Ds^s(U,\R^n)$.
Moreover, following the arguments of the proof of Lemma \ref{continuous_inverse} one gets

\begin{Lemma}\label{continuous_inverse_charts}
Let $U,W$, and $s$ be as above. Then, for any $\varphi \in \Ds^s_{U,W}$, $\left. \varphi^{-1} \right|_{W} \in \Ds^s(W,\R^n)$
and the map
\[
 \Ds^s_{U,W} \to \Ds^s(W,\R^n), \quad \varphi \mapsto \left. \varphi^{-1} \right|_{W}
\]
is continuous.
\end{Lemma}

\begin{proof}[Proof of Lemma \ref{Lemma:inverse_continuous}]
Let $\varphi$ be an arbitrary element in $\Ds^s(M)$. To see that its inverse $\varphi^{-1}$ is again in $\Ds^s(M)$, 
it suffices to verify that when expressed in local coordinates, the map $\varphi^{-1}$ is of Sobolev class $H^s$. 
To be more precise, let 
\[
\chi: \U \to U \subseteq \R^n \quad \mbox{and} \quad \eta:\V \to V \subseteq \R^n
\]
be coordinate charts so that $U,V$ are open, bounded subsets of $\R^n$ with $C^\infty$-boundaries and 
$\varphi(\U) \Subset \V$. 
By the construction of the fine cover in Lemma \ref{lem:fine_cover_existence} we can assume that $(\U,\V)$ is
a part of a fine cover $(\U_I,\V_I)$ with respect to $\varphi\in\D^s(M)$.
Then, by Lemma \ref{lemma_local_pt}, $\psi := \eta \circ \varphi \circ \chi^{-1}$
is in $H^s(U,\R^n)$. Choose $\W \Subset \varphi(\U)$ so that $W:=\eta(\W)$ is an open bounded subset of $\R^n$ with
$C^\infty$-boundary. By Lemma \ref{continuous_inverse_charts}, it follows that 
$\left. \psi^{-1} \right|_W:W \to \R^n$ is in $\Ds^s(W,\R^n)$. As the chart $\U,\V$ as well as $\W$ were chosen
arbitrarily, we conclude that $\varphi^{-1}$ is in $\Ds^s(M)$. 
By the construction of the fine cover in Lemma \ref{lem:fine_cover_existence} we can choose a fine cover $(\U_I,\V_I)$
with respect to $\varphi\in\D^s(M)$ and $\W_I\Subset\varphi(\U_I)$ such that $(\W_I,\U_I)$ is a fine cover with
respect to $\varphi^{-1}\in\D^s(M)$. Then, Lemma \ref{continuous_inverse_charts} implies that the map
$\Ds^s(M) \to \Ds^s(M)$, $\varphi \mapsto \varphi^{-1}$ is continuous.
\end{proof}

\section{Appendix}\label{appendix B}

In this appendix we discuss the extension of Theorem \ref{thm1} and Theorem \ref{thm2} to the case where $s$ is a real number with $s>n/2+1$.

For $s \in \R_{\geq 0}$, denote by $H^s(\R^n,\R)$ the Hilbert space
\[
 H^s(\R^n,\R):=\big\{ f \in L^2(\R^n,\R) \ \big| \ (1+|\xi|^2)^{s/2} \hat f(\xi) \in L^2(\R^n,\R) \big\}
\]
with inner product
\[
 {\langle f,g \rangle}_s^\sim=\int_{\R^n} \hat f(\xi) \overline{\hat{g}(\xi)} (1+|\xi|^2)^{s} d\xi
\]
and induced norm
\[
 \|f\|_s^\sim:= \left(\langle f,f \rangle_s^\sim \right)^{1/2}.
\]
By \eqref{equivalent_norm}, the norms $\|f\|_s^\sim$ and $\|f\|_s$ are equivalent for any integer $s \geq 0$. In the
sequel, by a slight abuse of notation, we will write $\|f\|_s$ instead of $\|f\|_s^\sim$ and 
$\langle \cdot,\cdot \rangle_s$ instead of $\langle \cdot,\cdot \rangle_s^\sim$ for any $s \in \mathbb R_{\geq 0}$. 
In a straightforward way one proves the following lemma.

\begin{Lemma}\label{lem_inductive_norm}
For any $f \in L^2(\R^n,\R)$ and $s \in \R_{\geq 1}$, $f \in H^s(\R^n,\R)$ iff for any $1 \leq i \leq n$,
the distributional derivate $\dx{i}f$ is in $H^{s-1}(\R^n,\R)$. Moreover $\|f\|+\sum_{i=1}^n \|\dx{i}f\|_{s-1}$ is a
norm on $H^s(\R^n,\R)$ which is equivalent to $\|f\|_s$.
\end{Lemma}
\vspace{0.2cm}
For $s \in \R_{> 0} \setminus \mathbb N$, elements in $H^s(\R^n,\R)$ can be conveniently characterized as follows --
see e.g. \cite[Theorem 7.48]{adams}.

\begin{Lemma}\label{lem_charakterisierung} Let $s \in \R_{> 0} \setminus \N$ and $f \in L^2(\R^n,\R)$. Then
$f \in H^s(\R^n,\R)$ iff $f \in H^{\lfloor s \rfloor}(\R^n,\R)$ and $[\partial^\alpha f]_\lambda < \infty$ for any
multi-index $\alpha=(\alpha_1,\ldots,\alpha_n)$ with $|\alpha|=\lfloor s \rfloor$ where $\lambda=s-\lfloor s \rfloor$
and where $[\partial^\alpha f]_\lambda$ denotes the $L^2$-norm of the function 
\[
\R^n \times \R^n \to \R, \quad (x,y) \mapsto \frac{|\partial^\alpha f(x)-\partial^\alpha f(y)|}{|x-y|^{\lambda+n/2}}.
\]
Moreover $\sqrt{{\langle\!\langle f,f \rangle\!\rangle}_s}$ is a norm on $H^s(\R^n,\R)$, equivalent to $\|\cdot\|_s$,
where ${\langle\!\langle\cdot,\cdot\rangle\!\rangle}_s$ is the inner product
\[
{\langle\!\langle f, g \rangle\!\rangle}_s = \langle f, g\rangle_{\lfloor s \rfloor} + 
\!\!\sum_{\substack{\alpha \in \mathbb Z_{\geq 0}^n\\|\alpha|=\lfloor s \rfloor}}\!\!\!
\int_{\R^n}\!\!\!\int_{\R^n}\!\!\!\frac{\big(\partial^\alpha f(x)- \partial^\alpha f(y)\big) 
\big(\partial^\alpha g(x)- \partial^\alpha g(y)\big)}{|x-y|^{n+2\lambda}}\,dxdy.
\]
\end{Lemma}

\begin{proof} We argue by induction with respect to $s$. In view of Lemma \ref{lem_inductive_norm}, 
it suffices to prove the claimed statement in the case $0 < s < 1$. Then $\lambda=s$ and we have
\begin{eqnarray*}
\int_{\R^n}\int_{\R^n} \frac{|f(x)-f(y)|^2}{|x-y|^{n+2s}}\,\,dxdy &=& \int_{\R^n}\int_{\R^n}
\frac{|f(x+z)-f(x)|^2}{|z|^{n+2s}}\,\,dxdz\\
&=& \int_{\R^n} \frac{1}{|z|^{n+2s}} \left(\int_{\R^n} |f(x+z)-f(x)|^2 dx\right) dz.
\end{eqnarray*}
By Plancherel's theorem, 
\begin{eqnarray*}
\int_{\R^n} |f(x+z)-f(x)|^2\,dx &=& \int_{\R^n} |\widehat{f(\cdot+z)}(\xi)-\hat f(\xi)|^2\,d\xi\\
&=& \int_{\R^n} |e^{iz \cdot \xi}-1|^2 |\hat f(\xi)|^2\,d\xi.
\end{eqnarray*}
Therefore 
\begin{eqnarray*}
\int_{\R^n}\int_{\R^n} \frac{|f(x)-f(y)|^2}{|x-y|^{n+2s}}\,\,dxdy &=& \int_{\R^n} |\hat f(\xi)|^2 
\left( \int_{\R^n} \frac{|e^{iz\cdot \xi}-1|^2}{|z|^{n+2s}}\,dz \right)\,d\xi\\
&=& \int_{\R^n} |\xi|^{2s} |\hat f(\xi)|^2 \left( \int_{\R^n} \frac{|e^{iz\cdot \xi}-1|^2}{|\xi|^{2s}|z|^{n+2s}}\,dz
\right)\,d\xi.
\end{eqnarray*}
Let $U \in SO(n)$ such that $U(\xi)=|\xi| e_1$ where $e_1=(1,0,\ldots,0) \in \R^n$. For $\xi \neq 0$ introduce
the new variable $y$ defined by $z=\frac{1}{|\xi|} U^{-1}(y)$. With this change of variable, the inner integral becomes,
\[
\int_{\R^n} \frac{|e^{iz\cdot \xi}-1|^2}{|\xi|^{2s}|z|^{n+2s}}\,dz = \int_{\R^n} 
\frac{|e^{iy_1}-1|^2}{|y|^{n+2s}}\,dy < \infty.
\]
Note that the latter integral converges and equals a positive constant that is independent of $\xi$.
Hence we conclude that for any $f \in L^2(\R^n,\R)$ one has $\|f\|_s^2 < \infty$ iff
\[
\int_{\R^n}\int_{\R^n} \frac{|f(x)-f(y)|^2}{|x-y|^{n+2s}}\,\,dxdy < \infty.
\]
The statement on the norms is easily verified.
\end{proof}

The following result extends part $(ii)$ of Lemma \ref{lemma_chain_rule}.

\begin{Lemma}\label{uniform_fractional_bound}
Let $\varphi \in \emph{Diff}_+^1(\R^n)$ with $d\varphi$ and $d \varphi^{-1}$ bounded on all of $\R^n$. Then for any $0 < s' < 1$, the right translation by $\varphi$, $f \mapsto R_\varphi(f)=f \circ \varphi$ is a bounded linear operator on $H^{s'}(\R^n,\R)$.
\end{Lemma}

\begin{proof}
In view of statement $(i)$ of Lemma \ref{lemma_chain_rule}, it remains to show that $[R_\varphi f]_{s'} < \infty$. By a change of variables one gets
\begin{eqnarray*}
[f \circ \varphi]_{s'}^2 &=& \int_{\R^n} \int_{\R^n} \frac{| f\big(\varphi(x)\big) - f\big(\varphi(y)\big)|^2}{|x-y|^{n+2s'}}\,\,dxdy\\
&\leq& \frac{1}{M^2} \int_{\R^n} \int_{\R^n} \frac{|f(x)-f(y)|^2}{|\varphi^{-1}(x)-\varphi^{-1}(y)|^{n+2s'}}\,\,dxdy
\end{eqnarray*}
where $M:=\inf_{x \in \R^n}(\det d_x \varphi)$. As $d\varphi$ is bounded on $\R^n$, one has for any $x,y \in \R^n$
\[
 |x-y| = |\varphi\big(\varphi^{-1}(x)\big)-\varphi\big(\varphi^{-1}(y)\big)| \leq L |\varphi^{-1}(x)-\varphi^{-1}(y)|
\]
where $L:=\sup_{x \in \R^n} |d_x\varphi| < \infty$. Hence
\begin{equation}\label{bounded_fractional_map}
 [f \circ \varphi]_{s'} \leq M^{-1} L^{n/2+s'} [f]_{s'} \quad \forall f \in H^{s'}(\R^n,\R). 
\end{equation}
Hence $f \circ \varphi \in H^{s'}(\R^n,\R)$ and it follows that $R_\varphi$ is a bounded linear operator on
$H^{s'}(\R^n,\R)$.
\end{proof}

Next we extend Lemma \ref{lemma_division} to the case where $s$ and $s'$ are real. Using the notation introduced in
Section \ref{Section 2}, one has

\begin{Lemma}\label{lemma_fractional_division}
Let $s,s'$ be real with $s > n/2$ and $0 \leq s' \leq s$. Then for any $\varepsilon > 0$ and $K > 0$ there exists a
constant $C \equiv C(\varepsilon,K;s,s') > 0$ so that for any $f \in H^{s'}(\R^n,\R)$ and $g \in U_\varepsilon^s$ with
$\|g\|_s < K$ one has $f/(1+g) \in H^{s'}(\R^n,\R)$ and 
\begin{equation}\label{eq:inequality}
\|f/(1+g)\|_{s'} \leq C \|f\|_{s'}.
\end{equation}
Moreover, the map
\begin{equation}\label{eq:map}
H^{s'}(\R^n,\R) \times U^s \to H^{s'}(\R^n,\R),\quad (f,g) \mapsto f/(1+g)
\end{equation}
is continuous.
\end{Lemma}

\begin{proof}
In view of Lemma \ref{lemma_division} and Remark \ref{rem_lemma_division},
the claimed statement holds for real $s$ with $s > n/2$ and integers $s'$ satisfying $0 \leq s' \leq s$. 
Arguing by induction we will prove the first statement of the Lemma. 
Let us first show that \eqref{eq:inequality} holds for any $0<s'<1$, $s'\le s$.
Take an arbitrary $g\in U^s_\varepsilon$, $\varepsilon>0$. Then $\forall f\in H^{s'}(\R^n,\R)$,
$f/(1+g)\in L^2(\R^n,\R)$ and 
\begin{equation}\label{eq:L2}
\|f/(1+g)\|\le\frac{1}{\varepsilon}\,\|f\|\,.
\end{equation}
According to Lemma \ref{lem_charakterisierung} it remains to show that $[f/(1+g)]_{s'} < \infty$. Write
\begin{eqnarray*}
\frac{f(x)}{1+g(x)} - \frac{f(y)}{1+g(y)} &=& \big( f(x)-f(y) \big)
\frac{1+g(x)+g(y)}{\big(1+g(x)\big)\big(1+g(y)\big)}\\
&-& \frac{f(x)g(x)-f(y)g(y)}{\big(1+g(x)\big)\big(1+g(y)\big)}
\end{eqnarray*}
and note that by Remark \ref{rem_lem_imbedding}, $f \cdot g \in H^{s'}(\R^n,\R)$ and
\[
\sup_{x,y \in \R^n} \frac{1+|g(x)|+|g(y)|}{\big(1+g(x)\big)\big(1+g(y)\big)} \leq C_1
\]
for some constant $C_1 > 0$. This together with Lemma \ref{lem_imbedding} and 
Remark \ref{rem_lem_imbedding} implies
\begin{eqnarray}
\left[\frac{f}{1+g} \right]_{s'}^2 &\le& 2C_1^2 [f]_{s'}^2 + 2C_1^2 [fg]_{s'}^2\nonumber\\
&\le& 2 C_1^2(\|f\|_{s'}^2+\|g f\|_{s'}^2)\le C_2\|f\|_{s'}^2<\infty\label{eq:fractional_part}
\end{eqnarray}
where $C_2>0$.\footnote{The positive constants $C_1$ and $C_2$ depend on the $s$-norm of $g$.}
Combining \eqref{eq:L2} with \eqref{eq:fractional_part} we see that \eqref{eq:inequality} holds
for any $0<s'<1$, $s'\le s$. This completes the proof of the Lemma when $s<0$. If $s>1$ we 
assume that \eqref{eq:inequality} holds for any $0\le s'\le k$, with $1\le k< s$, $k\in\Z$.
We will show that then \eqref{eq:inequality} holds for $k<s'<k+1$, $s'\le s$.
Take an arbitrary $f\in H^{s'}(\R^n,\R)$. As $H^{s'}(\R^n,\R)\subseteq H^k(\R^n,\R)$ we get from
the proof of Lemma \ref{lemma_division},
\begin{equation}\label{eq:product_rule''}
\dx{i}\left(\frac{f}{1+g}\right) = \frac{\dx{i}f}{1+g}- 
\frac{\frac{\dx{i}(fg)}{1+g}-\frac{g\cdot\dx{i}f}{1+g}}{1+g}\,.
\end{equation}
By Remark \ref{rem_lem_imbedding}, $g\cdot\dx{i}f$ and $\dx{i}(fg)$ are in $H^{s'-1}(\R^n,\R)$.
This together with the induction hypothesis and \eqref{eq:product_rule''} implies that $f/(1+g)\in H^{s'}(\R^n,\R)$.
Inequality \eqref{eq:inequality} follows immediately from the induction hypothesis and \eqref{eq:product_rule''}.

In order to prove that \eqref{eq:map} is continuous we argue as follows.
Take an arbitrary $g\in U^s_\varepsilon$, $\varepsilon>0$. In view of Proposition \ref{prop_sobolev_imbedding},
Remark \ref{rem_prop_sobolev_imbedding}, and \eqref{eq:inequality},
there exists $\kappa>0$ such that for any $\delta g\in B^s_\kappa$\footnote{$B^s_\kappa$ is the
open ball of radius $\kappa$ centered at zero in $H^s(\R^n,\R)$.},
\begin{equation}\label{eq:preparation}
\|\delta g/(1+g)\|_s<1 \quad\mbox{and}\quad\|\delta g\|_{C^0}<\varepsilon/2\,.
\end{equation}
Consider the map, $H^{s'}(\R^n,\R)\times B^s_\kappa\to H^{s'}(\R^n,\R)$,
\begin{equation}\label{eq:delta_map}
(\delta f, \delta g)\mapsto\frac{\delta f}{1+(g+\delta g)}\,.
\end{equation}
In view of \eqref{eq:preparation} and the first statement of the Lemma, the map \eqref{eq:delta_map} is
well-defined. We have
\begin{eqnarray}
\frac{\delta f}{1+g+\delta g}&=&\frac{\delta f}{1+g}\cdot\frac{1}{1+\frac{\delta g}{1+g}}=\frac{\delta f}{1+g}+
\frac{\delta f}{1+g}\cdot\sum_{j=0}^\infty(-1)^j\Big(\frac{\delta g}{1+g}\Big)^j\nonumber\\
&=&\frac{\delta f}{1+g}+\frac{\delta f}{1+g}\cdot{\cal S}(\delta g)\label{eq:good_form}
\end{eqnarray}
where ${\cal S}: B^s_\kappa\to H^s(\R^n,\R)$ is an analytic function. Finally, the continuity of 
\eqref{eq:delta_map} follows from \eqref{eq:good_form}, \eqref{eq:inequality},
Lemma \ref{lem_imbedding} and Remark \ref{rem_lem_imbedding}.
\end{proof}

The following lemma extends Lemma \ref{continuous_composition} to the case where $s$ and $s'$ are real numbers
instead of integers. For any real number $s > n/2+1$ introduce 
\[
\Ds^s(\R^n):=\big\{ \varphi \in \mbox{Diff}_+^1(\R^n) \, \big| \, \varphi - \id \in H^s(\R^n) \big\}.
\]

\begin{Lemma}\label{fractional_continuous_composition}
Let $s,s'$ be real numbers with $s>n/2+1$ and $0 \leq s' \leq s$. Then the composition 
\[
\mu^{s'}:H^{s'}(\R^n,\R) \times \Ds^s(\R^n) \to H^{s'}(\R^n,\R),\quad (f,\varphi) \mapsto f \circ \varphi
\]
is continuous.
\end{Lemma}

\begin{proof}
We argue by induction on intervals of values of $s'$, $k \leq s' < k+1$.
Let us begin with the case where $0 \leq s' < 1$. 
Note that the case where $s$ is real and $s'$ integer is already dealt with in Lemma \ref{continuous_composition}
-- see Remark \ref{rem_continuous_composition}. In particular,
\[
L^2(\R^n,\R) \times \Ds^s(\R^n) \to L^2(\R^n,\R), \, (f,\varphi) \mapsto f \circ \varphi
\]
is continuous. Next assume that $0 < s' < 1$. Then for any $f, f_\bullet \in H^{s'}(\R^n,\R)$ and 
$\varphi, \varphi_\bullet \in \Ds^s(\R^n)$, the expression 
$[f \circ \varphi - f_\bullet \circ \varphi_\bullet]_{s'}^2$ is bounded by
\begin{eqnarray}
\int_{\R^n} \int_{\R^n} 
\frac{|f\big(\varphi(x)\big)-f_\bullet\big(\varphi(x)\big)-f\big(\varphi(y)\big)+
f_\bullet\big(\varphi(y)\big)|^2}{|x-y|^{n+2s'}}\,\,dxdy &&\nonumber\\
+ \int_{\R^n} \int_{\R^n} 
\frac{|f_\bullet\big(\varphi(x)\big)-f_\bullet\big(\varphi_\bullet(x)\big)-
f_\bullet\big(\varphi(y)\big)+f_\bullet\big(\varphi_\bullet(y)\big)|^2}{|x-y|^{n+2s'}}\,\,dxdy.&&
\label{eqn_fractional_estimate}
\end{eqnarray}
By (\ref{bounded_fractional_map}), the first integral in (\ref{eqn_fractional_estimate}) can be estimated  by
$C [f-f_\bullet]_{s'}^2$ where $C > 0$ can be chosen locally uniformly for $\varphi$ in $\Ds^s(\R^n)$. 
The second integral in (\ref{eqn_fractional_estimate}) we write as
\[
\int_{\R^n} \int_{\R^n} \frac{\Big|\Big(f_\bullet\big(\varphi(x)\big) - f_\bullet\big(\varphi(y)\big)\Big) - 
\Big(f_\bullet\big(\varphi_\bullet(x)\big) - f_\bullet\big(\varphi_\bullet(y)\big)\Big)\Big|^2}{|x-y|^{n+2s'}}\,\,dxdy.
\]
By Lemma \ref{lem_charakterisierung},
\[
F(x,y):=\frac{f_\bullet(x)-f_\bullet(y)}{|x-y|^{n/2+s'}}
\]
is in $L^2(\R^n \times \R^n,\R)$. Hence again by Remark \ref{rem_continuous_composition},
\[
F\big(\varphi(x),\varphi(y)\big) \to F\big(\varphi_\bullet(x),\varphi_\bullet(y)\big) 
\quad \mbox{in} \quad L^2(\R^n \times \R^n,\R).
\]
In view of the estimate
\[
\Big|\frac{\varphi(y)-\varphi(x)}{|y-x|}\Big|\le\Big|\int_0^1 (d_{x+(y-x)t}\varphi)\Big(\frac{y-x}{|y-x|}\Big)\Big|
\le \|d\varphi\|_{C^0}
\]
and the continuity of $\Ds^s(\R^n) \to C_0^1(\R^n), \varphi \mapsto \varphi - \id$ 
(Remark \ref{rem_prop_sobolev_imbedding}) one sees that
\[
\frac{\varphi(x)-\varphi(y)}{|x-y|} \to \frac{\varphi_\bullet(x)-\varphi_\bullet(y)}{|x-y|} 
\quad \mbox{in} \quad L^\infty(\R^n \times \R^n,\R).
\]
Writing
\[
\frac{f_\bullet\big(\varphi(x)\big)-f_\bullet\big(\varphi(y)\big)}{|x-y|^{n/2+s'}} = 
F\big(\varphi(x),\varphi(y)\big) \frac{|\varphi(x)-\varphi(y)|^{n/2+s'}}{|x-y|^{n/2+s'}}
\]
it then follows that as $\varphi \to \varphi_\bullet$ in $\Ds^s(\R^n)$
\[
\frac{f_\bullet\big(\varphi(x)\big)-f_\bullet\big(\varphi(y)\big)}{|x-y|^{n/2+s'}} \to 
\frac{f_\bullet\big(\varphi_\bullet(x)\big)-f_\bullet\big(\varphi_\bullet(y)\big)}{|x-y|^{n/2+s'}} 
\quad \mbox{in} \quad L^2(\R^n \times \R^n,\R).
\]
Now let us prove the induction step. Assume that the continuity of the composition $\mu^{s'}$ has been established for
any $s'$ with $0 \leq s' \le k$ where $k \in \mathbb Z_{\geq 1}$ satisfies $k < s$. Consider $s' \in \R$ with
$k \leq s' \leq s$ (if $s < k+1$) resp. $k \leq s' < k+1$ (if $s \geq k+1$). By Lemma \ref{lemma_chain_rule}$(ii)$,
\[
d(f \circ \varphi) = df \circ \varphi \cdot d\varphi.
\]
In view of Lemma \ref{lem_inductive_norm}, $df \in H^{s'-1}(\R^n,\R^n)$, hence by the induction hypothesis, 
if $f \to f_\bullet$ in $H^{s'}(\R^n,\R)$ and $\varphi \to \varphi_\bullet$ in $\Ds^s(\R^n)$, one has
\[
 df \circ \varphi \to df_\bullet \circ \varphi_\bullet \quad \mbox{in} \quad H^{s'-1}(\R^n,\R^n)\,.
\]
As $d\varphi \in H^{s'-1}(\R^n,\R^{n \times n})$ and $s-1>n/2$ one then concludes from Remark \ref{rem_lem_imbedding},
\[
df \circ \varphi \cdot d\varphi \to df_\bullet \circ \varphi_\bullet \cdot d\varphi_\bullet \quad \mbox{in} 
\quad H^{s'-1}(\R^n,\R^n)
\]
and Lemma \ref{lem_inductive_norm} implies that $f \circ \varphi \to f_\bullet \circ \varphi_\bullet$ in
$H^{s'}(\R^n,\R)$. This establishes the continuity of $\mu^{s'}$ and proves the induction step. 
\end{proof}

Next we extend Lemma \ref{continuous_inverse} to the case where $s$ is fractional.

\begin{Lemma}\label{fractional_continuous_inverse}
Let $s$ be real with $s > n/2+1$. Then for any $\varphi \in \Ds^s(\R^n)$, its inverse $\varphi^{-1}$ is again in
 $\Ds^s(\R^n)$ and
\[
 {\tt inv}:\Ds^s(\R^n) \to \Ds^s(\R^n),\quad \varphi \mapsto \varphi^{-1}
\]
is continuous.
\end{Lemma}

\begin{proof}
Let $\varphi \in \Ds^s(\R^n)$. Then $\varphi$ is in $\mbox{Diff}_+^1(\R^n)$ and so is its inverse $\varphi^{-1}$.
We claim that $\varphi^{-1}$ is in $\Ds^s(\R^n)$. It follows from the proof of Lemma \ref{continuous_inverse},
together with Remark \ref{rem_lem_imbedding} and Lemma \ref{lemma_fractional_division} that
for any $\alpha \in \mathbb Z_{\geq 0}^n$ with $0 \leq |\alpha| \leq s$,
$\partial^\alpha(\varphi^{-1}-\id)$ is of the form
\[
 \partial^\alpha(\varphi^{-1}-\id) = F^{(\alpha)} \circ \varphi^{-1}
\]
where $F^{(\alpha)} \in H^{s-|\alpha|}(\R^n)$. In addition, by Remark \ref{rem_lem_imbedding} and 
Lemma \ref{lemma_fractional_division}, the map $\Ds^s(\R^n) \to H^{s-|\alpha|}(\R^n)$, 
$\varphi \mapsto F^{(\alpha)}$ is continuous. It then follows that
\[
 \int_{\R^n} |\partial^\alpha (\varphi^{-1}-\id)|^2 dx = \int_{\R^n} |F^{(\alpha)}|^2 \det(d_y \varphi) dy < \infty.
\]
Moreover, in case $|\alpha|=\lfloor s \rfloor$ and $s \notin \N$ one has for $0 < \lambda:=s-\lfloor s \rfloor < 1$,
\begin{multline*}
\left[ F^{(\alpha)} \circ \varphi^{-1} \right]_\lambda^2 = 
\int_{\R^n} \int_{\R^n} \frac{|F^{(\alpha)}\big( \varphi^{-1}(x)\big) - 
F^{(\alpha)}\big(\varphi^{-1}(y)\big)|^2}{|x-y|^{n+2\lambda}}\,\,dxdy \\
\leq M^2 \int_{\R^n} \int_{\R^n} \frac{|F^{(\alpha)}(x')-F^{(\alpha)}(y')|^2}{|x'-y'|^{n+2\lambda}} 
\frac{|x'-y'|^{n+2\lambda}}{|\varphi(x')-\varphi(y')|^{n+2\lambda}}\,\,dx'dy'
\end{multline*}
where $M:=\sup_{x \in \R^n}(\det d_x \varphi)$. As $|\varphi^{-1}(x)-\varphi^{-1}(y)| \leq L |x-y|$ for any 
$x,y \in \R^n$ with
\[
L:=\sup_{z \in \R^n} |d_z \varphi^{-1}| < \infty 
\]
it follows that
\[
\frac{|x'-y'|}{|\varphi(x')-\varphi(y')|} \leq L \quad \forall x',y' \in \R^n, \, x' \neq y'.
\]
Altogether one has, for any $\alpha \in \mathbb Z_{\geq 0}^n$ with $|\alpha|=\lfloor s \rfloor$,
\begin{equation}\label{fractional_L2_continuity}
\left[ F^{(\alpha)} \circ \varphi^{-1} \right]_\lambda^2 \leq M^2 L^{n+2\lambda} \left[ F^{(\alpha)} \right]_\lambda^2.
\end{equation}
By Lemma \ref{lem_charakterisierung} it then follows that $\varphi^{-1} - \id \in H^s(\R^n)$. 
In addition, the estimates obtained show that the map $\Ds^s(\R^n) \to H^s(\R^n)$, 
$\varphi \mapsto \varphi^{-1} - \id$ is locally bounded. It remains to show that this map is continuous. 
By the proof of Lemma \ref{continuous_inverse}, the map $\Ds^s(\R^n) \to L^2(\R^n)$, 
$\varphi \mapsto \varphi^{-1} -\id$ is continuous. Using that $F^{(\alpha)}:\Ds^s(\R^n) \to H^{s-|\alpha|}(\R^n)$ is
continuous for any $\alpha \in \mathbb Z_{\geq 0}^n$ with $|\alpha| \leq s$ one shows in a similar way as in 
Lemma \ref{continuous_inverse} that $\Ds^s(\R^n) \to L^2(\R^n)$, 
$\varphi \mapsto \partial^\alpha(\varphi^{-1} - \id) = F^{(\alpha)} \circ \varphi^{-1}$ is continuous. 
Now consider the case where $\alpha \in \mathbb Z_{\geq 0}^n$ satisfies $|\alpha|=\lfloor s \rfloor$ and 
$\lambda:=s - \lfloor s \rfloor > 0$. For any $\varphi_\bullet \in \Ds^s(\R^n)$ consider
\begin{multline*}
\left[ \partial^\alpha(\varphi^{-1} - \varphi_\bullet^{-1})\right]_\lambda = \left[ F^{(\alpha)} \circ 
\varphi^{-1} - F_\bullet^{(\alpha)} \circ \varphi_\bullet^{-1} \right]_\lambda \\
\leq  \left[ F^{(\alpha)} \circ \varphi^{-1} - F_\bullet^{(\alpha)} \circ \varphi^{-1} \right]_\lambda + 
\left[ F_\bullet^{(\alpha)} \circ \varphi^{-1} - F_\bullet^{(\alpha)} \circ \varphi_\bullet^{-1} \right]_\lambda\,.
\end{multline*}
It follows from (\ref{fractional_L2_continuity}) that
\[
\left[ F^{(\alpha)} \circ \varphi^{-1} - F_\bullet^{(\alpha)} \circ \varphi^{-1} \right]_\lambda 
\leq M L^{\lambda+n/2} \left[ F^{(\alpha)}-F_\bullet^{(\alpha)} \right]_\lambda.
\]
As $F^{(\alpha)}:\Ds^s(\R^n) \to H^\lambda(\R^n)$ is continuous, 
$\left[ F^{(\alpha)} \circ \varphi^{-1} - F_\bullet^{(\alpha)} \circ \varphi^{-1} \right]_\lambda \to 0$ as 
$\varphi \to \varphi_\bullet$ in $\D^s(\R^n)$. Finally consider the term 
$\left[ F_\bullet^{(\alpha)} \circ \varphi^{-1} - F_\bullet^{(\alpha)} \circ \varphi_\bullet^{-1} \right]_\lambda$.
Arguing as in the proof of Lemma \ref{continuous_inverse}, we approximate $\varphi_\bullet$ by
$\tilde \varphi \in \Ds^s(\R^n)$ with $\tilde \varphi - \id \in C_c^\infty(\R^n,\R^n)$. Then
\begin{multline*}
  \left[ F_\bullet^{(\alpha)} \circ \varphi^{-1} - F_\bullet^{(\alpha)} \circ \varphi_\bullet^{-1} 
\right]_\lambda \leq \left[F_\bullet^{(\alpha)} \circ \varphi^{-1} - \tilde F^{(\alpha)} \circ \varphi^{-1}
\right]_\lambda +\\
  + \left[ \tilde F^{(\alpha)} \circ \varphi^{-1} - \tilde F^{(\alpha)} \circ \varphi_\bullet^{-1}\right]_
\lambda + \left[\tilde F^{(\alpha)} \circ \varphi_\bullet^{-1} - F_\bullet^{(\alpha)} \circ \varphi_\bullet^{-1}
 \right]_\lambda
\end{multline*}
where $\tilde F^{(\alpha)} = \left. F^{(\alpha)} \right|_{\tilde \varphi}$. 
For $\varphi$ near $\varphi_\bullet$ one has as above,
\[
\left[ F_\bullet^{(\alpha)} \circ \varphi^{-1} - \tilde F^{(\alpha)} \circ \varphi^{-1}\right]_\lambda 
\leq M L^{\lambda + n/2} \left[ F_\bullet^{(\alpha)}- \tilde F^{(\alpha)} \right]_\lambda.
\]
Similarly, the expression $\left[\tilde F^{(\alpha)} \circ \varphi_\bullet^{-1} - F_\bullet^{(\alpha)} 
\circ \varphi_\bullet^{-1}\right]_\lambda$ can be bounded in terms of 
$\left[\tilde F^{(\alpha)}-F_\bullet^{(\alpha)}\right]_\lambda$. 
To estimate the remaining term it suffices to show that, as $\varphi \to \varphi_\bullet$ in $\D^s(\R^n)$, 
\[
 \|\tilde F^{(\alpha)} \circ \varphi^{-1}- \tilde F^{(\alpha)} \circ \varphi_\bullet^{-1} \|_{1} \to 0.
\]
First we show that 
$\|\tilde F^{(\alpha)} \circ \varphi^{-1}- \tilde F^{(\alpha)} \circ \varphi_\bullet^{-1} \| \to 0$
as $\varphi \to \varphi_\bullet$ in $\D^s(\R^n)$. Indeed, arguing as in the proof of
Lemma \ref{continuous_inverse}, we note that $\tilde F^{(\alpha)}$ is Lipschitz continuous, i.e.
\[
 |\tilde F^{(\alpha)}(x)-\tilde F^{(\alpha)}(y)| \leq L_1 |x-y| \quad \forall x,y \in \R^n
\]
for some constant $L_1 > 0$. Then
\[
 \int_{\R^n} |\tilde F^{(\alpha)} \circ \varphi^{-1} - \tilde F^{(\alpha)} \circ \varphi_\bullet^{-1}|^2 dx \leq L_1^2
 \int_{\R^n} |\varphi^{-1} - \varphi_\bullet^{-1}|^2 dx
\]
and therefore
\[
 \|\tilde F^{(\alpha)} \circ \varphi^{-1} - \tilde F^{(\alpha)} \circ \varphi_\bullet^{-1}\| \to 0 \quad \mbox{as}
 \quad \varphi \to \varphi_\bullet\quad\mbox{in}\quad \D^s(\R^n)\,.
\]
It remains to show that
\[
 \|d(\tilde F^{(\alpha)} \circ \varphi^{-1})- d(\tilde F^{(\alpha)} \circ \varphi_\bullet^{-1}) \| \to 0 \quad 
\mbox{as} \quad \varphi \to \varphi_\bullet\quad\mbox{in}\quad \D^s(\R^n)\,.
\]
By the chain rule we have
\[
d(F^{(\alpha)} \circ \varphi^{-1}) = dF^{(\alpha)} \circ \varphi^{-1} \cdot d\varphi^{-1}. 
\]
Hence 
\begin{eqnarray*}
\|d(\tilde F^{(\alpha)} \circ \varphi^{-1})- d(\tilde F^{(\alpha)} \circ \varphi_\bullet^{-1})\| &
\leq& \|d \tilde F^{(\alpha)} \circ \varphi^{-1} - d \tilde F^{(\alpha)} \varphi_\bullet^{-1}\| \,\, 
\|d\varphi^{-1}\|_{L^\infty} \\
&+& \|d\tilde F^{(\alpha)} \circ \varphi_\bullet^{-1}\| \,\, \|d\varphi^{-1} - d\varphi_\bullet^{-1}\|_{L^\infty}.
\end{eqnarray*}
Arguing as above one has, as $\varphi \to \varphi_\bullet$ in $\D^s(\R^n)$,
\[
 \|d \tilde F^{(\alpha)} \circ \varphi^{-1} - d\tilde F^{(\alpha)} \circ \varphi_\bullet^{-1}\| \to 0
\]
and, by Remark \ref{rem_prop_sobolev_imbedding} and inequality \eqref{eq:infty_close},
\[
\|d\varphi^{-1} - d\varphi_\bullet^{-1}\|_{L^\infty} \to 0.
\]
Altogether we thus have shown that
\[
\|\tilde F^{(\alpha)} \circ \varphi^{-1} - \tilde F^{(\alpha)} \circ \varphi_\bullet^{-1}\|_1 \to 0 \quad \mbox{as}
\quad \varphi \to \varphi_\bullet\quad\mbox{in}\quad \D^s(\R^n)\,.
\]
This finishes the proof of the claimed statement that $\varphi \to \varphi^{-1}$ is continuous on $\Ds^s(\R^n)$.
\end{proof}

\begin{Prop}\label{fractional_tg}
For any real number $s > n/2+1$, $(\Ds^s,\circ)$ is a topological group.
\end{Prop}

\begin{proof}
The claimed statement follows from Lemma \ref{fractional_continuous_composition} and
Lemma \ref{fractional_continuous_inverse}.
\end{proof}

Now we have established all ingredients to show the following extension of Theorem \ref{thm1}.

\begin{Th}\label{fractional_thm1}
For any $r \in \mathbb Z_{\geq 0}$ and any real number $s$ with $s > n/2+1$
\[
 \mu: H^{s+r}(\R^n,\R^d) \times \Ds^s(\R^n) \to H^s(\R^n,\R^d), \quad (u,\varphi) \mapsto u \circ \varphi
\]
and
\[
 {\tt inv}:\Ds^{s+r}(\R^n) \to \Ds^s(\R^n), \quad \varphi \mapsto \varphi^{-1}
\]
are $C^r$-maps.
\end{Th}

\begin{proof}
Using the results established above in this appendix, the proof of Theorem \ref{thm1}, given in Subsection 
\ref{subsection_proof_of_thm1}, extends in a straightforward way to the case where $s$ is real.
\end{proof}


Finally we want to extend the results of Subsection \ref{sobolev_spaces_on_open_sets}, Section \ref{Section 3},
and Section \ref{sec:diff_structure} to Sobolev spaces of fractional exponents.
\begin{Def}
Let $U$ be a bounded open set in $\R^n$ with Lipschitz boundary.
Then $f\in H^s(U,\R)$ if there exists ${\tilde f}\in H^s(\R^n,\R)$ such that ${\tilde f}\big|_U=f$.
\end{Def} 
\noindent Note that for our purposes it is enough to consider only the case when the boundary of $U$ is
a finite (possibly empty) union of transversally intersecting $C^\infty$-embedded hypersurfaces in $\R^n$
(cf. Definition \ref{def_fine_cover}).

As in the case where $s$ is an integer, the spaces $H^s(U,\R)$ and $H^s(\R^n,\R)$ are closely related.
In view of \cite{Rych}, item (ii) of Proposition \ref{prop_extension} holds. Note that 
$H^s(\R^n,\R)=F^s_{22}(\R^n,\R)$ where $F^s_{22}$ is the corresponding Triebel-Lizorkin space.
This allows us to define maps of class $H^s$ between manifolds and extend the results in
Subsection \ref{subsec3:prelim} to Sobolev spaces of fractional exponents.

The corresponding space of maps is denoted by $H^s(M,N)$. Similarly, one extends the definition of 
$\Ds^s(M)$ for $s$ fractional. Following the line of arguments of Section \ref{Section 3} and Section 
\ref{sec:diff_structure} one then concludes that Theorem \ref{thm2} can be extended as follows

\begin{Th}\label{fractional_thm2}
Let $M$ be a closed oriented manifold of dimension $n$, $N$ a $C^\infty$-manifold and $s$ any real number satisfying
$s > n/2+1$. Then for any $r \in \mathbb Z_{\geq 0}$,
\begin{itemize}
\item[(i)]  $\mu:H^{s+r}(M,N) \times \Ds^s(M) \to H^s(M,N),\quad (f,\varphi) \mapsto f \circ \varphi$
\item[(ii)] ${\tt inv}:\Ds^{s+r}(M) \to \Ds^s(M),\quad \varphi \mapsto \varphi^{-1}$
\end{itemize}
are both $C^r$-maps.
\end{Th}

\begin{Rem}
Note that our construction can be used to prove analogous results for maps between manifolds
in Besov or Triebel-Lizorkin spaces.
\end{Rem}

\bibliographystyle{plain}

\author{
 \begin{tabular}{llll}
         Institut f\"ur Mathematik        &&Department of Mathematics \\
         Universit\"at Z\"urich           &&Northeastern University \\
         Winterthurerstrasse 190          &&Boston, MA 02115 \\
         CH-8057 Z\"urich     &&USA \\
Schwitzerland & \\
        {\it email: } {hasan.inci@math.uzh.ch}
        &&{\it email: } {p.topalov@neu.edu}\\
{\it email: } {tk@math.uzh.ch}
 \end{tabular}

\end{document}